\documentclass[a4paper,11pt]{article}

\usepackage{BeaCed_article}

\title{Fock's dimer model on the Aztec diamond}
\author{%
  C\'edric Boutillier%
  \thanks{%
    {\small Sorbonne Université, CNRS,
      Laboratoire de Probabilités Statistique et Modélisation, LPSM, UMR 8001,
      F-75005 Paris, France; Institut Universitaire de France.
      \texttt{cedric.boutillier@sorbonne-universite.fr}
}},
  B\'eatrice de Tili\`ere%
  \thanks{{\small %
PSL University-Dauphine, CNRS, UMR 7534, CEREMADE, 75016 Paris, France.}{\small\texttt{ detiliere@ceremade.dauphine.fr}}
}
}

\begin{document}
\maketitle

\begin{abstract}
We consider the dimer model on the Aztec diamond with Fock's weights, which is
gauge equivalent to the model with any choice of positive weight function. We
prove an explicit, compact formula for the inverse Kasteleyn matrix, thus
extending numerous results in the case of periodic graphs. We also show an
explicit product formula for the partition function; as a specific instance of
the genus 0 case, we recover Stanley's
formula~\cite{Propp_talk,BYYangPhD}. We then use our explicit formula for the
inverse Kasteleyn matrix to recover, in a simple way, limit shape results; we
also obtain new ones. In doing so, we extend the correspondence between the
limit shape and the amoeba  of the corresponding spectral curve
of~\cite{BerggrenBorodin} to the case of non-generic weights.  
\end{abstract}

\section{Introduction}\label{sec:intro}

We consider the dimer model on the Aztec diamond with Fock's
weights~\cite{Fock,BCdT:genusg}. Our first main result is an explicit expression
for the inverse Kasteleyn matrix, which only uses theta functions, prime forms
and local functions in the kernel of the associated Kasteleyn operator. We also
prove that any dimer model on the Aztec diamond is gauge equivalent to a dimer
model with Fock's weights, thus showing that we actually treat the dimer model
on the Aztec diamond in full generality. Next, building on an idea of
Propp~\cite{Propp_talk}, we prove an induction formula for the partition
function, showing that it admits a product form; as a specific case, we recover
Stanley's celebrated formula~\cite{BYYangPhD}.
Finally, we use our explicit formula for the inverse Kasteleyn matrix to recover
and extend,
in a simple way, results on limit shapes in
genus 0, 1, and higher,
generalizing the geometric correspondence
  between the limit shape of the Aztec diamond for periodic weights and the
  amoeba of the corresponding spectral curve proven by Berggren and
Borodin~\cite{BerggrenBorodin} under some technical assumptions, which are now
lifted.

Let us now be more specific, and give some historical background. This paper
aims at, in some sense, closing a long history of formulas for the inverse
Kasteleyn matrix of the dimer model on the Aztec diamond. This was initiated
in~\cite{Helfgott} in the case of uniform weights. Then,
in~\cite{ChhitaJohanssonYoung}, Chhita, Johansson and Young consider the case
where vertical and horizontal dominos have different weights, also known as
1-periodic weights; their proof consists in connecting dimer configurations to
particle systems together with a direct check of their guessed formula. Next
in~\cite{ChhitaYoung}, Chhita and Young treat the case of 1-periodic weights
with a possible volume penalization, and 2-periodic weights; to prove their
result, the authors use generating functions. The 1-periodic weights with volume
penalization is then generalized in~\cite{ryg} using the connection to Schur
processes. In~\cite{ChhitaJohansson} Chhita and Johannsson again consider
2-periodic weights; they prove a simplification of the formula
of~\cite{ChhitaYoung}, compute asymptotics of the inverse and limit shapes, in
particular they discuss the arctic curves separating the various phases given by
an algebraic curve of degree 8.
In~\cite{DuitsKuijlaars}, Duits and Kuijlaars propose a different approach for the 2-periodic weights of~\cite{ChhitaYoung} using non intersecting lattice paths and matrix valued orthogonal polynomials; allowing them to compute finer asymptotics using an associated Riemann-Hilbert problem. A new approach, involving block Toepliz matrices and the Wiener-Hopf factorization is proposed in~\cite{BerggrenDuits}; it is then extended to $2\times k$-periodic weights in~\cite{Berggren}; in particular, the authors give a rigorous proof of the arctic curves derived in~\cite{diFrancescoSotoGarrido}. In~\cite{BorodinDuits}, Borodin and Duits introduce biased $2\times 2$ periodic weights; two specific cases are the 1-periodic weights of~\cite{ChhitaJohanssonYoung} and the 2-periodic ones of~\cite{ChhitaYoung}; the method used is that of block Toepliz matrices and Wiener-Hopf factorization. This approach culminates in the paper~\cite{BerggrenBorodin}, where the authors consider generic $k\times l$-periodic weights, where generic means that the underlying spectral curve has maximal genus. Note that weights of the above $k\times l$ and $2\times k$ models can be specified so as to recover the 2-periodic weights of~\cite{ChhitaYoung}, but then the weights are not generic as assumed in the latter papers, so that the setting is actually different. 

Our first main result, Theorem~\ref{thm:Kinv}, encompasses all of the above cases. Here is the setting, see Sections~\ref{sec:Aztec_dimer}, \ref{sec:Fock} for more details. Consider an M-curve $\Sigma$ of genus $g$, with a distinguished real component $A_0$, and a real element $t$ of the Jacobian variety $\Jac(\Sigma)$. Denote by $\theta$ the associated Riemann theta function and by $E$ the prime form. Consider an Aztec diamond $\Az_n$ of size $n$, with its set of oriented train-tracks $\vec{\T}_n$ naturally split into four, and let $\alphab=(\alpha_j)_{j=1}^n$, $\betab=(\beta_j)_{j=1}^n$, $\gammab=(\gamma_j)_{j=1}^n$, $\deltab=(\delta_j)_{j=1}^n$ be the associated angles on $A_0$ satisfying the cyclic order conditions $\alphab<\gammab<\betab<\deltab$. Suppose that edges of the Aztec diamond are assigned \emph{Fock's weights}~\cite{Fock,BCdT:genusg}, meaning that, for every edge $\ws\bs$ with train-track angles $\alpha,\beta$, the corresponding coefficient of the Kasteleyn matrix $\Ks$ is given by:
\begin{equation*}
\Ks_{\ws,\bs}=\frac{E(\alpha,\beta)}{\theta(t+\mapd(\fs))\theta(t+\mapd(\fs'))},
\end{equation*}
where $\fs,\fs'$ are the dual faces adjacent to $\ws\bs$, and $\mapd$ is the discrete Abel map.
In Proposition~\ref{prop:reconstruction}, we prove that all dimer models on the Aztec diamond can be re-parameterized using Fock's weights. Then, Theorem~\ref{thm:Kinv} can loosely be stated as follows.

\begin{thm}\label{thm:Kinv_intro}
For every pair $(\bs,\ws)$ of black and white vertices of $\Az_n$, the coefficient $(\bs,\ws)$ of the inverse Kasteleyn matrix is explicitly given by 
\begin{multline}\label{eq:Kinv_intro}
\Ks^{-1}_{\bs,\ws}=\frac{1}{(2\pi i)^2}\frac{1}{\theta(p)}\int_{\C_2}\int_{\C_1}\frac{\theta(p+(v-u))}{E(u,v)}g_{\bs,0}(u)g_{0,\ws}(v)\prod_{j=1}^n \frac{E(\beta_j,u)}{E(\delta_j,u)}\frac{E(\delta_j,v)}{E(\beta_j,v)}+\\
-\II_{\{\bs \text{ right of }\ws\}}\frac{1}{2\pi i}\int_{\C_2}g_{\bs,\ws}(v),
\end{multline}
where $\C_1$, $\C_2$ are closed contours on $\Sigma$ used to integrate over $u$ and $v$, defined in Section~\ref{sec:K_inv}; $p=\sum\limits_{j=1}^n(\delta_j-\beta_j)-t-\mapd(0)$, and $g$ is the form in the kernel of the Kasteleyn matrix introduced in~\cite{BCdT:genusg}, see also Equation~\eqref{eq:form_g}. 
\end{thm}
Note that to recover known results on the Aztec diamond requires to identify the associated weights in Fock's form. We do this explicitly in Section~\ref{sec:example_0_1} in the genus 0 case, which includes Stanley's weights~\cite{Propp_talk,BYYangPhD}, and in the genus 1 case, which includes the biased $2\times 2$ periodic case of~\cite{BorodinDuits}.  

Our next result, Theorem~\ref{thm:partition_function}, proves that the partition function admits a product form. Denote by $Z_n(\alphab,\betab,\gammab,\deltab;d)$ the partition function of the Aztec diamond $\Az_n$ of size $n$, with train-track angle parameters $\alphab,\betab,\gammab,\deltab$, and value $d$ for the Abel map at the vertex with coordinates $(0,0)$. Then, Theorem~\ref{thm:partition_function} can be stated as follows. 
\begin{thm}\label{thm:partition_function_intro}
For every $n\geq 1$, the partition function of the Aztec diamond $\Az_n$ with Fock's weights satisfies the following recurrence:
{\small 
\begin{align*}
Z_n(\alphab,\betab,\gammab,\deltab;d)&\cdot 
\prod_{f\in\odd_n}\frac{\theta(t+\mapd(\fs))}{\theta(t+\mapd(\fs)+\alpha+\beta-\gamma-\delta)}\prod_{\fs\in\bry_n}\theta(t+\mapd(\fs))=\\
=Z_{n-1}&((\alpha_j)_{j=1}^{n-1},(\beta_j)_{j=2}^n,(\gamma_j)_{j=1}^{n-1},(\delta_j)_{j=2}^n;d+\beta_1-\delta_1)\cdot
\Bigl[\prod_{j=1}^{n}\ |E(\alpha_j,\beta_j)E(\gamma_j,\delta_j)|\Bigr],
\end{align*}}
with the convention that $Z_0=1$.  
\end{thm}
As specific instances of the genus 0 case, we re-derive that the number of domino tilings of the Aztec diamond is $2^{\frac{n(n+1)}{2}}$, see Remark~\ref{rem:partition_function_genus0} after Corollary~\ref{cor:genus_0}, and Stanley's formula~\cite{Propp_talk,BYYangPhD}, see Corollary~\ref{cor:Stanley}.

We then explain in Lemma~\ref{lem:inverse_Ktilde} that the formula for $\Ks^{-1}$
for the finite Aztec diamond can be extended in a natural way to an operator
$\Kinv$
on an infinite minimal graph containing the Aztec diamond as a subgraph,
introduced in~\cite{Speyer} and this formula is an inverse for the Fock
Kasteleyn operator on this infinite graph. This operator $\Kinv$ does not belong to
the family of inverses defined in~\cite{BCdT:genusg}. It nevertheless carries a
probabilistic meaning and allows to define a Gibbs measure on dimer
configurations of this infinite graph, as stated in
Proposition~\ref{prop:gibbs}. Here is a less formal version:
\begin{prop}
  The determinantal process on edges given by $\Kinv$ defines a Gibbs measure on the
  dimer configurations of the infinite minimal graph, which is supported on
  configurations which are frozen outside the Aztec region and whose marginals on
  the edges of the Aztec region coincide with the Boltzmann measure given by
  Fock's weights.
\end{prop}
The slope in the frozen region is not constant. The operator $\Kinv$ is the first
inverse on an infinite minimal graph for which the corresponding probability
measure exhibits different phases (solid/liquid/gas) depending on the location
on the same infinite graph.

We end this paper with a section with application to the computation of limit
shapes for the Aztec diamond with weights coming from surfaces of genus 0, 1 or
higher. We give in particular a short derivation of the arctic ellipse theorem
for 1-periodic weights~\cite{JPS98,Johansson}. We also explain how to extend
results by Berggren and Borodin from~\cite{BerggrenBorodin} obtained under some
technical assumption to a more general context, see also the forthecoming paper~\cite{BobenkoBobenkoSuris}. 

\subsection*{Outline of the paper}

\begin{itemize}
  \item In Section~\ref{sec:Def}, we recall the definition of the Aztec diamond
    graph, and give the necessary geometric definitions to define Fock's weights for the 
    dimer model. We then discuss some known examples which fit this framework in
    genus 0 and 1.
    Proposition~\ref{prop:reconstruction} establishes that Fock's weights are gauge
    equivalent to any choice of positive weights on the edges.
  \item In Section~\ref{sec:inverse_Kasteleyn}, we state and give a proof
    of~Theorem~\ref{thm:Kinv} giving
    a compact, explicit formula for the inverse Kasteleyn matrix.
  \item Section~\ref{sec:partition} is devoted to results about the partition
    function of the dimer model. It contains the statement and the proof of
    Theorem~\ref{thm:partition_function} showing a product formula for the
    partition function. Application to the genus 0 case and Stanley's
    formula~\cite{Propp_talk,BYYangPhD} are given.
  \item In Section~\ref{sec:aztec_minimal}, we discuss the extension of
    $\Ks^{-1}$ to an infinite minimal graph. Lemma~\ref{lem:inverse_Ktilde}
    gives the structure of this inverse, and Proposition~\ref{prop:gibbs}
    describes the corresponding Gibbs measure.
  \item Section~\ref{sec:limit_shapes} is a discussion of applications of the formula for
    $\Ks^{-1}$ to derive limit shapes and limit of local statistics in the
    thermodynamical limit, recovering and extending several well-known results
    on the topic.
\end{itemize}

\subsection*{Acknowledgments}
We thank Tomas Berggren for interesting discussions about the dimer model on the Aztec
diamond and for explaining some technical details of~\cite{BerggrenBorodin}.
The authors are partially supported by the \emph{DIMERS}
project~ANR-18-CE40-0033 funded by the French National Research Agency.
Part of this research was performed while the
first author was visiting the Institute for Pure and Applied Mathematics (IPAM),
which is supported by the National Science Foundation (Grant No. DMS-1925919).

\section{Definitions and first result} \label{sec:Def}

This section is devoted to introducing and motivating the framework of this
paper. In Section~\ref{sec:Aztec_dimer}, we define the dimer model on the Aztec
diamond, and in Section~\ref{sec:Fock} we introduce Fock's
weights~\cite{Fock,BCdT:genusg}. Then, in Section~\ref{sec:example_0_1} we
explicitly treat the genus 0, resp. genus 1, case and prove how to recover
Stanley's weights~\cite{Propp_talk,BYYangPhD}, resp. the biased $2\times 2$ periodic weights of Borodin-Duits~\cite{BorodinDuits}. Finally, in Section~\ref{sec:space_parameters}, we prove that all dimer models on the Aztec diamond can be parameterized using Fock's weights, implying that we actually study the most general setting of the dimer model on the Aztec diamond.

\subsection{Dimer model on the Aztec diamond}\label{sec:Aztec_dimer}

\paragraph{Aztec diamond.} The \emph{Aztec diamond of size $n$}, denoted by $\Az_n=(\Vs_n,\Es_n)$, is a subgraph of $\ZZ^2$ rotated by $\pi/4$ made of $n$ rows and $n$ columns, see Figure~\ref{fig:fig_Aztec}. This graph is bipartite and the set of vertices $\Vs_n$ is naturally split into black and white: $\Vs_n=\Bs_n\sqcup \Ws_n$. 

\begin{figure}[h]
  \centering
  \begin{overpic}[width=6cm]{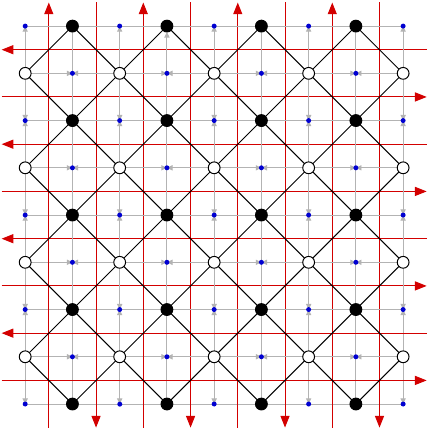}
  \put(-6,11){\scriptsize $\alpha_1$} 
  \put(-6,22){\scriptsize $\beta_1$} 
  \put(8,-4){\scriptsize $\gamma_1$} 
  \put(21,-5){\scriptsize $\delta_1$} 
  \put(-6,32){\scriptsize $\alpha_2$} 
  \put(-6,43){\scriptsize $\beta_2$} 
  \put(32,-4){\scriptsize $\gamma_2$} 
  \put(43,-5){\scriptsize $\delta_2$} 
  \put(-6,55){\scriptsize $\alpha_3$} 
  \put(-6,66){\scriptsize $\beta_3$} 
  \put(53,-4){\scriptsize $\gamma_3$} 
  \put(64,-5){\scriptsize $\delta_3$} 
  \put(-6,77){\scriptsize $\alpha_4$} 
  \put(-6,88){\scriptsize $\beta_4$} 
  \put(75,-4){\scriptsize $\gamma_4$} 
  \put(86,-5){\scriptsize $\delta_4$} 
  \end{overpic}
\caption{Aztec diamond $\Az_n$ with $n=4$. The diamond graph $\Azdiam_n$ is represented with light grey lines, the dual vertices with blue circles, and the four families of oriented train-tracks in red. This figure also sets the notation for the angles $\alphab=(\alpha_j)_{j=1}^n,\dots,\deltab=(\delta_j)_{j=1}^n$ assigned to train-tracks.}\label{fig:fig_Aztec}  
\end{figure}
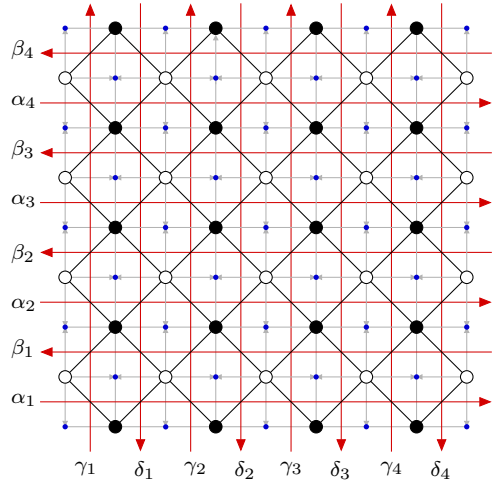

The \emph{set of faces} of $\Az_n$, denoted by $\Fs_n$, consists of the inner faces and the boundary faces, where the latter correspond to all faces of the rotated $\ZZ^2$ adjacent to the boundary of $\Az_n$. To every face of $\Fs_n$, one assigns a vertex generically denoted by $\fs$; it can be seen as a vertex of the dual graph of $\As_n$ modified along the boundary, see Figure~\ref{fig:fig_Aztec} (small blue circles), hence $\fs$ will be referred to as \emph{face} or \emph{dual vertex}. 

The Aztec diamond $\Az_n$ comes with a natural coordinate system: the dual vertex on the bottom left is the origin $0=(0,0)$, every white vertex can be written as $\ws=(\ws_x,\ws_y)$, with $\ws_x\in\{0,2,\dots,2n\}$, $\ws_y\in\{1,3,\dots,2n-1\}$, every black vertex can be written as $\bs=(\bs_x,\bs_y)$, with $\bs_x\in\{1,3,\dots,2n-1\}$, $\bs_y\in\{0,2,\dots,2n\}$, every dual vertex $\fs=(\fs_x,\fs_y)$ either has both even coordinates, or both odd coordinates. 
 
The vertex set of the \emph{diamond graph} $\Azdiam_n$ consists of the dual vertices of $\Fs_n$ and of the vertices of $\Vs_n$,  its edges are obtained by joining every vertex of $\Fs_n$ to the vertices of $\Vs_n$ on the boundary of the face it corresponds to, see Figure~\ref{fig:fig_Aztec} (light grey); the diamond graph is made of quadrilaterals. A \emph{train-track} of $\Az_n$ is a maximal path crossing edges of the diamond graph so that when it enters a quadrilateral, it exits through the opposite side. Since the graph $\Az_n$ is bipartite, each train-track has a natural orientation such that white, resp. black, vertices of $\Az_n$, are on its left, resp. right; we let $\vec{\T}_n$ denote the set of oriented train-tracks, see Figure~\ref{fig:fig_Aztec} (red). Train-tracks of $\vec{\T}_n$ come in four families: two sets of horizontal train-tracks with left-to-right or right-to-left orientation, and two sets of vertical train-tracks with bottom-to-top or top-to-bottom orientation. 
 
\paragraph{Dimer model.} A \emph{dimer configuration} of $\Az_n$ is a subset $\Ms$ of edges such that each vertex of $\Az_n$ is incident to exactly one edge of $\Ms$; we let $\M(\Az_n)$ denote the set of dimer configurations of $\Az_n$. Assume a positive weight function $\nu$ is assigned to edges of $\Az_n$. The \emph{dimer Boltzmann measure}, denoted by $\PP$, is the probability measure on $\M(\Az_n)$ defined by
\begin{equation*}
\forall\, \Ms\in \M(\Az_n),\quad \PP(\Ms)=\frac{\nu(\Ms)}{Z(\nu)}, 
\end{equation*}
where $\nu(\Ms)=\prod_{\es\in\Ms}\nu_e$, and $Z(\nu)=\sum_{\Ms\in\M(\Az_n)} \nu(M)$ is the normalizing constant known as the \emph{partition function}. 

Suppose that we have two positive weight functions $\nu,\tilde{\nu}$ assigned to edges of $\Az_n$. The corresponding dimer models are said to be \emph{gauge equivalent} \cite{KOS}, if there exists a function $\psi$ defined on $\Vs_n$ such that, for every edge $\ws\bs$ of $\Es_n$, we have $\tilde{\nu}_{\ws\bs}=\psi(\ws)\nu_{\ws\bs}\psi(\bs)$. Two gauge equivalent dimer models yield the same dimer Boltzmann measure. An equivalent useful way of defining gauge equivalence is the following: consider an inner face $\fs$ of degree $2k$ of $\Az_n$, and denote by $\ws_1,\bs_1,\dots,\ws_k,\bs_k$ its boundary vertices in counterclockwise order, then the \emph{face weight} $\Wscr_\fs(\nu)$ of $\fs$ is defined as the alternate product of the weights:
\begin{align*}
\Wscr_\fs(\nu)=\prod_{j=1}^k \frac{\nu_{\ws_j\bs_j}}{\nu_{\ws_j\bs_{j-1}}},
\end{align*}
using cyclic notation for vertices. Then, two weight functions $\nu,\tilde{\nu}$ are gauge equivalent if and only if, for every inner face $\fs$ of $\Az_n$, $\Wscr_\fs(\nu)=\Wscr_\fs(\tilde{\nu})$~\cite{KOS}. As a consequence of~\cite{KOS}, we have an explicit relation between partition functions associated to gauge equivalent weight functions $\nu$ and $\tilde{\nu}$. 
\begin{equation}\label{equ:partition_gauge}
Z(\tilde{\nu})=\frac{\tilde{\nu}(\Ms_0)}{\nu(\Ms_0)}\,Z(\nu),
\end{equation}
where $\Ms_0$ is any fixed dimer configuration of $\Az_n$. 

A key tool for studying the dimer model is the \emph{Kasteleyn matrix}~\cite{Kasteleyn1,TF,Percus}. Instead of defining the original Kasteleyn matrix which uses orientation of the edges, we define a modified version, introduced by Kuperberg~\cite{Kuperberg}, that uses complex phases assigned to edges. Suppose that every edge $\ws\bs$ of $\Es_n$ is assigned a modulus one complex number $\phi_{\ws\bs}$ such that, for every inner face $\fs$ of degree $2k$ of $\Az_n$, the \emph{Kasteleyn condition} is satisfied:
\begin{align}\label{equ:Kast_orientation}
\prod_{j=1}^k \frac{\phi_{\ws_j\bs_j}}{\phi_{\ws_j\bs_{j-1}}}=(-1)^{k+1}.
\end{align}
Then, the associated \emph{Kasteleyn matrix} $\Ks$ is the corresponding weighted adjacency matrix: rows, resp. columns, of $\Ks$ are indexed by white, resp. black, vertices of $\As_n$; non-zero coefficients of $\Ks$ correspond to edges of $\Es_n$ and, for every edge $\ws\bs$, the coefficient $\Ks_{\ws,\bs}$ is defined by
\begin{align*}
\Ks_{\ws,\bs}=\phi_{\ws\bs}\nu_{\ws\bs}.  
\end{align*}
Two key results of the dimer model are the expression of the partition function and of the dimer Boltzmann measure using the determinant of the matrix $\Ks$ and its inverse. More precisely, we have
\begin{thm}\cite{Kasteleyn1,TF,Percus,Kuperberg}\label{thm:partition_Kast}
The dimer partition function is equal to
\begin{align*}
Z(\nu)=|\det(\Ks)|. 
\end{align*}
\end{thm}

\begin{thm}\cite{Kenyon1}\label{thm:measure_Kenyon}
The probability of all dimer configurations containing a fixed subset of edges $\{\es_1=\ws_1\bs_1,\dots,\es_k=\ws_k\bs_k\}$ of $\Es_n$ is explicitly given by
\begin{align*}
\PP(\es_1,\dots,\es_k)=\prod_{j=1}^k \Ks_{\ws_j,\bs_j}\det \bigl(\Ks^{-1}_{\bs_i,\ws_j}\bigr)_{1\leq i,j\leq k}.
\end{align*}
\end{thm}

\subsection{Fock's dimer model on the Aztec diamond}\label{sec:Fock}

We now turn to the definition of \emph{Fock's weights}~\cite{Fock} underlying the dimer model of interest to this paper. We only highlight the main tools needed, more details can be found in the paper~\cite{BCdT:genusg} providing a thorough study of this model in the case of infinite minimal graphs. We first need some tools from Riemannian geometry. 

\paragraph{M-curves.} Let $\Sigma$ be an \emph{M-curve}, that is a compact Riemann surface endowed with an anti-holomorphic involution $\sigma$ whose set of fixed points is given by $g+1$ topological circles, where $g$ is the genus of $\Sigma$. Fix $x_0$ a real point of $\Sigma$ and denote by $A_0$ the corresponding real component and by $A_1,\dots,A_g$ the remaining ones. The real locus separates $\Sigma$ into two connected surfaces with boundary $\Sigma^-,\Sigma^+$, and we fix an orientation of the real locus so that the boundary of $\Sigma^+$ is equal to $A_0-(A_1+\dots+A_g)$. We use the same notation $A_j$ for the oriented cycle in $\Sigma$ and its homology class in $H_1(\Sigma;\ZZ)$.

There are homology classes $B_1,\dots,B_g\in H_1(\Sigma,\ZZ)$ with $\sigma_*(B_i)=-B_i$ such that 
$\{A_1,\dots,$ $A_g,B_1,\dots,B_g\}$ forms a basis of $H_1(\Sigma,\ZZ)$ and satisfies, for all $i,j\in\{1,\dots,g\}$,
\[
A_i\wedge A_j=0,\quad B_i\wedge B_j=0,\quad A_i\wedge B_j=\delta_{i,j}, 
\]
where $\wedge$ denotes the intersection form. 

The complex vector space of holomorphic differential forms has dimension $g$. Denote by $\vec{\omega}=(\omega_1,\dots,\omega_g)$ the basis of this space determined by
\[
\forall\ i,j\in\{1,\dots,g\},\quad \int_{A_i}\omega_j=\delta_{i,j}. 
\]
Let $\Omega$ be the matrix with entries $\Omega_{i,j}=\int_{B_i}\omega_j$. This
matrix is purely imaginary (in our case of M-curves), symmetric and
its imaginary part is positive definite. The \emph{period matrix} $\begin{pmatrix}
I_g& \Omega                                                                                                                                                  \end{pmatrix}
$
generates the full rank lattice $\Lambda=\ZZ^g\oplus \Omega\ZZ^g$ in $\CC^g$. The \emph{Jacobian variety} of $\Sigma$ is defined to be $\Jac(\Sigma)=\CC^g/\Lambda$.

\paragraph{Abel-Jacobi map.} A \emph{divisor} on $\Sigma$ is a formal linear combination of points on $\Sigma$ with integer coefficients. The set of divisors is endowed with a natural grading $\Div(\Sigma)=\bigoplus_{n\in\ZZ}\Div^n(\Sigma)$, where the \emph{degree} of a divisor is the sum of its integer coefficients. A divisor is said to be \emph{principal} if it represents the zeros and the poles of a non-zero meromorphic function $f$ on $\Sigma$; it thus has degree 0. Two divisors are linearly equivalent if their difference is a principal divisor; the set of linear equivalence classes of divisors forms a $\ZZ$-graded Abelian group, denoted by $\Pic(\Sigma)=\bigoplus_{n\in\ZZ}\Pic^n(\Sigma)$. By Abel's theorem, there is an injection, the \emph{Abel-Jacobi map}, from $\Pic^0(\Sigma)$ to $\Jac(\Sigma)$ defined by 
\[
D=\sum_{i} (y_i-x_i)\mapsto \sum_{i}\int_{x_i}^{y_i}\vec{\omega}.
\]
By Jacobi's inversion theorem, this map induces an isomorphism of Abelian groups $\Pic^0(\Sigma)\simeq \Jac(\Sigma)$; following standard practice, we use the same notation for the equivalence class of a degree 0 divisor and for its corresponding element in $\Jac(\Sigma)$. 

\paragraph{Riemann theta functions.} The \emph{Riemann theta function} $\theta(z|\Omega)$ associated to $\Sigma$ is defined by
\[
\forall\, z\in\CC^g,\quad \theta(z|\Omega)=\sum_{n\in\ZZ^g}e^{i\pi(n\cdot\Omega n+2n\cdot z)}. 
\]
For $\thcharp{\delta'}{\delta''}\in\bigl(\frac{1}{2}\ZZ\bigr)^{2g}$, the \emph{theta function with characteristic $\thcharp{\delta'}{\delta''}$}, is defined by 
\[
\theta\thchar{\delta'}{\delta''}(z|\Omega)=\sum_{n\in\ZZ^g}
e^{i\pi[(n+\delta')\cdot\Omega(n+\delta')+2(n+\delta')\cdot(z+\delta'')]},
\]
so that the theta function with $\thcharp{0}{0}$ characteristic is the Riemann theta function. A characteristic $\thcharp{\delta'}{\delta''}$ is said to be \emph{even} (\emph{odd}) if $2\delta\cdot 2\delta'$ is even (odd). A theta function with even (odd) characteristic is an even (odd) function.

\paragraph{Prime form.} The \emph{prime form} is the building block for meromorphic functions on $\Sigma$. We outline the definition and refer to~\cite{ThetaTata2} for more details. 
Consider a non-degenerate theta characteristic, that is a characteristic $\thcharp{\delta'}{\delta''}$ such that $\ud_z \theta\thchar{\delta'}{\delta''}(0)\neq 0$, this implies in particular that it must be odd. Consider also $\xi\thchar{\delta'}{\delta''}$ the square root of the holomorphic form $\ud_z \theta\thchar{\delta'}{\delta''}(0)\cdot\vec{\omega}$. Then, the \emph{prime form} is defined to be: 
\[
\forall\,x,y\in\Sigma, \quad 
E(x,y)=\frac{\theta\thchar{\delta'}{\delta''}(y-x)}{\xi\thchar{\delta'}{\delta''}(x)\xi\thchar{\delta'}{\delta''}(y)},
\]
where $\theta\thchar{\delta'}{\delta''}(y-x):=\theta\thchar{\delta'}{\delta''}(\int_{x}^y \vec{\omega})$. The prime form is independent of the choice of non-degenerate theta characteristic. Note that, as a function, it is only well defined on the universal cover $\widetilde{\Sigma}$ of $\Sigma$: it has well identified quasi-periods along lifts of the cycles $(A_j)$ and $(B_j)$. Its main properties are that it is equal to $0$ if and only if $x=y$, it is skew symmetric and has first order zeros, \cite[p. 3.210]{ThetaTata2}.

\paragraph{Angles and discrete Abel map.} In order to define Fock's weights, we need another type of data related to graph properties of the Aztec diamond. A bipartite graph is said to be \emph{minimal}~\cite{Thurston,GK} if oriented train-tracks of $\vec{\T}$ do not self intersect and do not form parallel bigons, where a \emph{parallel bigon} is a  pair of train-tracks intersecting more than once in the same direction; the Aztec diamond is of course a minimal graph.

To every train-track $T$ of $\vec{\T}_n$, we assign an element $\alpha_T$ of $A_0$, referred to as its \emph{angle}. Let us introduce some notation related to the fact that our set of train-tracks is naturally split into four: $\alphab=(\alpha_j)_{j=1}^n$, resp. $\betab=(\beta_j)_{j=1}^n$ are the angles of the left-to-right, resp. right-to-left, horizontal train-tracks starting from the bottom; and $\gammab=(\gamma_j)_{j=1}^n$, $\deltab=(\delta_j)_{j=1}^n$ are the angles of the bottom-to-top, resp.~top-to-bottom, vertical train-tracks starting from the left, see Figure~\ref{fig:fig_Aztec}. From now on, we suppose that the angles satisfy the following condition: for all $i,j,k,\ell\in\{1,\dots,n\}$, the cyclic order $\alpha_i<\gamma_j<\beta_k<\delta_\ell$ is satisfied on $A_0$. Note that there is no ordering condition on the angles within one of these subsets. By~\cite[Corollary 29]{BCdT:immersion}, this is a necessary and sufficient condition for 
the angles $\alphab,\betab,\gammab,\deltab$ to define a \emph{minimal immersion} of the natural periodic extension of the Aztec diamond. 
We introduce the short notation 
\begin{align}\label{equ:angle_condition}
\alphab <\gammab<\betab<\deltab, 
\end{align}
for angles $\alphab,\betab,\gammab,\deltab$ satisfying the above cyclic ordering on $A_0$.

Following Fock~\cite{Fock}, the~\emph{discrete Abel map}, denoted by $\mapd$, is
a map from the vertices of $\Azdiam_n$ (that is, the dual vertices of $\Fs_n$ and the vertices of $\Vs_n$)
to $\Pic(\Sigma)$, defined as follows: take as reference dual vertex the origin
$0=(0,0)$ and set $\mapd(0)=d$, for some $d\in\Pic^0(\Sigma)$. Then, the
discrete Abel map is defined inductively along edges and is such that along a
directed edge of $\Azdiam_n$ crossing a train-track $T$, the value of $\mapd$
formally increases, resp. decreases, by $\alpha_T$ if one arrives at a black
vertex or leaves a white vertex, resp. leaves a black vertex of arrives at a
white vertex. This map is well defined~\cite{Fock}, and for every vertex $\xs$
of $\Azdiam_n$, the degree $\mapd(\xs)\in \Pic(\Sigma)$ is equal to $1$, resp. 0,
resp. $-1$, at every black, resp. dual, resp. white vertex of $\Azdiam_n$. In particular, for every dual vertex $\fs$, the divisor $\mapd(\fs)$ belongs to $\Pic^0(\Sigma)$, and by~\cite[Lemma 15]{BCdT:genusg}, its image through the Abel-Jacobi map belongs to $(\RR/\ZZ)^g\subset\Jac(\Sigma)$. An example of computation of the discrete Abel map is given in Figure~\ref{fig:fig_genus1} (right).

\paragraph{Fock's Kasteleyn matrix.} Consider a maximal curve $\Sigma$, and angles $\alphab,\betab,\gammab,\deltab$ in $A_0$ assigned to train-tracks of $\Az_n$ satisfying the cyclic condition $\alphab<\gammab<\betab<\deltab$; consider an additional parameter $t\in(\RR/\ZZ)^g\subset \Jac(\Sigma)$. Then, \emph{Fock's Kasteleyn matrix}, denoted by $\Ks$, has rows indexed by white vertices, columns by black ones, and non-zero coefficients defined by, for every edge $\ws\bs$ of $\Az_n$ crossed by two train-tracks with angles $\alpha,\beta$,
\begin{equation}\label{eq:Focks_weight}
\Ks_{\ws,\bs}=\frac{E(\alpha,\beta)}{\theta(t+\mapd(\fs))\theta(t+\mapd(\fs'))},
\end{equation}
where $\fs,\fs'$ are the dual faces adjacent to $\ws\bs$, see Figure~\ref{fig:quad}.

\begin{figure}[ht]
\begin{center}
\begin{overpic}[width=3cm]{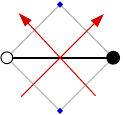}
  \put(-8,45){\scriptsize $\ws$} 
  \put(100,45){\scriptsize $\bs$} 
  \put(48,94){\scriptsize $\fs$} 
  \put(48,-9){\scriptsize $\fs'$} 
  \put(85,82){\scriptsize $\alpha$} 
  \put(8,82){\scriptsize $\beta$} 
\end{overpic}
\caption{Notation used in the definition of Fock's Kasteleyn matrix.}\label{fig:quad}
\end{center}
\end{figure}
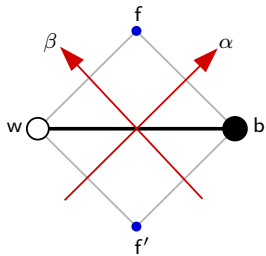

Entries of this matrix are complex, but we prove in~\cite[Proposition 31]{BCdT:genusg} that it is a Kasteleyn matrix, where recall that this means that it corresponds to a positive weight function $\nu$ multiplied by a complex phase $\phi$ satisfying the Kasteleyn condition~\eqref{equ:Kast_orientation}. 

\begin{rem}
We have noted that, seen as a function, the prime form $E$ is only defined on the universal cover $\widetilde{\Sigma}$ of $\Sigma$. Nevertheless, in~\cite[Remark 30]{BCdT:genusg}, we prove that for any choice of lifts of the angles $\alpha,\beta$ in the universal cover $\widetilde{A_0}$, the corresponding Kasteleyn operators are gauge equivalent. As a consequence, in the sequel, whenever results are true up to gauge equivalence, the choice of lifts does not matter; and whenever results are not of this type, one should keep in mind that a specific choice of lift has to be made, and that the result is true for any choice of lift. Since these subtle questions have been treated in great detail in~\cite{BCdT:genusg}, we choose not to re-address them here and to consider the Kasteleyn operator as directly defined with parameters on the base surface~$\Sigma$. 
\end{rem}

\subsection{Examples: genus 0 and genus 1 cases}\label{sec:example_0_1}

By way of example, we make the genus 0 and 1 cases explicit and prove that we
recover as specific cases well known models, namely the Aztec diamond with
Stanley's weights (genus 0)~\cite{Propp_talk,BYYangPhD}, and Borodin-Duits' biased $2\times 2$ periodic weights (genus 1)~\cite{BorodinDuits}, see also Figures~\ref{fig:fig_Stanley} and~\ref{fig:fig_genus1} below. 

\begin{figure}[h]
\begin{minipage}{0.48\textwidth}
  \centering
  \begin{overpic}[width=5cm]{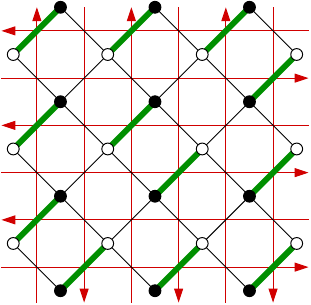}
  \put(-6,11){\scriptsize $\alpha_1$} 
  \put(-6,26){\scriptsize $\beta_1$} 
  \put(8,-4){\scriptsize $\gamma$} 
  \put(24,-5){\scriptsize $\delta$} 
  \put(5,9){\scriptsize $x_1$} 
  \put(27,9){\scriptsize $y_1$} 
  \put(5,28){\scriptsize $z_1$} 
  \put(27,28){\scriptsize $w_1$} 
  \put(36,9){\scriptsize $x_1$} 
  \put(58,9){\scriptsize $y_1$} 
  \put(36,28){\scriptsize $z_1$} 
  \put(58,28){\scriptsize $w_1$} 
  \put(66,9){\scriptsize $x_1$} 
  \put(89,9){\scriptsize $y_1$} 
  \put(66,28){\scriptsize $z_1$} 
  \put(89,28){\scriptsize $w_1$} 
  \put(5,39){\scriptsize $x_2$} 
  \put(27,39){\scriptsize $y_2$} 
  \put(5,58){\scriptsize $z_2$} 
  \put(27,58){\scriptsize $w_2$} 
  \put(36,39){\scriptsize $x_2$} 
  \put(58,39){\scriptsize $y_2$} 
  \put(36,58){\scriptsize $z_2$} 
  \put(58,58){\scriptsize $w_2$} 
  \put(66,39){\scriptsize $x_2$} 
  \put(89,39){\scriptsize $y_2$} 
  \put(66,58){\scriptsize $z_2$} 
  \put(89,58){\scriptsize $w_2$} 
  \put(5,70){\scriptsize $x_3$} 
  \put(27,70){\scriptsize $y_3$} 
  \put(5,89){\scriptsize $z_3$} 
  \put(27,89){\scriptsize $w_3$} 
  \put(36,70){\scriptsize $x_3$} 
  \put(58,70){\scriptsize $y_3$} 
  \put(36,89){\scriptsize $z_3$} 
  \put(58,89){\scriptsize $w_3$} 
  \put(66,70){\scriptsize $x_3$} 
  \put(89,70){\scriptsize $y_3$} 
  \put(66,89){\scriptsize $z_3$} 
  \put(89,89){\scriptsize $w_3$} 
  \end{overpic}
\caption{Aztec diamond with Stanley's weights, and choice of reference dimer configuration $\Ms_0$.}\label{fig:fig_Stanley}  
\end{minipage}
\begin{minipage}{0.04\textwidth}
\end{minipage}
\begin{minipage}{0.48\textwidth}
  \centering
  \begin{overpic}[width=5cm]{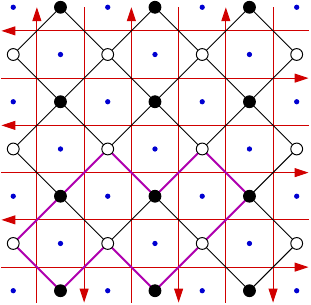}
  \put(-6,11){\scriptsize $\alpha$} 
  \put(-6,26){\scriptsize $\beta$} 
  \put(8,-4){\scriptsize $\gamma$} 
  \put(24,-5){\scriptsize $\delta$} 
  \put(0,0){\textcolor{blue}{\tiny $t$}} 
  \put(30,0){\textcolor{blue}{\tiny $t+\frac{\pi}{2}$}}
  \put(63,0){\textcolor{blue}{\tiny $t$} }
  \put(90,0){\textcolor{blue}{\tiny $t+\frac{\pi}{2}$}} 
  \put(8,8){\tiny $\frac{1}{b}$} 
  \put(25,8){\tiny $a$}
  \put(39,8){\tiny $b$} 
  \put(58,8){\tiny $a$}
  \put(69,8){\tiny $\frac{1}{b}$} 
  \put(88,8){\tiny $a$}
  \put(16,15){\textcolor{blue}{\tiny $t+\rho$} }
  \put(42,15){\textcolor{blue}{\tiny $t+\frac{\pi}{2}+\rho$} }
  \put(75,15){\textcolor{blue}{\tiny $t+\rho$} }
  \put(10,23){\tiny $\frac{a}{b}$} 
  \put(25,23){\tiny $1$}
  \put(40,23){\tiny $ab$} 
  \put(56,23){\tiny $1$}
  \put(71,23){\tiny $\frac{a}{b}$} 
  \put(92,23){\tiny $1$}
  \put(0,30){\textcolor{blue}{\tiny $t+\frac{\pi}{2}$} }
  \put(32,30){\textcolor{blue}{\tiny $t$} }
  \put(60,30){\textcolor{blue}{\tiny $t+\frac{\pi}{2}$} }
  \put(92,30){\textcolor{blue}{\tiny $t$} }
  \put(8,38){\tiny $b$} 
  \put(25,38){\tiny $a$}
  \put(39,38){\tiny $\frac{1}{b}$} 
  \put(58,38){\tiny $a$}
  \put(69,38){\tiny $b$} 
  \put(88,38){\tiny $a$}
  \put(12,45){\textcolor{blue}{\tiny $t+\frac{\pi}{2}+\rho$} }
  \put(44,45){\textcolor{blue}{\tiny $t+\rho$} }
  \put(72,45){\textcolor{blue}{\tiny $t+\frac{\pi}{2}+\rho$} }
  \put(10,54){\tiny $ab$} 
  \put(25,54){\tiny $1$}
  \put(40,54){\tiny $\frac{a}{b}$} 
  \put(56,54){\tiny $1$}
  \put(71,54){\tiny $ab$} 
  \put(92,54){\tiny $1$}
  \put(0,61){\textcolor{blue}{\tiny $t$} }
  \put(30,61){\textcolor{blue}{\tiny $t+\frac{\pi}{2}$} }
  \put(63,61){\textcolor{blue}{\tiny $t$} }
  \put(90,61){\textcolor{blue}{\tiny $t+\frac{\pi}{2}$} }
  \put(8,70){\tiny $\frac{1}{b}$} 
  \put(25,70){\tiny $a$}
  \put(39,70){\tiny $b$} 
  \put(58,70){\tiny $a$}
  \put(69,70){\tiny $\frac{1}{b}$} 
  \put(88,70){\tiny $a$}
  \put(16,76){\textcolor{blue}{\tiny $t+\rho$} }
  \put(42,76){\textcolor{blue}{\tiny $t+\frac{\pi}{2}+\rho$} }
  \put(75,76){\textcolor{blue}{\tiny $t+\rho$} }
  \put(10,84){\tiny $\frac{a}{b}$} 
  \put(25,84){\tiny $1$}
  \put(40,84){\tiny $ab$} 
  \put(56,84){\tiny $1$}
  \put(71,84){\tiny $\frac{a}{b}$} 
  \put(92,84){\tiny $1$}
  \put(0,91){\textcolor{blue}{\tiny $t+\frac{\pi}{2}$} }
  \put(32,91){\textcolor{blue}{\tiny $t$} }
  \put(60,91){\textcolor{blue}{\tiny $t+\frac{\pi}{2}$} }
  \put(92,91){\textcolor{blue}{\tiny $t$} }
  \end{overpic}
\caption{Aztec diamond with biased $2\times 2$ periodic weights~\cite{BorodinDuits}, and the discrete Abel map.}\label{fig:fig_genus1}
\end{minipage}
\end{figure}

\paragraph{Genus 0.}
The underlying M-curve $\Sigma$ is the Riemann sphere $\hat{\CC}$, with
involution $z\mapsto\frac{1}{\bar z}$ and $A_0=S^1=\{z\in\hat{\CC}:\, |z|=1\}$. The prime form $E(u,v)$ is equal to $v-u$, and the Riemann theta function is the constant function $1$, so that the discrete Abel map is not needed. To a point $\alpha\in A_0=S^1$ is assigned in a bijective way an angle $\bar{\alpha}\in\bar{A}_0:=\RR/\pi\ZZ$ defined by $\alpha=e^{2i\bar{\alpha}}$; observe that this bijection preserves the cyclic orders on $A_0$ and $\bar{A}_0$. For $\alpha,\beta\in A_0$, we have $E(\alpha,\beta)=e^{2i\bar{\beta}}-e^{2i\bar{\alpha}}$, and
Fock's Kasteleyn weights are equal to:
 \begin{equation}\label{equ:critical_weights}
\Ks_{\ws,\bs}=\beta-\alpha=e^{2i\bar{\beta}}-e^{2i\bar{\alpha}}=2ie^{i(\bar{\alpha}+\bar{\beta})}\sin(\bar{\beta}-\bar{\alpha}).
 \end{equation}
Up to a factor $i$,
these are Kenyon's critical weights on isoradial graphs introduced in~\cite{Kenyon:crit}.

Let us now prove that, as a specific case, one recovers the dimer model with
\emph{Stanley's weights}~\cite{Propp_talk,BYYangPhD} defined as follows: edges are assigned positive weights $\xb=(x_j)_{j=1}^n,\yb=(y_j)_{j=1}^n, \wb=(w_j)_{j=1}^n,\zb=(z_j)_{j=1}^n$ as in Figure~\ref{fig:fig_Stanley} (left). On Fock's weights side, impose the following condition: choose $\alpha_1,\gamma,\delta\in A_0=S^1$ satisfying the cyclic order 
$\alpha_1<\gamma<\delta$, and for all $j\in\{1,\dots,n\}$, set $\gamma_j=\gamma,\delta_j=\delta$.

\begin{prop}\label{prop:gauge_Stanley}
Suppose that we are given $\xb,\yb,\zb,\wb$ defining Stanley's weights, and consider $\alpha_1,\gamma,\delta\in A_0$ as above. Then, there exists $(\alpha_j)_{j=2}^n,(\beta_j)_{j=1}^n$ such that $\alphab<\gamma<\betab<\delta$ and such that the dimer model with Fock's weights $\alphab,\betab,\gammab\equiv\gamma,\deltab\equiv\delta$ in genus 0 is gauge equivalent to the dimer model with Stanley's weights. 
\end{prop}
\begin{proof}
For the purpose of this proof, let us denote by $\nu$ Fock's weight function and by $\tilde{\nu}$ Stanley's one. We are given $\tilde{\nu}$, that is $\xb,\yb,\zb,\wb$, and on Fock's side we are given $\alpha_1,\gamma,\delta\in A_0$ such that $\alpha_1<\gamma<\delta$. We need to prove that there exist $(\alpha_j)_{j=2}^n,(\beta_j)_{j=1}^n$ such that $\alphab<\gamma<\betab<\delta$, and such that $\nu$ and $\tilde{\nu}$ are gauge equivalent. Returning to Section~\ref{sec:Aztec_dimer}, this means that all face weights have to be equal. 

For the weight function $\tilde{\nu}$, there are two kinds of face weights corresponding to inner faces with both odd coordinates, resp. both even coordinates
\begin{align*}
\forall\, j\in\{1,\dots,n\}, \quad \frac{x_j w_j}{y_j z_j}, \text{ resp.}\quad \forall\,j\in\{1,\dots,n-1\},\quad & \frac{y_{j+1}z_j}{x_{j+1}w_j}.
\end{align*}
Note that all face weights along a given row are equal. 

Returning to the definition of Fock's weights with the above specification in the genus $0$ case, see Equation~\eqref{equ:critical_weights}, we want to have the following equalities
\begin{equation}\label{equ:gauge_Stanley}
\begin{aligned}
\forall\, j\in\{1,\dots,n\}, \quad &\frac{x_j w_j}{y_j z_j}=\frac{|\sin(\bar{\gamma}-\bar{\alpha}_j)|}{|\sin(\bar{\delta}-\bar{\alpha}_j)|}\frac{|\sin(\bar{\delta}-\bar{\beta}_j)|}{|\sin(\bar{\beta}_j-\bar{\gamma})|}\\
\forall\,j\in\{1,\dots,n-1\},\quad & \frac{y_{j+1}z_j}{x_{j+1}w_j}=\frac{|\sin(\bar{\delta}-\bar{\alpha}_{j+1})|}{|\sin(\bar{\gamma}-\bar{\alpha}_{j+1})|}\frac{|\sin(\bar{\beta}_j-\bar{\gamma})|}{|\sin(\bar{\delta}-\bar{\beta}_j)|}.
\end{aligned}
\end{equation}
The main tool we use is that, for all $\bar{u},\bar{v}\in \bar{A}_0$, 
the function $\bar{x}\in \bar{A}_0\mapsto \frac{|\sin(\bar{x}-\bar{u})|}{|\sin(\bar{x}-\bar{v})|}$ is non-negative, has a zero at $\bar{x}=\bar{u}$, a pole at $\bar{x}=\bar{v}$, is continuous except at $\bar{x}=\bar{v}$, and takes all values in $[0,\infty)$ on the intervals $[\bar{u},\bar{v}),(\bar{v},\bar{u}]$ of $\bar{A}_0$.

Fix any $\alpha_1,\gamma,\delta$ satisfying $\alpha_1<\gamma<\delta$. Then, by the above there exists $\beta_1\in(\gamma,\delta)$ such that $\frac{x_1 w_1}{y_1 z_1}=\frac{|\sin(\bar{\gamma}-\bar{\alpha}_1)|}{|\sin(\bar{\delta}-\bar{\alpha}_1)|}\frac{|\sin(\bar{\delta}-\bar{\beta}_1)|}{|\sin(\bar{\beta}_1-\bar{\gamma})|}$. Using the same argument, there exists $\alpha_2\in(\delta,\gamma)$ such that 
$\frac{y_{2}z_1}{x_{2}w_1}=\frac{|\sin(\bar{\delta}-\bar{\alpha}_{2})|}{|\sin(\bar{\gamma}-\bar{\alpha}_{2})|}\frac{|\sin(\bar{\beta}_1-\bar{\gamma})|}{|\sin(\bar{\delta}-\bar{\beta}_1)|}$.
Again, using this argument, there exists $\beta_2\in(\gamma,\delta)$ such that $\frac{x_2 w_2}{y_2 z_2}=\frac{|\sin(\bar{\gamma}-\bar{\alpha}_2)|}{|\sin(\bar{\delta}-\bar{\alpha}_2)|}\frac{|\sin(\bar{\delta}-\bar{\beta}_2)|}{|\sin(\bar{\beta}_2-\bar{\gamma})|}$, and we continue determining $\alpha_3$ using the second equation, etc. To solve these equations we need $2n-1$ parameters $(\alpha_j)_{j=2}^n,(\beta_j)_{j=1}^n$, and we can choose them so that $\alphab<\gamma<\betab<\delta$.   
\end{proof}

\begin{rem}
  The fact that there are three parameters (we chose here $\alpha_1$, $\gamma$ and
  $\delta$) that we can fix to arbitrary values is a consequence of the fact
  that the probability measure is invariant if we apply to all parameters a
  Möbius transform preserving the unit circle, and such Möbius transformations
  are transitives on triple of points. See e.g.\ \cite[Section~5]{KO:Harnack},
  and Section~\ref{sec:arctic_ellipse} of the present article.
\end{rem}

\paragraph{Genus 1.} The M-curve $\Sigma$ is the complex torus
$\TT(\tau)=\CC/(\ZZ+\tau\ZZ)$, for some modular parameter $\tau \in i\RR^+$,
where the involution is given by the complex conjugation and $A_0=\RR/\ZZ$. The theta function $\theta(u|\Omega)$ is the Jacobi theta function $\theta_3(\pi u|\tau)$, see~\cite[Equation (1.2.13)]{Lawden}, and the prime form $E(u,v)$ is equal to $\frac{\theta_1(\pi(v-u)|\tau)}{\pi\theta_1'(0)}$, where $\theta_1$ is the rescaled version of the theta function with characteristic $\thcharp{1/2}{1/2}$. Whenever no confusion occurs, the reference to $\tau$ is omitted in the notation of the theta functions.
As a consequence, Fock's Kasteleyn weights are equal to
\begin{equation}\label{equ:genus1_weights}
\Ks_{\ws,\bs}=\frac{\theta_1(\pi(\beta-\alpha))}{\pi\theta_1'(0)\theta_3(\pi(t+\mapd(\fs)))\theta_3(\pi(t+\mapd(\fs')))}.
\end{equation}

We now prove that as a specific case one recovers the \emph{biased $2\times 2$ periodic} dimer model studied by Borodin and Duits~\cite{BorodinDuits}, defined as follows. Consider two parameters $a>0, b\in(0,1]$; the Aztec diamond is naturally made of $2n$ rows of edges. Row weights repeat each four rows; a pattern of four rows is given by: $(\frac{1}{b},a,b,a \text{ + horiz. repetitions})$, $(\frac{a}{b},1,ab,1\text{ + horiz. repetitions})$, $(b,a,\frac{1}{b},a$\text{ + horiz. repetitions})$, (ab,1,\frac{a}{b},1\text{ + horiz. repetitions})$, see Figure~\ref{fig:fig_genus1} (right). Note that if the Aztec diamond has odd size, the last two columns and rows contain only half of the pattern; note also that the weights come in periods of $2\times 2$, circled in magenta in Figure~\ref{fig:fig_genus1} (right), hence the name. For Fock's weights, impose the following preliminary conditions: for all $j\in\{1,\dots,n\}$, set $\alpha_j=\alpha,\beta_j=\beta,\gamma_j=\gamma,\delta_j=\delta$, where $\alpha,\beta,\gamma,\delta\in A_0=\RR/\ZZ$ are such that $\alpha<\gamma<\beta<\delta$; moreover suppose that $\beta-\alpha=\delta-\gamma=\frac{1}{2}$, implying that we have $\gamma-\alpha=\delta-\beta:=\rho\in(0,\frac{1}{2})$, and $\beta-\gamma=\alpha-\delta=\frac{1}{2}-\rho\in(0,\frac{1}{2})$, so that we have one free angle parameter $\rho\in(0,\frac{1}{2})$.

\begin{prop}\label{prop:gauge_Borodin_Duits}
Suppose that we are given $a>0, b\in(0,1)$ defining biased $2\times 2$ periodic
weights, and consider $\alpha,\beta,\gamma,\delta\in A_0$ and
$\rho\in(0,\frac{1}{2})$ as specified above.
Then, there exists $\rho\in(0,\frac{1}{2})$ and $\tau\in i \RR^*_+$ such that the
dimer model with Fock's weights in genus 1 for $t=\frac{1}{4}$
is gauge equivalent to the dimer
model with biased $2\times 2$ periodic weights. 
\end{prop}
\begin{proof}
We are given $a>0$ and $b\in(0,1)$ and need to prove that there exists
$\rho\in(0,\frac{1}{2})$, $\tau\in i \RR^*_+$, such that
for a good choice of $t$, 
the face weight of each face is the same in both settings.

The computation of the discrete Abel map is illustrated in Figure~\ref{fig:fig_genus1} (right, blue), taking into account that the theta function $\theta_3$ is $\pi$-periodic. 
Computing the face weights similarly to the genus 0 case, for each of the weight functions there are four distinct face weights, which can for instance be computed from the four faces surrounded by the magenta line of Figure~\ref{fig:fig_genus1}. We are looking for $\rho,\tau,t$ satisfying the following four equalities:
\begin{align*}
\frac{1}{a^2}& =
\frac{\theta_3(\pi (t+\frac{1}{2}))^2}{\theta_3(\pi t)^2}
\frac{\theta_1(\pi \rho)^2}{\theta_1(\pi(\frac{1}{2}- \rho))^2},&
\frac{1}{a^2}& =
\frac{\theta_3(\pi t)^2}{\theta_3(\pi(t+\frac{1}{2}))^2}
\frac{\theta_1(\pi \rho)^2}{\theta_1(\pi(\frac{1}{2}-\rho))^2},\\
(ab)^2 &=
\frac{\theta_3(\pi (t+\rho))^2}{\theta_3(\pi(t+\frac{1}{2}+ \rho))^2}
\frac{\theta_1(\pi(\frac{1}{2}-\rho))^2}{\theta_1(\pi \rho)^2},&
\frac{a^2}{b^2}&=
\frac{\theta_3(\pi(t+\frac{1}{2}+\rho))^2}{\theta_3(\pi( t+ \rho))^2}
\frac{\theta_1(\pi(\frac{1}{2}-\rho))^2}{\theta_1(\pi \rho)^2}.
\end{align*}
Since all the quantities involved are positive, we can remove the squares.

  The two equalities on the first line imply that
$\theta_3(\pi(t+\frac{1}{2}))=\theta_3(\pi t)$, so $t=\pm\frac{1}{4}$.
Taking the product of the first two equalities of both lines yields
  $b=\theta_3(\pi(t+\rho))/\theta_3(\pi(t+\frac{1}{2}+\rho))$, which is strictly less than 1
  for $\rho\in(0,\frac{1}{2})$ only if we choose the plus sign, and
  set $t=+\frac{1}{4}$.
Using that the function $\theta_3$ is even, we are looking for $\rho,\tau$ satisfying
\begin{equation*}
a=\frac{\theta_1(\frac{\pi}{2}-\pi\rho)}{\theta_1(\pi\rho)}=\frac{\theta_2(\pi\rho)}{\theta_1(\pi\rho)},\qquad 
b=\frac{\theta_3(\pi(\rho+\frac{1}{4}))}{\theta_3(\pi(\rho-\frac{1}{4}))}=\frac{\theta_3(\pi(\rho+\frac{1}{4}))}{\theta_4(\pi(\rho-\frac{1}{4}))},
\end{equation*}
where in both second equalities, we used~\cite[Eq. (1.3.2,1.3.4)]{Lawden}.
Let us now express these functions using Jacobi trigonometric functions
$\sc,\dn$ and their inverses $\cs,\nd$, see~\cite[Chap. 2]{Lawden} for details.
Recall the following relations between the parameters of Jacobi's trigonometric
and theta functions: $k=\frac{\theta_2(0)^2}{\theta_3(0)^2},\
k'=\sqrt{1-k^2}=\frac{\theta_4(0)^2}{\theta_3(0)^2},\
K=\frac{\pi}{2}\theta_3(0)^2,\ iK'=\tau K$, and let us write $\tilde{\rho}=2K
\rho$. We have, see~\cite[2.1.1-2.1.3]{Lawden},
\begin{equation*}
\sc(\tilde{\rho}|k)=
\frac{\theta_3(0|\tau)}{\theta_4(0|\tau)}
\frac{\theta_1(\pi\rho|\tau)}{\theta_2(\pi\rho|\tau)},\qquad 
\dn(\tilde{\rho}|k)=
\frac{\theta_4(0|\tau)}{\theta_3(0|\tau)}
\frac{\theta_3(\pi\rho|\tau)}{\theta_4(\pi\rho|\tau)}.
\end{equation*}
When working with Jacobi's trigonometric functions, the elliptic modulus $k$ (or
$k'$) is considered as given; then $q=e^{i\pi\tau}$ can be derived from $k$ and
$k'$, see~\cite[2.1.12,\ 2.1.23]{Lawden}. As $k$ decreases from $1$ to $0$ (or
$k'$ increases from $0$ to $1$), 
$\tau$ goes from $i0$ to $i\infty$. Again, whenever no confusion occurs, we
remove $k$, resp.\ $\tau$, from the argument of the Jacobi trigonometric
functions, resp.\ theta functions. 
As a consequence of the above discussion, we are looking for $\rho$ and $k$, or equivalently $k'$, satisfying 
\begin{align}
a&=(k')^{-\frac{1}{2}}\cs(\tilde{\rho})
\label{equ:a}\\
b&=(k')^{-\frac{1}{2}}\dn\Bigl(\tilde{\rho}+\frac{K}{2}\Bigr)
\label{equ:b}.
\end{align}
We first use the addition formula~\cite[22.8.17]{NIST:DLMF}, in
Equation~\eqref{equ:b}, divide numerator and denominator by
$\sn(\tilde{\rho})\sn(\frac{K}{2})$, giving 
\begin{equation}
    b=(k')^{\frac{-1}{2}}
    \frac{
      \cs(\tilde{\rho})\dn(\frac{K}{2})-\dn(\tilde{\rho})\cs(\frac{K}{2})
    }{
      \cs(\tilde{\rho})\dn(\tilde{\rho})-\dn(\frac{K}{2})\cs(\frac{K}{2})
    }.
\end{equation}
We then, use the
identities~\cite[22.6.1-22.6.2]{NIST:DLMF}
to express $\dn$ using the function $\cs$
\begin{equation*}
    \dn(\tilde{\rho})=\sqrt{\frac{k'^2+\cs^2(\tilde{\rho})}{1+\cs^2(\tilde{\rho})}},
\end{equation*}
and the special values
$\dn(\frac{K}{2})=\cs(\frac{K}{2})=(k')^{\frac{1}{2}}$
found for example in~\cite[Table 22.5.2]{NIST:DLMF}
to obtain
\begin{equation}
  b
  =\frac{\cs(\tilde{\rho})-\dn(\tilde{\rho})}{\cs(\tilde{\rho})\dn(\tilde{\rho})-k'}
  =\frac{%
    \sqrt{k'^{-1}+a^2}-\sqrt{k'^{-1}+a^{-2}}
  }{%
    \sqrt{k'+a^2}-\sqrt{k'+ a^{-2}}.
  }
  \label{equ:b3}
\end{equation}
This formula, valid for $a\neq 1$ tends to $b=\sqrt{k'}$ for $a\to 1$, and can thus be extended in the case $a=1$. 
For a fixed value of $a$, the formula for $b$ in Equation~\ref{equ:b3} is a
  continuous, increasing function of $k'$, ranging from 0 to 1 as $k'$ varies in
  $(0,1)$, see the plots on the right of
  Figure~\ref{fig:rho_b_function_of_kprime_for_fixed_a}.
Thus, for every $a>0$, there exists a unique $k'$ solving this equation. Let us
fix this $k'$,
then we are looking for $\tilde{\rho}$ such that 
$a=(k')^{-\frac{1}{2}}\cs(\tilde{\rho})=(k')^{\frac{1}{2}}\cs(2K\rho)$. The
right hand side, seen as a function of $\rho$, is monotone and takes all
values from $0$ to
$\infty$ so that, for this fixed $k'$, such a $\rho$ exists and is unique,
see the plots on the left of
  Figure~\ref{fig:rho_b_function_of_kprime_for_fixed_a} to see how $\rho$ varies
  with $k'$ for different values of $a$.
Hence, we have
proved that $\rho$ and $k'$ (thus $\tau$) are determined uniquely from $a$ and $b$
if we want the two models to give the same face weights.
\end{proof}

\begin{figure}
  \centering
  \includegraphics[width=6cm]{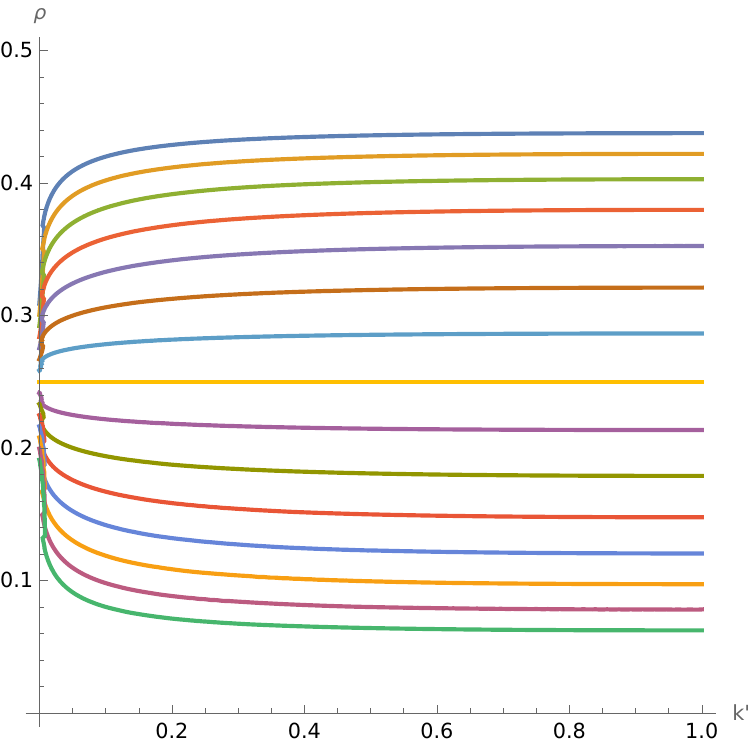}
  \hfill
  \includegraphics[width=6cm]{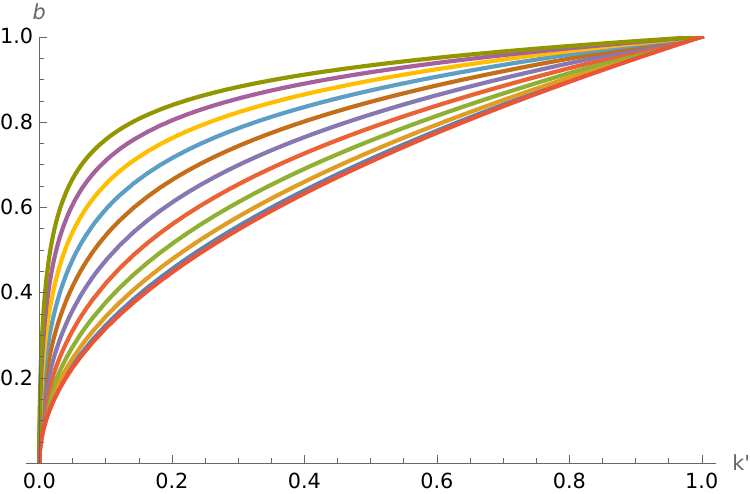}
  \caption{Left: the plot of $\rho$ as a function of $k'=\sqrt{1-k^2}$ satisfying
    Equation~\eqref{equ:a}, for various values of $a$ (of the form $2^{j/3}$,
    with $-7\leq j \leq 7$). Smaller values of $a$ correspond to higher curves.
    Right: the plot of $b$ as a function of $k'$ for several values of $a$ (of
    the form $2^{j/3}$, $0\leq j \leq 10$). The plot for $a<1$ is the same as for
    $1/a$.
  }
  \label{fig:rho_b_function_of_kprime_for_fixed_a}
\end{figure}

\begin{rem}\leavevmode
\begin{itemize}
 \item Assuming that the train-track angles are given by
   $\alpha,\beta,\gamma,\delta$ (independently of $j$) and supposing that
   $\beta-\alpha=\delta-\gamma=\frac{1}{2}$ imply that Fock's weights have a
   period of size $2\times 2$. As a consequence, in the genus 1 case, for
   general $2\times 2$ periodic weights, on top of the angle parameter $\rho$
   and the modular parameter $\tau$, we have one additional parameter
   $t\in\RR/\ZZ$. This parameter is fixed to the value $+\frac{1}{4}$ in the
   case of the biased $2\times 2$ periodic weights of~\cite{BorodinDuits}. 
 \item Using the notation of the proof of Proposition~\ref{prop:gauge_Borodin_Duits}, for every $\rho\in(0,\frac{1}{2})$, setting $k=0$ (or equivalently $k'=1$), amounting to considering the limit $\tau\rightarrow i\RR_+^*$, we obtain $b=1, a=\cot(\pi\rho)\in(0,\infty)$. The genus 1 case degenerates to a genus 0 case, and we recover a parameterization of the 1-periodic weights of~\cite{ChhitaJohanssonYoung}. 
 \item Returning to the proof of Proposition~\ref{prop:gauge_Borodin_Duits}, for
   every $k'\in(0,1)$, choosing
   $a=1$ corresponds to taking $\rho=\frac{1}{4}$, and according to
   Equation~\eqref{equ:b3}, $b=\sqrt{k'}$.
   We recover a parameterization of the 2-periodic weights of~\cite{ChhitaYoung}.
\end{itemize}
\end{rem}

\subsection{Space of parameters}\label{sec:space_parameters}

The special form of Fock's weights, see Equation~\eqref{eq:Focks_weight}, may seem to be restrictive; we explain here that this is not the case. More precisely, we prove that given a dimer model on the Aztec diamond with any positive weight function, it is gauge equivalent to a dimer model with Fock's weights. Recall from Section~\ref{sec:Aztec_dimer} that the two dimer models then yield the same dimer Boltzmann measure. 

\begin{prop}\label{prop:reconstruction}
For any $n\in\mathbb{N}^*$, and any choice of positive edge weights $\nu$ on
$\Az_n$, there exists an M-curve $\Sigma$ of genus $g$, a parameter
$t\in(\mathbb{R}/\mathbb{Z})^g$, and angles $\alphab,\betab,\gammab,\deltab\in A_0$ assigned to oriented train-tracks of $\Az_n$ satisfying $\alphab<\gammab<\betab<\deltab$ on $A_0$, such that 
the dimer models with Fock's weights and weight function $\nu$ are gauge equivalent. 
\end{prop}

\begin{proof}
Consider for a moment $\As_n$ as a subgraph of the
infinite square lattice $\mathbb{Z}^2$, and extend the edge weights $\nu$ in a
periodic system of weights $\bar{\nu}$ for the whole square lattice.
Following~\cite{KOS}, we can construct the spectral curve $\C$ associated to the periodic weights $\bar{\nu}$ as the set of zeros of the characteristic polynomial, see~\cite[Sections 3.1.3, 3.2.3]{KOS} for definitions. This algebraic curve is a Harnack curve. Moreover, if a vertex is distinguished, there is a natural associated standard divisor, where a \emph{standard divisor} corresponds to a collection of $g$ points on each oval of $\C$ if $\C$ has genus $g$~\cite{KO:Harnack}; this defines the spectral data of the model. The dimer spectral theorem by Kenyon and Okounkov implies that that $\C$ together with its standard
divisor (the spectral data) characterize the periodic weights up to gauge
transformation. By Remark~\cite[50.2]{BCdT:genusg}, which relies on~\cite[Theorem 49]{BCdT:genusg} and~\cite[Theorem 7.3]{GK}, it follows that there is a $t\in(\mathbb{R}/\mathbb{Z})^g$ and a periodic assignment of angle parameters to the train-tracks, satisfying the cyclic order condition, such that the corresponding dimer model with periodic Fock's weights is gauge equivalent to the original one. When restricting back to the subgraph $\As_n$, we get weights of the form~\eqref{eq:Focks_weight} which are gauge-equivalent to the initial weights $\nu$.
\end{proof}

\begin{rem}\leavevmode
\begin{itemize}
\item In Section~\ref{sec:example_0_1}, we prove two explicit realizations of
  Proposition~\ref{prop:reconstruction} in the case of Stanley's weights and of
  the biased $2\times 2$ periodic weights. In general, making the content of
  Proposition~\ref{prop:reconstruction} explicit is not easy because the argument underlying
  the proof is a general parameterization theorem. The idea to proceed would be
  to compute the associated characteristic polynomial using the weight function
  $\nu$, and the associated spectral curve, Newton polygon and amoeba. The genus
  $g$ is given by the number of holes in the amoeba, and from the tentacles one
  can recover the angles of the train-tracks. However the parameter $t$ is
  encoded by a standard divisor on the curve, which is some additional
  information that needs to be given. 
\item Note that, generically, holes in the amoeba are in correspondence with
  integer points in the interior of the Newton polygon; a lower number of holes
  means that there are isolated singularities on the curve. As a consequence,
  generically, the genus $g$ of Proposition~\ref{prop:reconstruction} increases
  with $n$, the size of the Aztec diamond. 
\end{itemize}
\end{rem}

\section{Explicit formula for the inverse Kasteleyn matrix}\label{sec:inverse_Kasteleyn}

The setting is that of Section~\ref{sec:Fock}: we consider an M-curve $\Sigma$ of genus $g$, the associated Riemann theta function $\theta$ and prime form $E$, and fix a parameter $t\in(\RR/\ZZ)^g$; we suppose that oriented train-tracks are assigned angles in $A_0$ satisfying condition~\eqref{equ:angle_condition}, and consider the associated discrete Abel map $\mapd$. In this section we state and prove one of the main results of this paper, Theorem~\ref{thm:Kinv}, consisting of an explicit expression for the inverse of the Kasteleyn matrix with Fock's weights. Recalling Theorem~\ref{thm:measure_Kenyon}, an immediate consequence of Theorem~\ref{thm:Kinv} is an explicit formula for the dimer Boltzmann measure in the very general framework of Fock's weights. As explained in the introduction, Theorem~\ref{thm:Kinv} aims at, in some sense, closing a long history of explicit expressions for the inverse Kasteleyn matrix of the Aztec diamond; refer to Section~\ref{sec:intro} for historical background. 

This section is organised as follows: in Section~\ref{sec:kernel}, we give some prerequisites, that is, the definition of the forms $g$ in the kernel of the Kasteleyn matrix $\Ks$~\cite{BCdT:genusg}, and Fay's identity~\cite{Fay}; then in Section~\ref{sec:K_inv} we state and prove Theorem~\ref{thm:Kinv}. 

\subsection{Prerequisites}\label{sec:kernel}

\paragraph{Forms in the kernel of $\Ks$.} We need the following ingredient from~\cite{BCdT:genusg}, namely forms in the kernel of the Kasteleyn matrix $\Ks$, defined as follows. Every edge of $\Azdiam_n$ consists of a dual vertex $\fs$, and a white or black vertex of $\Az_n$. For every edge of $\Azdiam_n$, and every $u\in\Sigma$, define:
\begin{align}\label{eq:form_g}
g_{\fs,\ws}(u)&=g_{\ws,\fs}(u)^{-1}=\frac{\theta(t+u+\mapd(\ws))}{E(\beta,u)}\\
g_{\bs,\fs}(u)&=g_{\fs,\bs}(u)^{-1}=\frac{\theta(-t+u-\mapd(\bs))}{E(\alpha,u)},
\end{align}
where $\alpha$, resp. $\beta$, is the angle of the oriented train-track crossing the edge $\ws\fs$, resp. $\bs\fs$, see Figure~\ref{fig:quad}. 
When $\xs,\ys$ are two vertices of $\Azdiam_n$, consider a path $\xs=\xs_1,\dots,\xs_n=\ys$ of $\Azdiam_n$, and set 
\begin{equation*}
g_{\xs,\ys}(u)=\prod_{j=1}^{n-1} g_{\xs_j,\xs_{j+1}}(u).
\end{equation*}
This quantity is well defined, \emph{i.e.}, independent of the choice of path in $\Azdiam_n$ from $\xs$ to $\ys$.

By~\cite[Lemma 33]{BCdT:genusg}, see also~\cite{Fock} we know that the forms $g$ are in the kernel of $\Ks$ for vertices that are not on the boundary;
written explicitly we have that, for every $u\in\Sigma$, for every vertex $\xs$ of~$\Azdiam_n$, and every white vertex in the bulk of $\As_n$,
\begin{equation}\label{equ:kernel_g}
\sum_{\bs\sim\ws}\Ks_{\ws,\bs}g_{\bs,\xs}(u)=0.  
\end{equation}

\paragraph{Fays' identity.} In the course of the proof of Theorem~\ref{thm:Kinv}, we will need
the following variant of Fay's identity~\cite[Proposition 2.10]{Fay} to expand the product $\Ks_{\ws,\bs}g_{\bs,\ws}(u)$, see also~\cite[Equation (9)]{BCdT:genusg},
\begin{align}
\Ks_{\ws,\bs}g_{\bs,\ws}(u)&=
\frac{E(\alpha,\beta)}{\theta(t+\mapd(\fs))\theta(t+\mapd(\fs'))}
\frac{\theta(t+u+\mapd(\ws))\theta(-t+u-\mapd(\bs))}{E(\alpha,u)E(\beta,u)}\nonumber\\
&=\omega_{\beta-\alpha}+\sum_{\ell=1}^g 
\Bigl(\frac{\partial \theta}{\partial z_\ell}(t+\mapd(\fs))-\frac{\partial
\theta}{\partial z_\ell}(t+\mapd(\fs'))\Bigr)\omega_j,\label{equ:Fay}
\end{align}
where $\omega_{\beta-\alpha}=\ud_u \log\frac{E(u,\alpha)}{E(u,\beta)}$ is the unique meromorphic 1-form with 0 integral along $A$-cycles, and two simple poles at $\beta$, resp. $\alpha$, with residue $1$, resp. $-1$.

\subsection{Explicit formula for the inverse Kasteleyn operator}\label{sec:K_inv}

We consider the Kasteleyn matrix $\Ks$ with Fock's weights, see Equation~\eqref{eq:Focks_weight}. 
In order to state our explicit formula for the inverse matrix, we first define contours of integration. By assumption, the angles assigned to oriented train-tracks satisfy the cyclic condition~\eqref{equ:angle_condition}:
$
\alphab<\gammab<\betab<\deltab.
$
This implies that the real component $A_0$ of $\Sigma$ can naturally be split into four disjoint connected subsets containing all of the angles of one type and none of the other types. We let $\C_1$ be a trivial contour on $\Sigma$, oriented counterclockwise, containing in its interior all of
the angles $\gammab=(\gamma_j)_{j=1}^n$ and none of the angles $\alphab,\betab,\deltab$.

Similarly $\C_2$ is a trivial contour on $\Sigma$, oriented counterclockwise, containing in its interior all of the angles $\alphab=(\alpha_j)_{j=1}^n$ and none of the others.

\begin{thm}\label{thm:Kinv}
For every pair $(\bs,\ws)$ of black and white vertices of $\Az_n$, the coefficient $(\bs,\ws)$ of the inverse Kasteleyn matrix is explicitly given by 
\begin{multline}\label{eq:Kinv}
\Ks^{-1}_{\bs,\ws}=\frac{1}{(2\pi i)^2}\frac{1}{\theta(p)}\int_{\C_2}\int_{\C_1}\frac{\theta(p+(v-u))}{E(u,v)}g_{\bs,0}(u)g_{0,\ws}(v)\prod_{j=1}^n \frac{E(\beta_j,u)}{E(\delta_j,u)}\frac{E(\delta_j,v)}{E(\beta_j,v)}+\\
-\II_{\{\bs \text{ right of }\ws\}}\frac{1}{2\pi i}\int_{\C_2}g_{\bs,\ws}(v),
\end{multline}
where $\C_1$, resp. $\C_2$, is the closed contour defined above used to integrate over $u$, resp.~$v$, and 
$
p=\sum\limits_{j=1}^n(\delta_j-\beta_j)-t-\mapd(0)
$.
\end{thm}

Before turning to the proof, let us make a few remarks.
\begin{rem}\leavevmode
\begin{itemize}
\item The integrand, seen as function of $u$, resp. $v$, is a meromorphic $1$-form. Indeed, looking at terms involving prime forms and recalling that the prime form is
a $(-\frac{1}{2},-\frac{1}{2})$ form, we have that the integrand is a $1$-form. Moreover, an explicit computation shows that this $1$-form has no period when $u$, resp. $v$, is translated by a horizontal/vertical period of the lattice $\Lambda$, implying that it is meromorphic.

\item In the double integral of~\eqref{eq:Kinv} the factor
  $g_{\bs,0}(u)\prod_j\frac{E(\beta_j,u)}{E(\delta_j,u)}$ has
  poles at every $\gamma_j$ on the left of $\bs$ and every $\delta_j$ on the
  right of $\bs$, and zeros at every $\alpha_j$ below $\bs$ and every $\beta_j$
  above $\bs$; similarly for
    $g_{0,\ws}(v)\prod_j\frac{E(\delta_j,v)}{E(\beta_j,v)}$, where the role of
  poles and zeros are exchanged. As a consequence, we can deform the contours of
  integration without changing the value of the integral if we do not cross poles.
  For example, one could move $\C_1$ into another contour $\C_1'$ to also include the $\beta_j$'s in its interior,
  and replace $\C_2$ by $\C_2'$
  depicted on Figure~\ref{fig:contours_deform},
  which is now oriented clockwise, and contains on its right the points from
  $\gammab$ and $\betab$ (but not those from $\alphab$). The
  relative position of $\C_1$ and $\C_2$ is important because of the presence of
  $E(u,v)$ in the denominator of the integrand. The contribution of the residue at
  $u=v$ is exactly given by the single contour integral in front of the indicator
  function. We can therefore absorb this second term inside the double contour
  integral by indicating that when $\bs$ is on the right of $\ws$, we require that
  $\C_1'$ is \emph{inside} $\C_2'$ (instead of outside).
  
  \begin{figure}[htpb]
  \centering
  \hfill
  \includegraphics[width=0.3\linewidth]{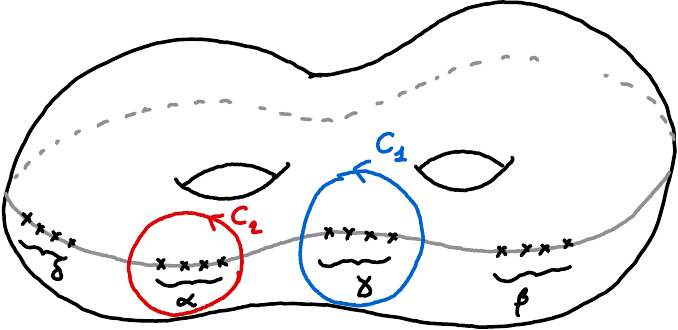}\hfill
  \includegraphics[width=0.3\linewidth]{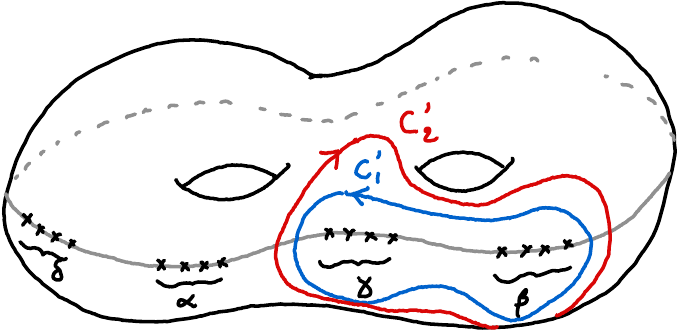}\hfill\mbox{}
  \caption{A possible deformation of integration contours for
  Formula~\eqref{eq:Kinv}.}
  \label{fig:contours_deform}
\end{figure}
  
\item The point $0=(0,0)$ seems to play a particular role in the formula: in the
  definition of $p$ and in the arguments of the functions $g$. This is actually
  not the case, one
  could express $p$ and the product $g$ times the terms involving the prim forms
  using another reference point and the geometry of the Aztec diamond.
\item Examples of Theorem~\ref{thm:Kinv} in specific cases are given after the proof. 
\end{itemize}
\end{rem}

\begin{proof}
Although our setting is much more general than the paper~\cite{ChhitaJohanssonYoung}, our inspiration for this proof and choice of notation is inspired by the latter. Consider two white vertices $\ws,\ws'$; the proof consists in showing that 
\[
(\Ks\Ks^{-1})_{\ws,\ws'}=\sum_{\bs\sim\ws} \Ks_{\ws,\bs}\Ks^{-1}_{\bs,\ws'}=\II_{\{\ws=\ws'\}}.
\]
Let us denote by $f^1$, resp. $f^2$, the first, resp. second term of Equation~\eqref{eq:Kinv}:
\begin{align}
f^1_{\bs,\ws}&=\frac{1}{(2\pi i)^2}\frac{1}{\theta(p)}\int_{\C_2}\int_{\C_1}\frac{\theta(p+(v-u))}{E(u,v)}g_{\bs,0}(u)g_{0,\ws}(v)\prod_{j=1}^n \frac{E(\beta_j,u)}{E(\delta_j,u)}\frac{E(\delta_j,v)}{E(\beta_j,v)}\label{equ:f_1}\\
f^2_{\bs,\ws}&=\II_{\{\bs \text{ right of }\ws\}}\frac{1}{2\pi i}\int_{\C_2}g_{\bs,\ws}(v).\label{equ:f_2}
\end{align}
The proof requires the following 5 steps. If $\ws,\ws'$ are such that 
\begin{enumerate}
 \item $\ws_x\neq 2n$ and $\ws_x\neq \ws_x'$, or $\ws_x=\ws_x'=2n$, then $(\Ks f^2)_{\ws,\ws'}=0$. 
 \item $\ws_x\neq 2n$ and $\ws_x=\ws_x'$, then $(\Ks f^2)_{\ws,\ws'}=-\II_{\{\ws=\ws'\}}.$ 
 \item $\ws_x\neq 2n$, then $(\Ks f^1)_{\ws,\ws'}=0$. 
 \item $\ws_x=2n$ and $\ws_x\neq \ws_x'$, then $(\Ks f^1)_{\ws,\ws'}=(\Ks f^2)_{\ws,\ws'}$.
 \item $\ws_x=2n$ and $\ws_x=\ws_x'$, then $(\Ks f^1)_{\ws,\ws'}=\II_{\{\ws=\ws'\}}.$
\end{enumerate}
Before proving each of the steps, let us show that they indeed allow to end the proof. 
If $\ws$ is such that:
\begin{itemize}
\item $\ws_x\in\{0,2,\dots,2n-2\}$, and $\ws_x\neq \ws_x'$, then by Point 1. we have $(\Ks f^2)_{\ws,\ws'}=0$, and by Point 3. $(\Ks f^1)_{\ws,\ws'}=0$, implying that $(\Ks\Ks^{-1})_{\ws,\ws'}=0$. 
\item $\ws_x\in\{0,2,\dots,2n-2\}$, and $\ws_x=\ws'_x$, then by Point 2. we have $(\Ks f^2)_{\ws,\ws'}=-\II_{\{\ws=\ws'\}}$, and by Point 3. $(\Ks f^1)_{\ws,\ws'}=0$, implying that $(\Ks\Ks^{-1})_{\ws,\ws'}=\II_{\{\ws=\ws'\}}$.
\item $\ws_x=2n$ and $\ws_x\neq \ws_x'$, then by Point 4, $(\Ks f^1)_{\ws,\ws'}=(\Ks f^2)_{\ws,\ws'}$, implying that $(\Ks\Ks^{-1})_{\ws,\ws'}=0$. 
\item $\ws_x=2n$ and $\ws_x=\ws_x'$ then by Point 1. we have $(\Ks f^2)_{\ws,\ws'}=0$, and by Point 5. $(\Ks f^1)_{\ws,\ws'}=\II_{\{\ws=\ws'\}}$, implying that $(\Ks\Ks^{-1})_{\ws,\ws'}=\II_{\{\ws=\ws'\}}$. 
\end{itemize}
We now turn to the proof of Points 1. to 5. 

1. When $\ws_x'>\ws_x$, all $\bs$-neighbors  of $\ws$ are on the left of $\ws'$, so that all terms of type $f_2$ are equal to 0 and $(\Ks f^2)_{\ws,\ws'}$ is trivially equal to 0 (note that there are four neighbors if $\ws_x\neq 0$ and two neighbors if $\ws_x=0$). If $\ws_x=\ws_x'=2n$, the same holds since then the two $\bs$-neighbors of $\ws$ are on the left of $\ws'$ and there are no right neighbors. 
When $\ws_x'<\ws_x$ and $\ws_x\neq 2n$, the four $\bs$-neighbors of $\ws$ are now on the right, and there are four contributions of type $f_2$. We have
\begin{align*}
(\Ks f^2)_{\ws,\ws'}&=\frac{1}{2\pi i}\sum_{\bs\sim\ws} \int_{\C_2}\Ks_{\ws,\bs}\, g_{\bs,\ws'}.
\end{align*}
Since the contour $\C_2$ in independent of $\bs,\ws'$, the sum can be moved inside the integral, and we have that $(\Ks f^2)_{\ws,\ws'}$ is equal to $0$ by Equation~\eqref{equ:kernel_g}.

2. Suppose that $\ws_x=\ws_x'$ and $\ws_x\neq 2n$. Let $\bs_1,\bs_2$ be the two neighbors of $\ws$ on the right, from bottom to top, and $\fs_1,\fs_2,\fs_3$ be the dual vertices as in Figure~\ref{fig:fig_proof_Kinv_1}.  

\begin{figure}[H]
  \centering
  \begin{overpic}[width=6cm]{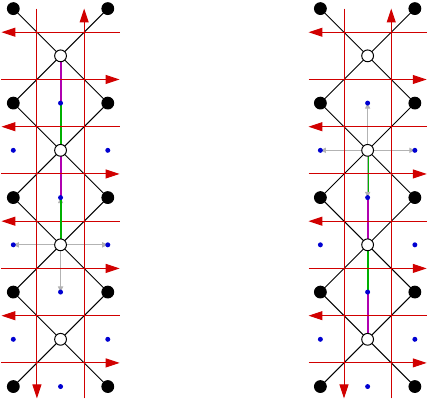}
  \put(15,28){\scriptsize $\alpha$} 
  \put(19,33){\scriptsize $\gamma$} 
  \put(15,42){\scriptsize $\beta$} 
  \put(6,32){\scriptsize $\delta$} 
  \put(9,36){\scriptsize $\ws$}
  \put(8,80){\scriptsize $\ws'$}
  \put(27,25){\scriptsize $\bs_1$}
  \put(27,48){\scriptsize $\bs_2$}
  \put(12,21){\scriptsize $\fs_1$} 
  \put(27,36){\scriptsize $\fs_2$} 
  \put(10,49){\scriptsize $\fs_3$}
  \put(87,50){\scriptsize $\alpha$} 
  \put(91,57){\scriptsize $\gamma$} 
  \put(87,64){\scriptsize $\beta$} 
  \put(77,56){\scriptsize $\delta$} 
  \put(80,60){\scriptsize $\ws$}
  \put(80,15){\scriptsize $\ws'$}
  \end{overpic}
\caption{Zeros and poles of $g_{\ws,\ws'}(v)$ when $\ws_x=\ws_x'$, $\ws_x\neq 2n$. Left: $\ws'$ is above $\ws$. Right: $\ws'$ is below $\ws$. Angles corresponding to zeros, resp. poles, are pictured in green, resp. magenta.}\label{fig:fig_proof_Kinv_1}  
\end{figure}

Using that the contour of integration $\C_2$ is independent of $\bs_1,\bs_2,\ws'$, and using the product structure of the meromorphic form $g$, we obtain
\begin{align*}
(\Ks f^2)_{\ws,\ws'}=\frac{1}{2\pi i}\int_{\C_2}\Bigl[\sum_{i=1}^2 \Ks_{\ws,\bs_i}g_{\bs_i,\ws}(v)\Bigr] g_{\ws,\ws'}(v).
\end{align*}
To simplify notation, we write the four angles around $\ws$ as $\alpha,\gamma,\beta,\delta$ going cclw starting from the bottom (omitting indices). From Fay's identity~\eqref{equ:Fay} we have
\begin{align}\label{equ:use_Fay_1}
\sum_{i=1}^2 \Ks_{\ws,\bs_i}g_{\bs_i,\ws}(v)&=\omega_{\gamma-\alpha}+\sum_{\ell=1}^g \Bigl(\frac{\partial \log\theta}{\partial z_\ell}(t+\mapd(\fs_3))-\frac{\partial \log\theta}{\partial z_\ell}(t+\mapd(\fs_2))\Bigr)\omega_\ell \nonumber \\
&\quad \quad \quad \quad+\omega_{\beta-\gamma}+\sum_{\ell=1}^g \Bigl(
\frac{\partial \log\theta}{\partial z_\ell}(t+\mapd(\fs_2))-\frac{\partial \log\theta}{\partial z_\ell}(t+\mapd(\fs_1))\Bigr)\omega_\ell\nonumber\\
&=\omega_{\beta-\alpha} +\sum_{\ell=1}^g \Bigl(\frac{\partial \log\theta}{\partial z_\ell}(t+\mapd(\fs_3))-\frac{\partial \log\theta}{\partial z_\ell}(t+\mapd(\fs_1))\Bigr)\omega_\ell.
\end{align}
As a consequence,
{\small 
\begin{align*}
(\Ks f^2)_{\ws,\ws'} &=\frac{1}{2\pi i}\Bigl(\int_{\C_2} g_{\ws,\ws'}(v)\,\omega_{\beta-\alpha}+
\sum_{\ell=1}^g \Bigl(\frac{\partial \log\theta}{\partial z_\ell}(t+\mapd(\fs_3))-\frac{\partial \log\theta}{\partial z_\ell}(t+\mapd(\fs_1))\Bigr) \int_{\C_2}g_{\ws,\ws'}(v)\,\omega_\ell\Bigr).
\end{align*}
}

Returning to the definition of the form $g$, we know that 
the zeros, resp. poles, of $g_{\ws,\ws'}(v)$ in $\C_2$ are the angles between $\ws$ and $\ws'$ of type
\begin{itemize}
 \item $(\beta_j)$ (including $\beta$) , resp. $(\alpha_j)$ (excluding $\alpha$), when $\ws'$ is above $\ws$, see Figure~\ref{fig:fig_proof_Kinv_1} (left),
 \item $(\alpha_j)$ (including $\alpha)$, resp. $(\beta_j)$ (excluding $\beta$), when $\ws'$ is below $\ws$, see Figure~\ref{fig:fig_proof_Kinv_1} (right).
\end{itemize}
Let us first consider the term $\frac{1}{2\pi i}\int_{\C_2} g_{\ws,\ws'}(v)\, \omega_{\beta-\alpha}\,$.
\begin{itemize}
 \item when $\ws'$ is above $\ws$, the pole at $\beta$ of $\omega_{\beta-\alpha}$ is cancelled by the zero at $\beta$ of $g_{\ws,\ws'}(v)$, and the integrand $\omega_{\beta-\alpha}\,g_{\ws,\ws'}(v)$ has as poles a subset of $\{\alpha_1,\dots,\alpha_n\}$. Since the contour $\C_2$ contains all the angles $\{\alpha_1,\dots,\alpha_n\}$, we obtain 0. 
 \item when $\ws'$ is below $\ws$, the pole at $\alpha$ of $\omega_{\beta-\alpha}$ is cancelled by the zero at $\alpha$ of $g_{\ws,\ws'}(v)$, and the integrand has as  poles a subset of $\{\beta_1,\dots,\beta_n\}$. Since the contour $\C_2$ contains none of these poles, we also obtain 0. 
 \item when $\ws=\ws'$, then $g_{\ws,\ws}(v)=1$, and the contour $\C_2$ contains the pole $\alpha$ of $\omega_{\beta-\alpha}$, which yields that $\frac{1}{2\pi i}\int_{\C_2} \omega_{\beta-\alpha}=-1$. 
\end{itemize}
Let us now consider the term $\int_{\C_2}g_{\ws,\ws'}(v)\,\omega_\ell$. The contour
$\C_2$ either contains all poles of $g_{\ws,\ws'}$ (when $\ws'$ is above $\ws$)
or none of them (if $\ws'$ is below or equal to $\ws$), and the integral
of the form $\omega_\ell$ around any trivial closed contour is 0, so that the second term is always equal to 0. Summarizing, we have proved Point 2.
\begin{equation}\label{equ:f2_summary}
(\Ks f^2)_{\ws,\ws'}=\frac{1}{2\pi i}\int_{\C_2}\Bigl[\sum_{i=1}^2 \Ks_{\ws,\bs_i}g_{\bs_i,\ws'}(v)\Bigr]=-\II_{\{\ws=\ws'\}}.
\end{equation}

3. Using that the contours of integration $\C_1,\C_2$ are independent of $\bs,\ws'$, we have, for every pair of white vertices $(\ws,\ws')$, 
\begin{align}
&(\Ks f^1)_{\ws,\ws'}=\nonumber \\
&=\frac{1}{(2\pi i)^2}\frac{1}{\theta(p)}\sum_{\bs\sim\ws} \Ks_{\ws,\bs}\int_{\C_2}\int_{\C_1}\frac{\theta(p+(v-u))}{E(u,v)}g_{\bs,0}(u)g_{0,\ws'}(v)\prod_{j=1}^n \frac{E(\beta_j,u)}{E(\delta_j,u)}\frac{E(\delta_j,v)}{E(\beta_j,v)}\nonumber \\
&=\frac{1}{(2\pi i)^2}\frac{1}{\theta(p)}\int_{\C_2}g_{0,\ws'}(v)\prod_{j=1}^n \frac{E(\delta_j,v)}{E(\beta_j,v)}\cdot \nonumber\\
&\hspace{2.5cm} \cdot \int_{\C_1}\frac{\theta(p+(v-u))}{E(u,v)}\prod_{j=1}^n \frac{E(\beta_j,u)}{E(\delta_j,u)}\Bigl[\sum_{\bs\sim\ws} \Ks_{\ws,\bs} g_{\bs,0}(u)\Bigr].\label{equ:split_integral}
\end{align}

As a consequence, when $\ws_x\notin \{0,2n\}$, this is equal to 0 by Equation~\eqref{equ:kernel_g}. When $\ws_x=0$, one can insert a column of black vertices on the left, and an angle $\delta_0$ (in the sector containing $\{\delta_1,\dots,\delta_n\}$). In this way, the white vertex $\ws$ has its four $\bs$-neighbors: $\bs_1,\bs_2$ present in the original Aztec diamond, and $\bs_3,\bs_4$ the two additional ones. Note that 
defining $f^1_{\bs_i,\ws}$ ($i\in\{3,4\}$) using Formula~\eqref{equ:f_1} gives 0 in this case. Indeed, 
$g_{\bs_i,0}(u)$ has as poles $\delta_0$ and a subset of the angles $(\beta_j)$, implying that the contour $\C_1$ (which contains all of the angles $\{\gamma_1,\dots,\gamma_n\}$) contains no pole of the integrand $g_{\bs_{i},0}(u)\frac{\theta(p+(v-u))}{E(u,v)}\prod_{j=1}^n \frac{E(u,\beta_j)}{E(u,\delta_j)}$, yielding 0. The proof can then be concluded using Equation~\eqref{equ:kernel_g} again. 

4. Suppose $\ws_x=2n$, and let $\bs_3, \bs_4$ be the two neighbors of $\ws$ on the left, from bottom to top, and $\fs_1,\fs_3,\fs_4$ be the dual vertices as in Figure~\ref{fig:fig_proof_Kinv_2}. 

\begin{figure}[H]
  \centering
  \begin{overpic}[width=6cm]{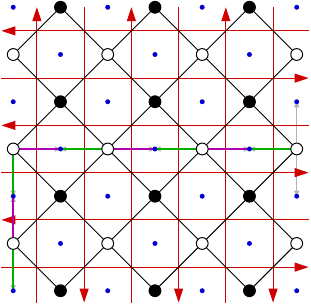}
  \put(98,43){\scriptsize $\alpha$} 
  \put(98,60){\scriptsize $\beta$} 
  \put(82,51){\scriptsize $\delta_n$} 
  \put(100,50){\scriptsize $\ws$}
 
  \put(73,35){\scriptsize $\bs_3$}
  \put(73,63){\scriptsize $\bs_4$}
  \put(95,29){\scriptsize $\fs_1$} 
  \put(74,50){\scriptsize $\fs_4$} 
  \put(95,67){\scriptsize $\fs_3$}

  \put(-5,12){\scriptsize $\alpha_1$} 
  \put(8,-3){\scriptsize $\gamma_1$} 
  \put(-5,28){\scriptsize $\beta_1$} 
  \put(26,-4){\scriptsize $\delta_1$} 
  \put(0,0){\scriptsize $0$}
  \end{overpic}
\caption{Zeros and poles of $g_{\ws,\ws'}(v)$ when $\ws_x=\ws_x'$, $\ws_x\neq 2n$. Left: $\ws'$ is above $\ws$. Right: $\ws'$ is below $\ws$. Angles corresponding to zeros, resp. poles, are pictured in green, resp. magenta.}\label{fig:fig_proof_Kinv_2}  
\end{figure}

To simplify notation, let us denote by 
$\alpha,\delta_n,\beta$ the angles around $\ws$ in clockwise order (omitting the indices for $\alpha,\beta$). Writing $(\Ks f^1)_{\ws,\ws'}$ as in Equation~\eqref{equ:split_integral}, we first consider the part containing the integral over $\C_1$. 
\begin{align}\label{equ:int_C1}
\int_{\C_1}\frac{\theta(p+(v-u))}{E(u,v)}\Bigl[ \sum_{i=3}^4 \Ks_{\ws,\bs_i}g_{\bs_i,\ws}(u)\Bigr]g_{\ws,0}(u)\prod_{j=1}^n \frac{E(\beta_j,u)}{E(\delta_j,u)}. 
\end{align}
Similarly to Equation~\eqref{equ:use_Fay_1}, using Fay's identity~\eqref{equ:Fay} gives
\begin{align*}
\sum_{i=3}^4 \Ks_{\ws,\bs_i}g_{\bs_i,\ws}(v)&=\omega_{\alpha-\beta} +\sum_{\ell=1}^g \Bigl(\frac{\partial \log\theta}{\partial z_\ell}(t+\mapd(\fs_1))-\frac{\partial \log\theta}{\partial z_\ell}(t+\mapd(\fs_3))\Bigr)\omega_\ell.
\end{align*}
Note that by definition of $g$, all poles of $g_{\bs_i,0}(u)$ are on $A_0$. Now, the term $g_{\ws,0}(u)$ contains
\begin{itemize}
\item as poles on $A_0$: all the angles $\gamma_1,\dots,\gamma_n$ and the angles of type $(\beta_j)$ from $\ws$ to the bottom boundary (and in particular not the angle $\beta$),
\item as zeros on $A_0$: the angles $\delta_1,\dots,\delta_{n}$ and the angles of type $(\alpha_j)$ from $\ws$ to the bottom boundary (and in particular the angle $\alpha$),
\end{itemize}
so the product $g_{\ws,0}(u)\,\omega_{\alpha-\beta}$ contains as poles on $A_0$: a subset of the angles of type $(\beta_j)$, no angle of type $(\alpha_j)$ (since the pole $\alpha$ of $\omega_{\alpha-\beta}$ is cancelled by the zero $\alpha$ of $g_{\ws,0}(u)$), all the angles $\gamma_1,\dots,\gamma_n$; note that the set of poles of $g_{\ws,0}(u)\omega_\ell$ has the same property. As a consequence when each of theses terms gets multiplied by $\frac{\theta(p+(v-u))}{E(u,v)}\prod_{j=1}^n \frac{E(\beta_j,u)}{E(\delta_j,u)}$; all the poles involving angles $(\beta_j)$ are cancelled, and moreover, the new poles involving $\delta_1,\dots,\delta_n$ are cancelled by the same zeros in $g_{\ws,0}(u)$. Summarizing, in the integrand
\[
\frac{\theta(p+(v-u))}{E(u,v)}\Bigl[ \sum_{i=3}^4 \Ks_{\ws,\bs_i}g_{\bs_i,\ws}(u)\Bigr]g_{\ws,0}(u)\prod_{j=1}^n \frac{E(\beta_j,u)}{E(\delta_j,u)},
\]
there remains as poles on $\Sigma$: all the angles $\gamma_1,\dots,\gamma_n$ and the pole at $u=v$. By definition, the contour of integration $\C_1$ contains all of the angles $\{\gamma_1,\dots,\gamma_n\}$, and the pole $v$ (which lives on $\C_2$) is outside of $\C_1$.
Since we are integrating a meromorphic form on a compact surface (implying that the sum of the residues is equal to 0), the integral~\eqref{equ:int_C1} is equal to $-2\pi i$ times the residue at $u=v$.
Noting that the residue at $u=v$ of $\frac{1}{E(u,v)}$ is equal to $-1$, we obtain:
\begin{multline*}
  \frac{1}{2\pi i }\frac{1}{\theta(p)}\int_{\C_1} \frac{\theta(p+(v-u))}{E(u,v)}\Bigl[ \sum_{i=3}^4 \Ks_{\ws,\bs_i}g_{\bs_i,\ws}(u)\Bigr]g_{\ws,0}(u)\prod_{j=1}^n \frac{E(\beta_j,u)}{E(\delta_j,u)}\\
  =\Bigl[ \sum_{i=3}^4 \Ks_{\ws,\bs_i}g_{\bs_i,\ws}(v)\Bigr]g_{\ws,0}(v)\prod_{j=1}^n \frac{E(\beta_j,v)}{E(\delta_j,v)}
  =\sum_{i=3}^4 \Ks_{\ws,\bs_i}g_{\bs_i,0}(v)
  \prod_{j=1}^n \frac{E(\beta_j,v)}{E(\delta_j,v)}.
\end{multline*}
Plugging this back into~\eqref{equ:split_integral} yields,
\begin{align}
  (\Ks f^1)_{\ws,\ws'}&=\frac{1}{2\pi i}\int_{\C_2}g_{0,\ws'}(v)\prod_{j=1}^n \frac{E(\delta_j,v)}{E(\beta_j,v)} \sum_{i=3}^4 \Ks_{\ws,\bs_i}g_{\bs_i,0}(v)\prod_{j=1}^n \frac{E(\beta_j,v)}{E(\delta_j,v)}\nonumber\\
                      &=\frac{1}{2\pi i}\int_{\C_2}\sum_{i=3}^4 \Ks_{\ws,\bs_i}g_{\bs_i,\ws'}(v).\label{equ:Case4to5}
\end{align}
Now, suppose moreover that $\ws_x\neq \ws_x'$, then the two $\bs$-neighbors of $\ws$ are on the right of~$\ws$, and we have that 
\begin{align*}
  (\Ks f^1)_{\ws,\ws'}=(\Ks f^2)_{\ws,\ws'}.
\end{align*}

5. Suppose that $\ws_x=2n$ and $\ws_x=\ws_x'$. Then, by Equation~\eqref{equ:Case4to5}, we have
\begin{align*}
(\Ks f^1)_{\ws,\ws'}=\frac{1}{2\pi i}\int_{\C_2}\sum_{i=3}^4 \Ks_{\ws,\bs_i}g_{\bs_i,\ws'}(v).
\end{align*}
Using the argument of Point 2. with black vertices on the left rather than on the right allows us to conclude that $(\Ks f^1)_{\ws,\ws'}=\II_{\{\ws=\ws'\}}$. \qedhere
\end{proof}

\begin{exm}
As an example of application, let us make explicit the case where angles in each family are constant, \emph{i.e.}, suppose that $\alphab\equiv \alpha,\betab=\beta,\gammab=\gamma,\deltab=\delta$, for some angles $\alpha,\beta,\gamma,\delta\in A_0$ satisfying the cyclic order $\alpha<\gamma<\beta<\delta$. Recalling the coordinate notation $\ws=(\ws_x,\ws_y)$, $\bs=(\bs_x,\bs_y)$, we obtain
\begin{multline} 
\Ks^{-1}_{\bs,\ws}=
\frac{1}{(2\pi i)^2}\frac{1}{\theta(p)}\int_{\C_2}\int_{\C_1}\frac{\theta(p+(v-u))}{E(u,v)}\cdot \\
\quad \cdot\ \theta(-t+u-\mapd(\bs))
\frac{%
  E(\alpha,u)^{\frac{\bs_y}{2}}
  E(\beta,u)^{n-\frac{\bs_y}{2}}
}{%
  E(\gamma,u)^{\frac{\bs_x+1}{2}}
  E(\delta,u)^{n-{\frac{\bs_x-1}{2}}}
}
\theta(t+v+\mapd(\ws))
\frac{E(\gamma,v)^{\frac{\ws_x}{2}}E(\delta,v)^{n-\frac{\ws_x}{2}}}
{E(\beta,v)^{n-\frac{\ws_y-1}{2}}E(\alpha,v)^{\frac{\ws_y+1}{2}}} 
\\
\quad -
\II_{\{\bs \text{ right of }\ws\}}\frac{1}{2\pi i}\int_{\C_2}
\theta(-t+v-\mapd(\bs))\theta(t+v+\mapd(\ws))\frac{E(\alpha,v)^{\frac{\bs_y-\ws_y-1}{2}}}{E(\beta,v)^{\frac{\bs_y-\ws_y+1}{2}}}\frac{E(\delta,v)^{\frac{\bs_x-\ws_x-1}{2}}}{E(\gamma,v)^{\frac{\bs_x-\ws_x+1}{2}}},
\label{eq:Kinv_homog}
\end{multline}
where $p= n(\delta-\beta)-t-\mapd(0)$, and 
\begin{equation*}
  \mapd(\bs)= \frac{\bs_y}{2}(\beta-\alpha) + \frac{\bs_x+1}{2}\gamma -
\frac{\bs_x-1}{2}\delta, \quad
  \mapd(\ws) = \frac{\ws_x}{2}(\gamma-\delta) - \frac{\ws_y+1}{2}\alpha +
  \frac{\ws_y-1}{2}\beta.
\end{equation*}

In particular, if $\bs$ is the south-west neighbor of $\ws$, meaning that
$(\ws_x,\ws_y)=(2i,2j+1)$ and $(\bs_x,\bs_y)=(2i-1,2j)$ for some $1\leq i\leq n$
and $0\leq j\leq n-1$, Formula~\eqref{eq:Kinv_homog} reduces further to
\begin{multline}
\Ks^{-1}_{\bs,\ws}=
\frac{1}{(2\pi
i)^2}\frac{1}{\theta(p)}\int_{\C_2}\int_{\C_1}\frac{\theta(p+(v-u))}{E(u,v)}\cdot
\theta(-t+u-\mapd(\bs))
\theta(t+v+\mapd(\ws))
\\
\cdot
\frac{%
  E(\alpha,u)^{j}
  E(\beta,u)^{n-j}
}{%
  E(\gamma,u)^{i}
  E(\delta,u)^{n-i+1}
}
\frac{%
  E(\gamma,v)^{i}
  E(\delta,v)^{n-i}
}{%
  E(\beta,v)^{n-j}
  E(\alpha,v)^{j+1}
} 
\label{eq:Kinv_homog_sw}
\end{multline}
where $\mapd(\bs)=j(\beta-\alpha)+i(\gamma-\delta)+\delta$ and
$\mapd(\ws)=j(\beta-\alpha)+i(\gamma-\delta)-\alpha$.
\end{exm}

\section{Partition function}\label{sec:partition}

The main result of this section is Theorem~\ref{thm:partition_function}
consisting of an induction formula for the partition function of the Aztec
diamond with Fock's weights, thus proving that the partition function can always
be expressed in product form;
this is quite a surprising fact since, a priori, it is defined as the
determinant of a matrix. As specific cases of the genus 0 case, we recover
Stanley's celebrated formula~\cite{Propp_talk, BYYangPhD} as well as the fact
that the number of dimer configurations of an Aztec diamond of size $n$ is equal
to $2^{\frac{n(n+1)}{2}}$~\cite{EKLP}. Our proof is inspired from one of the
inductive arguments used to establish the formula
$2^{\frac{n(n+1)}{2}}$~\cite{Propp_talk}; it goes through in this very general
setting because Fock's weights are invariant under spider moves, and
contraction/expansion of a degree 2
vertex~\cite{Fock,BCdT:elliptic,BCdT:genusg}. 

This section is organised as follows. In Section~\ref{sec:partition_function},
we state Theorem~\ref{thm:partition_function} giving the product formula for the
partition function, then we state and prove Corollary~\ref{cor:genus_0}
specifying Theorem~\ref{thm:partition_function} to the case of genus 0; finally,
in Corollary~\ref{cor:Stanley}, we explain how it allows to recover Stanley's
formula~\cite{Propp_talk,BYYangPhD}. Section~\ref{sec:proof_partition} contains the proof Theorem~\ref{thm:partition_function}. 

\subsection{Partition function formula}\label{sec:partition_function}

Consider the Aztec diamond $\Az_n$ with angles $\alphab,\betab,\gammab,\deltab\in A_0$ assigned to oriented train-tracks, satisfying the cyclic order $\alphab<\gammab<\betab<\deltab$. Consider the dimer model on $\Az_n$ with Fock's weight function defined in Equation~\eqref{eq:Focks_weight}. Suppose that the value of the discrete Abel map at the origin $0=(0,0)$ is $\mapd(0)=d$, for some $d\in \Pic^{0}(\Sigma)$.
Let us denote the corresponding partition function by 
\begin{align}\label{equ:def_partition}
Z_n(\alphab,\betab,\gammab,\deltab;d)=Z_n((\alpha_j)_{j=1}^n,(\beta_j)_{j=1}^n,(\gamma_j)_{j=1}^n,(\delta_j)_{j=1}^n;d)=\sum_{\Ms\in\M(\Az_n)} \prod_{\es\in \Ms}|\Ks_{\ws,\bs}|. 
\end{align}

Recall from Section~\eqref{sec:Aztec_dimer} that coordinates of dual vertices of $\Fs_n$ are either both odd or both even. This implies that the set $\Fs_n$ can naturally
split into:
\[
\odd_n,\quad\text{and }\ \even_n=\inteven_n\sqcup(\bry_n\setminus\corner_n)\sqcup\corner_n,
\]
where even vertices are furthermore split according to whether they are interior, boundary or corner vertices. 

Note that an odd face $\fs$ of $\Fs_n$ is surrounded by angles $\alpha_i,\beta_i,\gamma_j,\delta_j$ for some $i,j\in\{1,\dots,n\}$. In order to simplify notation, in Theorem~\ref{thm:partition_function} below we denote them generically by $\alpha,\beta,\gamma,\delta$, see Figure~\ref{fig:notation_partition}. We are now ready to state the main result of this section. 

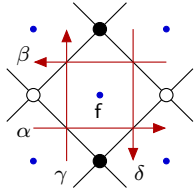
\begin{figure}[h!]
  \centering
  \begin{overpic}[width=2.5cm]{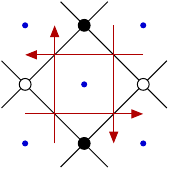}
  \put(6,28){\scriptsize $\alpha$} 
  \put(6,68){\scriptsize $\beta$} 
  \put(26,5){\scriptsize $\gamma$} 
  \put(68,5){\scriptsize $\delta$} 
  \put(47,38){\scriptsize $\fs$} 
  \end{overpic}
\caption{Generic notation $\alpha,\beta,\gamma,\delta$ for angles around a face $\fs$ corresponding to a dual vertex with both odd coordinates.}\label{fig:notation_partition}  
\end{figure}

\begin{thm}\label{thm:partition_function} 
For every $n\geq 1$, the partition function of the Aztec diamond $\Az_n$ with Fock's weights satisfies the following recurrence:
{\small 
\begin{align*}
Z_n(\alphab,\betab,\gammab,\deltab;d)&\cdot 
\prod_{f\in\odd_n}\frac{\theta(t+\mapd(\fs))}{\theta(t+\mapd(\fs)+\alpha+\beta-\gamma-\delta)}\prod_{\fs\in\bry_n}\theta(t+\mapd(\fs))=\\
=Z_{n-1}&((\alpha_j)_{j=1}^{n-1},(\beta_j)_{j=2}^n,(\gamma_j)_{j=1}^{n-1},(\delta_j)_{j=2}^n;d+\beta_1-\delta_1)\cdot
\Bigl[\prod_{j=1}^{n}\ |E(\alpha_j,\beta_j)E(\gamma_j,\delta_j)|\Bigr],
\end{align*}}
with the convention that $Z_0=1$. 
\end{thm}
\begin{rem}\leavevmode
\begin{itemize}
\item On the right-hand-side we have not used the compact notation $Z_{n-1}(\alphab,\betab,\gammab,\deltab;d+\beta_1-\delta_1)$ because not all angles are indexed by $1,\dots,n-1$: it alternates between $1,\dots,n-1$ and $2,\dots,n$. 
\item A consequence of Theorem~\ref{thm:partition_function} is that the partition function admits a product form; we do not make this explicit since it would involve heavy indices notation. Another remark is that having a product form, and assuming the weights to be periodic for example, allows to have an explicit formula for the \emph{free energy} of the model, defined as minus the exponential growth rate of the partition function. 
\end{itemize}
\end{rem}
The proof of Theorem~\ref{thm:partition_function} relies on the evolution of the partition function under local moves, it is postponed until Section~\ref{sec:proof_partition}. We first turn to interesting corollaries; the first one is the specification of Theorem~\ref{thm:partition_function} to the genus 0 case. 

\begin{cor}[Genus 0 case]\label{cor:genus_0}
Assume that the underlying M-curve $\Sigma$ is the Riemann sphere $\hat{\CC}$, implying that the weights are given by Kenyon's critical weights~\eqref{equ:critical_weights}. Then, the partition function of the Aztec diamond is equal to
{\small \begin{align}
Z_n&(\alphab,\betab,\gammab,\deltab)=2^{n(n+1)}
\prod_{\ell=0}^{n-1} \prod_{j=1}^{n-\ell}|\sin(\betabar_{j+\ell}-\alphabar_{j})\sin(\deltabar_{j+\ell}-\gammabar_j)|\label{equ:genus_0_gen_0}\\
=&2^{n(n+1)}\prod_{\ell=0}^{n-1} \prod_{j=1}^{n-\ell}|\sin(\deltabar_{j+\ell}-\betabar_{j+\ell})\sin(\gammabar_j-\alphabar_j)+\sin(\betabar_{j+\ell}-\gammabar_j)\sin(\deltabar_{j+\ell}-\alphabar_j)|.\label{equ:genus_0_gen}
\end{align}}
\begin{itemize}
\item If furthermore, for all $\alphab\equiv\alpha,\betab\equiv\beta,\gammab\equiv\gamma,\deltab\equiv\delta$ for some $\alpha<\gamma<\beta<\delta\in A_0=S^1$, then
\begin{align*}
Z_n(\alphab,\betab,\gammab,\deltab)
=2^{n(n+1)}|\sin(\betabar-\alphabar)\sin(\deltabar-\gammabar)|^{\frac{n(n+1)}{2}}.
\end{align*}
\item If furthermore,
  $\bar{\beta}-\bar{\alpha}=\bar{\gamma}-\bar{\delta}=\frac{1}{2}$, then 
  $\bar{\gamma}-\bar{\alpha}=\bar{\delta}-\bar{\beta}:=\rho\in(0,\frac{1}{2}),
  \ \bar{\beta}-\bar{\gamma}=\bar{\alpha}+1-\bar{\delta}=\frac{1}{2}-\rho\in(0,\frac{\pi}{2})$, and
\begin{align}\label{equ:Propp}
Z_n(\alphab,\betab,\gammab,\deltab)=2^{n(n+1)}.
\end{align}
\end{itemize}
\end{cor}
\begin{rem}\label{rem:partition_function_genus0}\leavevmode
\begin{itemize}
\item Recall that the Riemann theta function is equal to 1 in the genus $0$ case, so that there is no discrete Abel map.
\item Formula~\eqref{equ:genus_0_gen} can be seen as a generalization of
  Stanley's formula~\cite{BYYangPhD,Propp_talk} when more parameters are allowed, see also the proof of Corollary~\ref{cor:Stanley} below. 
\item In the last case, if we suppose moreover that $\rho=\frac{1}{4}$, all edge weights are equal to $2\sin(\frac{\pi}{4})=\sqrt{2}$. Since the number of edges in a dimer configuration is equal to the number of white vertices, that is $n(n+1)$; dividing~\eqref{equ:Propp} by $(\sqrt{2})^{n(n+1)}$, one recovers the celebrated result of~\cite{EKLP}, which states that the number of domino tilings of the Aztec diamond is equal to $2^{\frac{n(n+1)}{2}}$. It is interesting to note that this same formula also holds in some cases of non-uniform weights. 
\end{itemize}
\end{rem}

\begin{proof}
Returning to the definition of the weights in the genus 0 case, see Equation~\eqref{equ:critical_weights}, we have that all terms involving the Riemann theta function $\theta$ are equal to 1 and that the modulus $|E(\alpha,\beta)|$ of the prime form is equal to $|2\sin(\betabar-\alphabar)|$. As a consequence, the recurrence formula of Theorem~\ref{thm:partition_function} gives 
\begin{align*}
Z_n&(\alphab,\betab,\gammab,\deltab)=\\
&=2^{2n}\cdot\Bigl[\prod_{j=1}^{n}\ |\sin(\betabar_j-\alphabar_j)\sin(\deltabar_j-\gammabar_j)|\Bigr]\cdot Z_{n-1}((\alpha_j)_{j=1}^{n-1},(\beta_j)_{j=2}^n,(\gamma_j)_{j=1}^{n-1},(\delta_j)_{j=2}^n).
\end{align*}
Iterating this induction and using that the initial condition is $Z_0=1$, we obtain Formula~\eqref{equ:genus_0_gen_0}. Formula~\eqref{equ:genus_0_gen} is obtained using classical trigonometric identities: for all $\alpha,\beta,\gamma,\delta$,
{\small \begin{align*}
\sin\,&(\betabar-\alphabar)\sin(\deltabar-\gammabar)=\\
&=\frac{1}{2}[\cos(\betabar-\alphabar-\deltabar+\gammabar)-\cos(\betabar-\alphabar+\deltabar-\gammabar)]\\
&=\frac{1}{2}[\cos(\betabar-\alphabar-\deltabar+\gammabar)-\cos(\betabar-\gammabar+\alphabar-\deltabar)+\cos(\betabar-\gammabar+\alphabar-\deltabar)
-\cos(\betabar-\alphabar+\deltabar-\gammabar)]\\
&=\sin(\deltabar-\betabar)\sin(\gammabar-\alphabar)+\sin(\betabar-\gammabar)\sin(\deltabar-\alphabar),
\end{align*}}
where in the third line, we have subtracted and added $\cos(\betabar-\gammabar+\alphabar-\deltabar)$. 
\qedhere
\end{proof}

As a consequence of Corollary~\ref{cor:genus_0}, we also recover Stanley's
celebrated formula~\cite{Propp_talk,BYYangPhD}, see also~\cite[{Section~6}]{ryg}.
Suppose that edges are assigned weights $\xb=(x_j)_{j=1}^n,\yb=(y_j)_{j=1}^n, \wb=(w_j)_{j=1}^n,\zb=(z_j)_{j=1}^n$ as in Figure~\ref{fig:fig_Stanley}. Let us denote by $Z(\xb,\yb,\zb,\wb)$ the corresponding partition function, keeping in mind that the notation $\xb,\yb,\zb,\wb$ is used for Stanley's weights and that the notation $\alphab,\betab,\gammab,\deltab$ is used for Fock's weights.

\begin{cor}[Stanley's formula]\label{cor:Stanley}
The partition function of the Aztec diamond with weights $\xb,\yb,\zb,\wb$ as in Figure~\ref{fig:fig_Stanley} is equal to
\[
Z(\xb,\yb,\zb,\wb)=\prod_{\ell=0}^{n-1} \prod_{j=1}^{n-\ell}\bigl(x_j w_{j+\ell}+y_j z_{j+\ell}\bigr).
\]
\end{cor}
\begin{proof}
For the purpose of this proof, let us denote by $\nu$ Fock's weight function and by $\tilde{\nu}$ Stanley's one. Recall from Proposition~\ref{prop:gauge_Stanley} that, for every $\alpha_1<\gamma<\delta\in A_0=S^1$, there exists $(\alpha_j)_{j=2}^n$, $(\beta_j)_{j=1}^n$ such that $\alphab<\gamma<\betab<\gamma$ and such that the weight functions $\nu$ and $\tilde{\nu}$ are gauge equivalent. Let us take such angles and compute $Z(\alphab,\betab,\gammab,\deltab)$ using Corollary~\ref{cor:genus_0}, Equation~\eqref{equ:genus_0_gen}. We obtain 
{\small 
\begin{align*}
Z(&\alphab,\betab,\gammab,\deltab)=
2^{n(n+1)}\prod_{\ell=0}^{n-1} \prod_{j=1}^{n-\ell}
|\sin(\deltabar-\betabar_{j+\ell})\sin(\gammabar-\alphabar_j)+\sin(\betabar_{j+\ell}-\gammabar)\sin(\deltabar-\alphabar_j)|\\
&=2^{n(n+1)}\prod_{\ell=0}^{n-1} \prod_{j=1}^{n-\ell}|\sin(\betabar_{j+\ell}-\gammabar)\sin(\deltabar-\alphabar_j)|\prod_{\ell=0}^{n-1} \prod_{j=1}^{n-\ell}\Bigl(\frac{|\sin(\deltabar-\betabar_{j+\ell})\sin(\gammabar-\alphabar_j)|}{|\sin(\betabar_{j+\ell}-\gammabar)\sin(\deltabar-\alphabar_j)|}+1\Bigr)\\
&=2^{n(n+1)}\prod_{\ell=0}^{n-1} \prod_{j=1}^{n-\ell}[|\sin(\betabar_{j+\ell}-\gammabar)\sin(\deltabar-\alphabar_j)|]\prod_{\ell=0}^{n-1} \prod_{j=1}^{n-\ell}
\Bigl(\frac{x_j w_{j+\ell}}{y_j z_{j+\ell}}+1\Bigr)\\
&=2^{n(n+1)}\prod_{\ell=0}^{n-1} \prod_{j=1}^{n-\ell}\frac{|\sin(\betabar_{j+\ell}-\gammabar)\sin(\deltabar-\alphabar_j)|}{y_jz_{j+\ell}}
\prod_{\ell=0}^{n-1} \prod_{j=1}^{n-\ell}
\bigl(x_j w_{j+\ell}+y_j z_{j+\ell}\bigr),
\end{align*}}
where in the third equality we used gauge equivalence of the weight functions $\nu$ and $\tilde{\nu}$, see Equation~\eqref{equ:gauge_Stanley}. Using gauge equivalence again, we know from Equation~\eqref{equ:partition_gauge} that for any reference dimer configuration $\Ms_0$,
\begin{align*}
Z(\xb,\yb,\wb,\zb)&=\frac{\tilde{\nu}(\Ms_0)}{\nu(\Ms_0)}\,Z(\alphab,\betab,\gammab,\deltab).
\end{align*}
The proof is concluded by taking as reference dimer configuration $\Ms_0$ the one of Figure~\ref{fig:fig_Stanley}, and observing that 
\begin{equation*}
\nu(\Ms_0)= 2^{n(n+1)}\prod_{\ell=0}^{n-1} \prod_{j=1}^{n-\ell}|\sin(\betabar_{j+\ell}-\gammabar)\sin(\deltabar-\alphabar_j)|,\quad 
\tilde{\nu}(\Ms_0)=\prod_{\ell=0}^{n-1} \prod_{j=1}^{n-\ell} y_j z_{j+\ell}. \qedhere
\end{equation*}
\end{proof}

\subsection{Proof of Theorem~\ref{thm:partition_function}}\label{sec:proof_partition}

\paragraph{Local moves.}

Two local moves play a crucial rule in the study of the dimer model on minimal graphs: the spider move~\cite{Kuperberg,Propp} and contraction/expansion of a degree two vertex; see~\cite{Thurston,GK} for more on the subject. In the setting of the Aztec diamond, these moves allow to reduce a size $n$ Aztec diamond into a size $n-1$ Aztec diamond, see for example~\cite{Propp_talk}. Our argument consists in using this induction in the case of the dimer model with Fock's weights while controlling the effect on the partition function. The fact that this works in this very general framework relies on Fay's trisecant identity~\cite{Fay,Fock,BCdT:genusg}.

The setting for the next lemma is the dimer model with Fock's weights, see Equation~\eqref{eq:Focks_weight}, on any finite minimal graph. Let us denote by $Z$ the dimer partition function before the move is performed, and by $Z'$, resp. $Z''$, the partition function after having performed a spider move, resp. a contraction of a degree two vertex, see Figure~\ref{fig:elem_moves}. This figure also illustrates the evolution of the oriented train-tracks, of their angle parameters, and introduces the notation of Lemma~\ref{lem:elem_moves}. 

\begin{figure}[h!]
  \centering
  \begin{overpic}[width=\linewidth]{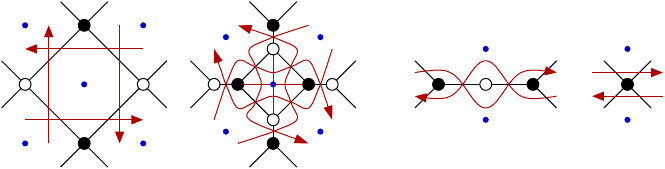}
  \put(2,7){\scriptsize $\alpha$} 
  \put(2,17){\scriptsize $\beta$} 
  \put(6,1){\scriptsize $\gamma$} 
  \put(18,1){\scriptsize $\delta$} 
  \put(1,2){\scriptsize $\fs_1$} 
  \put(23,2){\scriptsize $\fs_2$} 
  \put(23,22){\scriptsize $\fs_3$}
  \put(1,22){\scriptsize $\fs_4$}
  \put(10,12){\scriptsize $\fs$}
  \put(59,10){\scriptsize $\alpha$} 
  \put(59,14){\scriptsize $\beta$}
  \put(72,5){\scriptsize $\fs$} 
  \put(72,19){\scriptsize $\fs'$} 
  \end{overpic}
\caption{Spider move and contraction of a degree two vertex. }\label{fig:elem_moves}  
\end{figure}
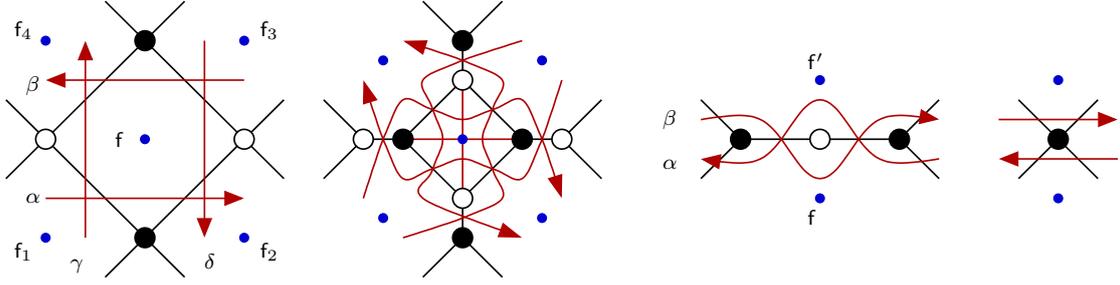

\begin{lem}\label{lem:elem_moves}
The following equations describe the evolution of the partition function under 
\begin{enumerate}
\item a spider move
\begin{equation*}
Z=\frac{\prod_{j=1}^4 \theta(t+\mapd(\fs_j))}{|E(\alpha,\beta)E(\gamma,\delta)|}\frac{\theta(t+\mapd(\fs)+\alpha+\beta-\gamma-\delta)}{\theta(t+\mapd(\fs))} Z',
\end{equation*}
where $\fs$ is the vertex at the center of the square, and $\fs_1,\dots,\fs_4$ are the vertices corresponding to faces bounding the square, see Figure~\ref{fig:elem_moves} (left). 
\item a contraction of a degree two vertex
\begin{equation*}
Z=\frac{|E(\alpha,\beta)|}{\theta(t+\mapd(\fs))\theta(t+\mapd(\fs'))}Z'',
\end{equation*}
where $\fs,\fs'$ are the two vertices corresponding to the faces above and below the degree two vertex, see Figure~\ref{fig:elem_moves} (right).
\end{enumerate}
\end{lem}

\begin{proof}
For the purpose of this proof it is convenient to write Fock's weight $\Ks_{\ws,\bs}$ of Equation~\eqref{eq:Focks_weight} as 
\[
\Ks_{\ws,\bs}=\Ks^{\fs,\fs'}_{\alpha,\beta}. 
\]

\begin{figure}[h!]
  \centering
  \begin{overpic}[width=\linewidth]{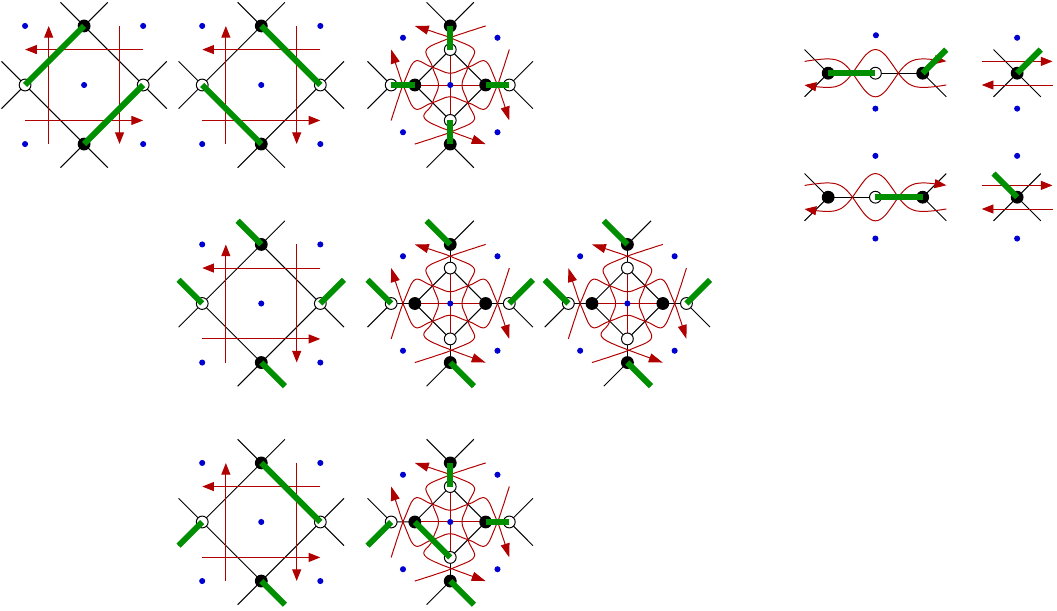}
  \put(15,40){\scriptsize $Z_1$} 
  \put(42,40){\scriptsize $Z_1'$} 
  \put(24,18){\scriptsize $Z_2$}
  \put(50,18){\scriptsize $Z_2'$}
  \put(24,-2){\scriptsize $Z_3$}
  \put(42,-2){\scriptsize $Z_3'$}
  \put(54,8){\scriptsize \text{+ symmetries $Z_4,Z_4',\dots,Z_6,Z_6'$}}
  \put(83,45){\scriptsize $Z_1$}
  \put(93,45){\scriptsize $Z_1''$}
  \put(83,33){\scriptsize $Z_2$}
  \put(93,33){\scriptsize $Z_2''$}
  \end{overpic}
\caption{Partial partition function $Z_1,\dots,Z_6$, resp. $Z_1',\dots,Z_6'$ before, resp. after, the spider moves, and $Z_1,Z_2$, resp. $Z_1'',Z_2''$, before, resp. after, the contraction of a degree two vertex.}\label{fig:fig_move_proof}  
\end{figure}
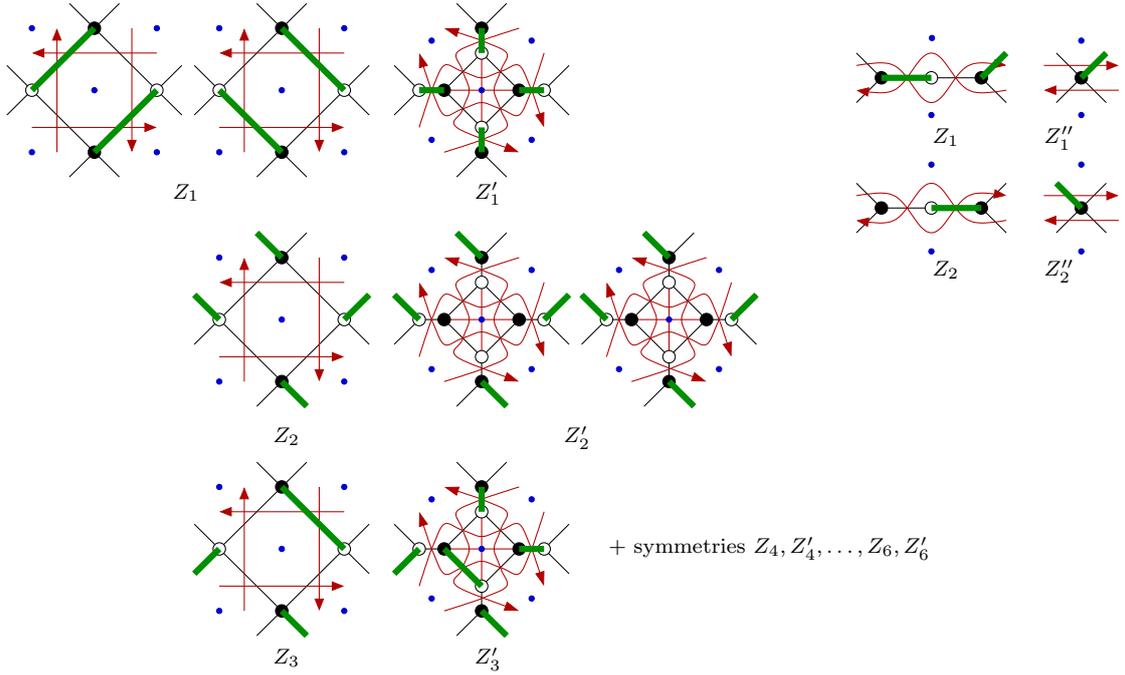

By~\cite{Fock,BCdT:genusg} we know that the dimer model with Fock's weights, when considered as a model with face weights, is invariant under these two moves. In particular, for the spider move, this implies that when considered as a model with edge weights, there exists a constant $C$ such that for all six partial partition function of Figure~\ref{fig:fig_move_proof}, obtained by fixing the dimer configuration on edges bounding the square, we have $Z_j=C Z_j'$, $j\in\{1,\dots,6\}$. 
In order to compute this constant, we can thus choose any of these six cases, and we choose the fourth one since it leads to simple computations (recovering $C$ from the others requires using Fay's trisecant identity). Using Figure~\ref{fig:elem_moves} to identify the angle parameters, we obtain
\begin{align*}
C&=\frac{Z}{Z'}=\frac{|\Ks^{\fs,\fs_3}_{\beta,\delta}|}{|\Ks_{\gamma,\delta}^{\fs_3,\fs_4}\Ks_{\beta,\alpha}^{\fs_3,\fs_2}\Ks_{\beta,\delta}^{\fs_1,\fs'}|}=\frac{\prod_{j=1}^4 \theta(t+\mapd(\fs_j))}{|E(\alpha,\beta)E(\gamma,\delta)|}\frac{\theta(t+\mapd'(\fs))}{\theta(t+\mapd(\fs))},
\end{align*}
where $\mapd'(\fs)$ is the value of the discrete Abel map after the move is performed. Returning to its definition, see Section~\ref{sec:Fock}, we have
\[
\mapd'(\fs)=\mapd(\fs)+\alpha+\beta-\gamma-\delta.
\]
Note that the value of the discrete Abel map at the other dual vertices remains unchanged. 

For the contraction of a degree two vertex, in a similar way, we can choose any of the two cases, and we obtain
\begin{align*}
C&=\frac{Z}{Z''}=|\Ks_{\alpha,\beta}^{\fs,\fs'}|=\frac{|E(\alpha,\beta)|}{\theta(t+\mapd(\fs))\theta(t+\mapd(\fs'))}.\qedhere
\end{align*}
\end{proof}

\paragraph{Proof of Theorem~\ref{thm:partition_function}.}

The proof consists in performing a sequence of spider moves and contraction of degree two vertices to transform an Aztec diamond of size $n$ into an Aztec diamond of size $n-1$~\cite{Propp_talk}, while keeping track of the evolution of the partition function using Lemma~\ref{lem:elem_moves}, see Figure~\ref{fig:induction}. 

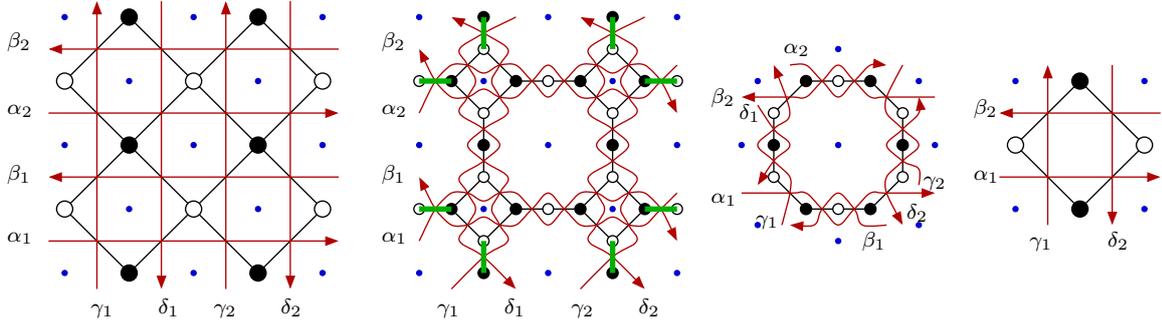
\begin{figure}[h!]
  \centering
  \begin{overpic}[width=\linewidth]{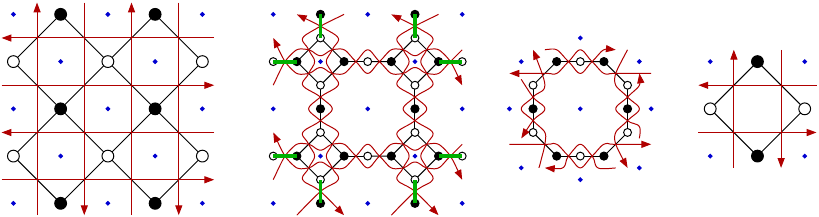}
  \put(4,-2){\scriptsize $\gamma_1$}
  \put(10,-2){\scriptsize $\delta_1$}
  \put(15,-2){\scriptsize $\gamma_2$}
  \put(21,-2){\scriptsize $\delta_2$}

  \put(-3.5,4.5){\scriptsize $\alpha_1$}
  \put(-3.5,10){\scriptsize $\beta_1$}
  \put(-3.5,16){\scriptsize $\alpha_2$}
  \put(-3.5,22){\scriptsize $\beta_2$}

  \put(35,-2){\scriptsize $\gamma_1$}
  \put(41,-2){\scriptsize $\delta_1$}
  \put(47,-2){\scriptsize $\gamma_2$}
  \put(53,-2){\scriptsize $\delta_2$}

  \put(30,4.5){\scriptsize $\alpha_1$}
  \put(30,10){\scriptsize $\beta_1$}
  \put(30,16){\scriptsize $\alpha_2$}
  \put(30,22){\scriptsize $\beta_2$}

  \put(63.5,6){\scriptsize $\gamma_1$}
  \put(59.5,8){\scriptsize $\alpha_1$}
  \put(59.5,17){\scriptsize $\beta_2$}
  \put(62,15.5){\scriptsize $\delta_1$}
  \put(66,21.5){\scriptsize $\alpha_2$}
  \put(73,4){\scriptsize $\beta_1$}
  \put(77,6.5){\scriptsize $\delta_2$}
  \put(78.5,9.5){\scriptsize $\gamma_2$}
  
  \put(83,10){\scriptsize $\alpha_1$}
  \put(83,16){\scriptsize $\beta_2$}
  \put(88,4){\scriptsize $\gamma_1$}
  \put(95,4){\scriptsize $\delta_2$}
  \end{overpic}
\caption{Induction step for computing the partition function of the Aztec diamond.}
\label{fig:induction}
\end{figure}  

Let us simply denote by $Z:=Z_n(\alphab,\betab,\gammab,\deltab;d)$ the partition function defined in~\eqref{equ:def_partition}. 

\paragraph{Step 1.} Suppose $n\geq 1$, then Step 1 consists in performing a spider move at each of the square whose dual vertex has odd coordinates. Let us denote by $Z^1$ the partition function of the dimer model on the resulting graph, see Figure~\ref{fig:induction} (first two figures). Using Point 1. of Lemma~\ref{lem:elem_moves}, we obtain
\begin{equation*}
\begin{split}
Z=Z^1\cdot \frac{[\prod_{\fs\in\{\corner_n\}}\theta(t+\mapd(\fs))] 
[\prod_{\fs\in\{\bry_n\setminus\corner_n\}}\theta(t+\mapd(\fs))^2]
[\prod_{\fs\in\{\inteven_n\}}\theta(t+\mapd(\fs))^4]}{\prod_{j=1}^n |E(\alpha_j,\beta_j)^n E(\gamma_j,\delta_j)^n|}\cdot \\
\cdot \prod_{\fs\in\{\odd_n\}}\frac{\theta(t+\mapd(\fs)+\alpha+\beta-\gamma-\delta)}{\theta(t+\mapd(\fs))},
\end{split}
\end{equation*}
where we generically denote by $\alpha,\beta,\gamma,\delta$ the angles around a dual vertex $\fs$ with odd coordinates, see Figure~\ref{fig:elem_moves} (left). 

\paragraph{Step 2.} Suppose $n\geq 1$. Looking at the graph obtained after Step 1, one notes that all the boundary edges have to belong to a dimer configuration, see Figure~\ref{fig:induction} (second graph, green edges). We can thus factor out their contribution and then remove all their incident edges, which yields the third graph of Figure~\ref{fig:induction}; let us denote by $Z^2$ its partition function. We thus have
\begin{align*}
Z^1&=Z^2 \cdot \frac{\prod_{j=1}^n |E(\alpha_j,\beta_j)^2 E(\gamma_j,\delta_j)^2|}{
\prod_{\fs\in\{\bry_n\}}\theta(t+\mapd(\fs))^2}.
\end{align*}
Note that if $n=1$, then this identity holds by setting $Z^2=1$. 

\paragraph{Step 3.} Suppose $n\geq 2$ (otherwise this step does not happen); one then performs the contraction move on all degree two (inner) vertices. Let us denote by $Z^3$ the partition function of the graph thus obtained, see Figure~\ref{fig:induction} (fourth graph). Using Point 2. of Lemma~\ref{lem:elem_moves}, we obtain
\begin{align*}
Z^2=Z^3 \cdot \frac{\prod_{j=1}^{n} |E(\alpha_j,\beta_j)^{n-1} E(\gamma_j,\delta_j)^{n-1}|}{[\prod_{\fs\in\{\bry_n\setminus\corner_n\}}\theta(t+\mapd(\fs))][\prod_{\fs\in\{\inteven_n\}}\theta(t+\mapd(\fs))^4]}.
\end{align*}

\paragraph{Conclusion.} When $n\geq 2$, combining all three steps gives
{\small 
\begin{align*}
Z=Z^3\cdot \Bigl[\prod_{j=1}^{n} |E(\alpha_j,\beta_j)E(\gamma_j,\delta_j)|\Bigr]\prod_{\fs\in\{\odd_n\}}
\frac{\theta(t+\mapd(\fs)+\alpha+\beta-\gamma-\delta)}{\theta(t+\mapd(\fs))}
\frac{1}{\prod_{\fs\in\{\bry_n\}}\theta(t+\mapd(\fs))}.
\end{align*}
}
The proof is concluded by observing that $Z^3$ is the partition function of an Aztec diamond of size $n-1$, with angles $(\alpha_j)_{j=1}^{n-1},(\beta_j)_{j=2}^{n-1},(\gamma_j)_{j=1}^{n-1},(\delta_j)_{j=2}^{n}$. Note moreover that the value of the discrete Abel map at the vertex $0=(0,0)$ in the Aztec diamond of size $n-1$, is the value of the discrete Abel map at the vertex $(1,1)$ in the original Aztec diamond of size $n$ after the spider move is performed, that is:
\[
(\gamma_1-\alpha_1)+(\alpha_1+\beta_1-\gamma_1-\delta_1)=\beta_1-\delta_1.
\]
When $n=1$, only the first two steps are performed, and we obtain
\begin{align*}
Z&=\frac{[\prod_{\fs\in\{\corner_n\}}\theta(t+\mapd(\fs))]}{|E(\alpha_1,\beta_1) E(\gamma_1,\delta_1)|}\cdot \frac{\theta(t+\mapd(\fs)+\alpha_1+\beta_1-\gamma_1-\delta_1)}{\theta(t+\mapd(\fs))}\frac{|E(\alpha_1,\beta_1)^2 E(\gamma_1,\delta_1)^2|}{[\prod_{\fs\in\{\bry_n\}}\theta(t+\mapd(\fs))^2]}\\
&=|E(\alpha_1,\beta_1) E(\gamma_1,\delta_1)|\frac{\theta(t+\mapd(\fs)+\alpha_1+\beta_1-\gamma_1-\delta_1)}{\theta(t+\mapd(\fs))} \frac{1}{\prod_{\fs\in\{\bry_n\}}\theta(t+\mapd(\fs))},
\end{align*}
where in the last line we used that the boundary vertices are exactly the corner
vertices in this case. We deduce that the
formula of Theorem~\ref{thm:partition_function} also holds in this case.\hfill$\square$

\section{Aztec diamond and minimal graphs}\label{sec:aztec_minimal}

The Aztec diamond of size $n$ can be seen as a finite subgraph of an infinite
minimal bipartite graph, obtained by gluing four ``quadrants'' made of hexagons
to the sides of the Aztec diamond, where minimal graphs are defined in Section~\ref{sec:Fock}. Let us denote this infinite graph by $\Gs_n$,
see Figure~\ref{fig:aztec_minimal} for a representation of $\Gs_3$. 
Similarly to the case of the Aztec diamond, the train-tracks of $\Gs_n$ are made
of four families, see Figure~\ref{fig:aztec_minimal}.
They are assigned angle parameters
  $\alphab=(\alpha_j)_{j=1}^{\infty}$, 
$\betab=(\beta_j)_{j=-\infty}^{n}$,
$\gammab=(\gamma_j)_{j=1}^{\infty}$,
$\deltab=(\delta_j)_{j=-\infty}^{n}$, in such a way that: the indices $1$
to $n$ correspond to the angle parameters of the train-tracks of $\Az_n$, and
the cyclic order~\eqref{equ:angle_condition} is preserved
for all the train-tracks (not only those crossing edges of $\Az_n$).

\begin{figure}[h!]
  \centering
  \def\svgwidth{6cm}
  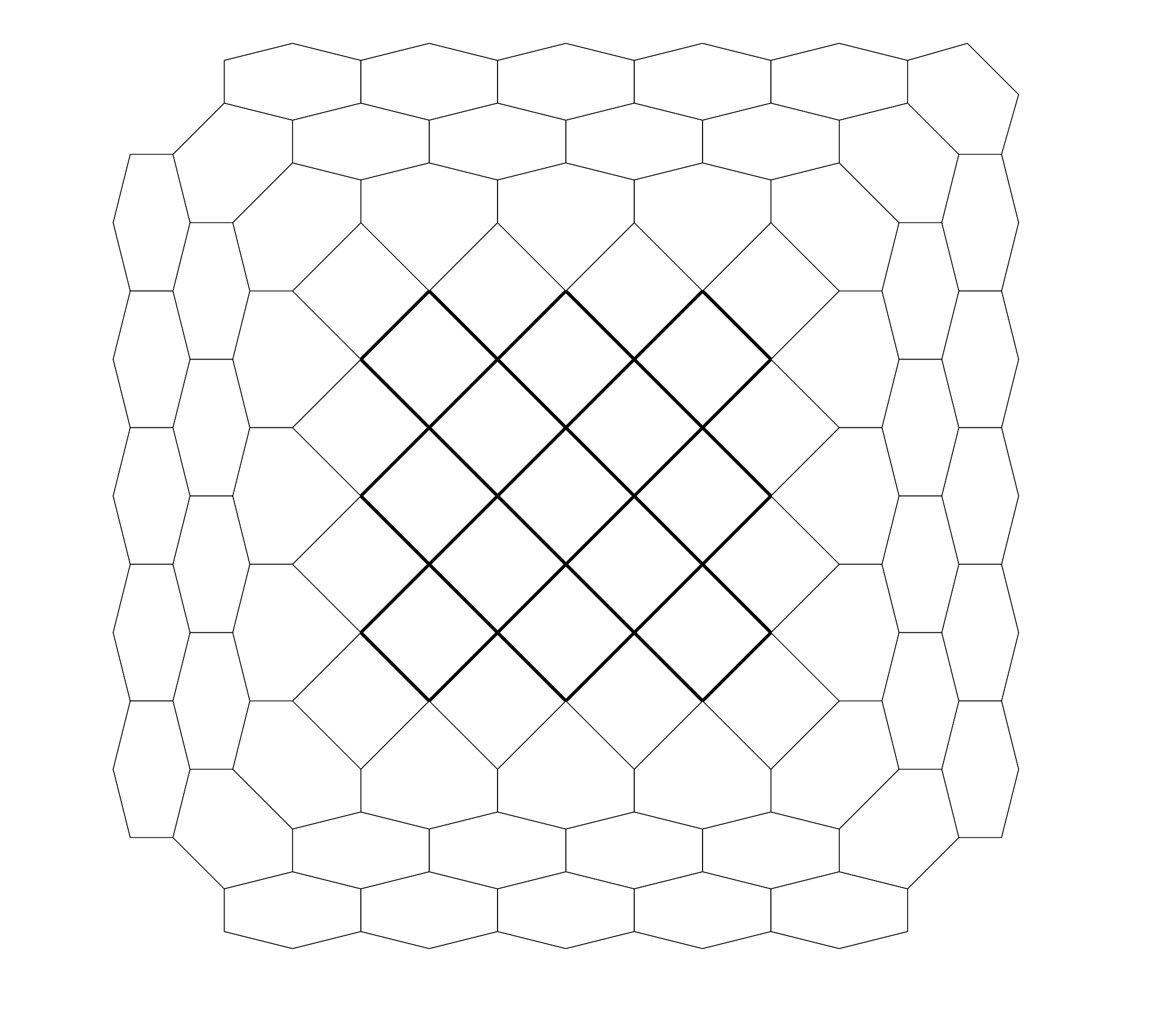
  \hfill
  \def\svgwidth{6cm}
\begingroup%
  \makeatletter%
  \providecommand\color[2][]{%
    \errmessage{(Inkscape) Color is used for the text in Inkscape, but the package 'color.sty' is not loaded}%
    \renewcommand\color[2][]{}%
  }%
  \providecommand\transparent[1]{%
    \errmessage{(Inkscape) Transparency is used (non-zero) for the text in Inkscape, but the package 'transparent.sty' is not loaded}%
    \renewcommand\transparent[1]{}%
  }%
  \providecommand\rotatebox[2]{#2}%
  \newcommand*\fsize{\dimexpr\f@size pt\relax}%
  \newcommand*\lineheight[1]{\fontsize{\fsize}{#1\fsize}\selectfont}%
  \ifx\svgwidth\undefined%
    \setlength{\unitlength}{920.2934405bp}%
    \ifx\svgscale\undefined%
      \relax%
    \else%
      \setlength{\unitlength}{\unitlength * \real{\svgscale}}%
    \fi%
  \else%
    \setlength{\unitlength}{\svgwidth}%
  \fi%
  \global\let\svgwidth\undefined%
  \global\let\svgscale\undefined%
  \makeatother%
  \begin{picture}(1,0.96681549)%
    \lineheight{1}%
    \setlength\tabcolsep{0pt}%
    \put(0,0){\includegraphics[width=\unitlength,page=1]{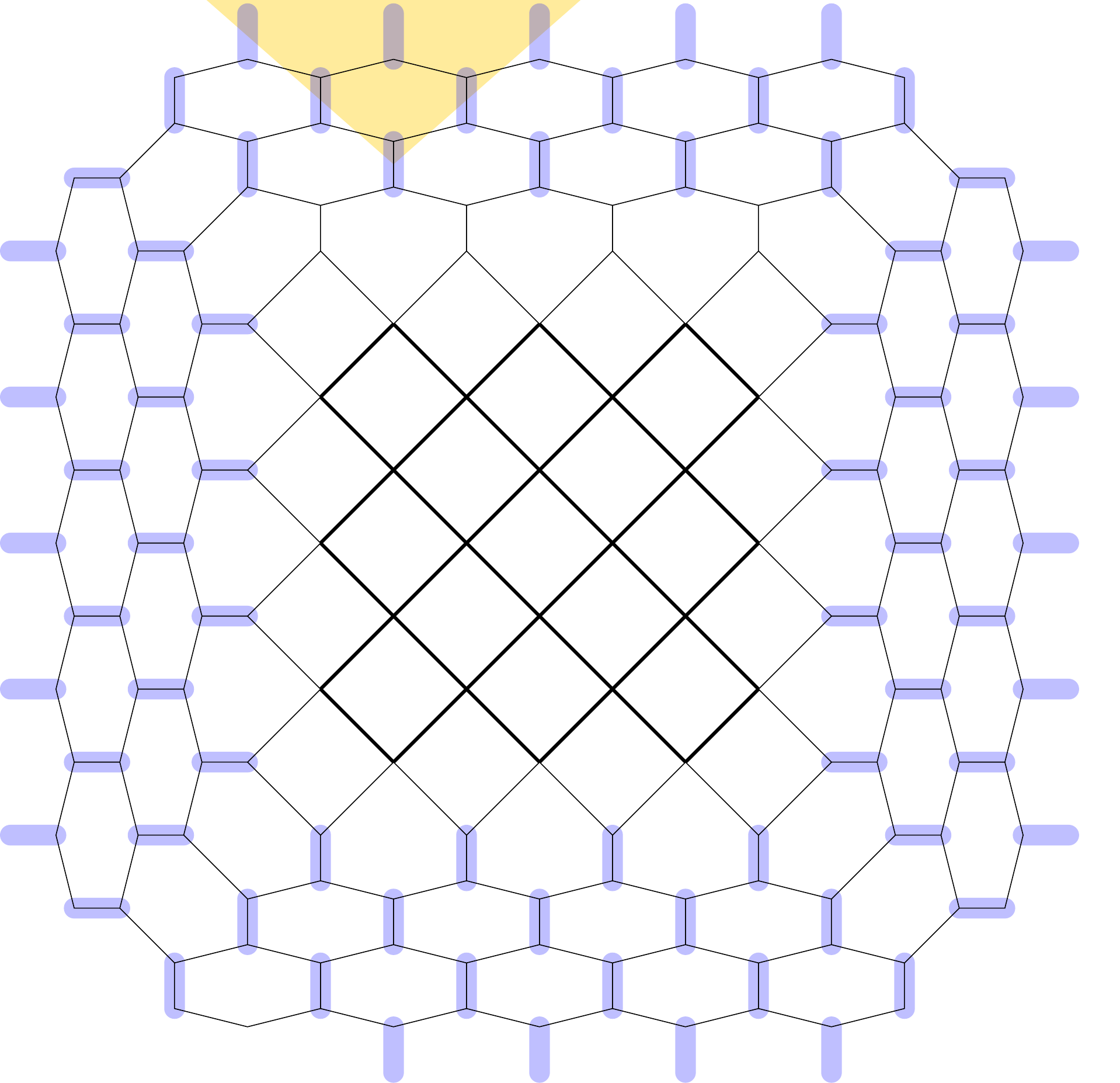}}%
    \put(0.03368917,0.88124496){\color[rgb]{0,0,0}\makebox(0,0)[lt]{\lineheight{1.25}\smash{\begin{tabular}[t]{l}$\ddots$\end{tabular}}}}%
    \put(0.84864661,0.07443711){\color[rgb]{0,0,0}\makebox(0,0)[lt]{\lineheight{1.25}\smash{\begin{tabular}[t]{l}$\ddots$\end{tabular}}}}%
    \put(0.85679617,0.87309539){\color[rgb]{0,0,0}\makebox(0,0)[lt]{\lineheight{1.25}\smash{\begin{tabular}[t]{l}$\iddots$\end{tabular}}}}%
    \put(0.04183875,0.08258662){\color[rgb]{0,0,0}\makebox(0,0)[lt]{\lineheight{1.25}\smash{\begin{tabular}[t]{l}$\iddots$\end{tabular}}}}%
    \put(0,0){\includegraphics[width=\unitlength,page=2]{fig_aztec_minimal.pdf}}%
    \put(0.34337302,0.76307613){\color[rgb]{0,0,0}\makebox(0,0)[lt]{\lineheight{1.25}\smash{\begin{tabular}[t]{l}$\ws$\end{tabular}}}}%
  \end{picture}%
\endgroup%

  \caption{\textbf{Left:} The four families of train-tracks of $\Gs_n$. The families
  $\alphab$ and $\gammab$ are indexed by positive integers. The family
$\betab$ and $\deltab$ are indexed by integers less or equal to $n$.
\textbf{Right:} The Aztec diamond of size $n$ as a finite part of an infinite minimal
  graph $\Gs_n$, for $n=3$. The added vertices form four quadrants with hexagonal
connectivity, in the north, south, west, east of the original Aztec region. The
(beginning of) the light-cone of a white vertex $\ws$ in the north quadrant is
represented in light yellow. The icy edges in those quadrants, made of the
horizontal and vertical edges, are highlighted in blue, and create a perfect
matching of the complement of the Aztec diamond inside $\Gs_n$.}
  \label{fig:aztec_minimal}
\end{figure}

The graph $\Gs_n$ belongs to a class of infinite graphs introduced
  in~\cite{Speyer} by Speyer, giving a combinatorial interpretation of the
  solution of the \emph{octahedron recurrence} as the partition function of
  dimer configurations on those graphs with specific condition at infinity.

In addition to the finite Kasteleyn matrix $\Ks=\Ks_{\Az_n}$ with rows and
columns indexed by vertices of the Aztec diamond of size $n$ and 
we consider also in this section $\tilde{\Ks}=\Ks_{\Gs_n}$, the infinite Fock Kasteleyn operator for
the infinite minimal graph $\Gs_n$. $\Ks$ is then the restriction of $\tilde{\Ks}$
to vertices of $\Az_n$.
The goal of this section is to use the framework of this paper to define an
  inverse of Fock's Kasteleyn operator $\tilde{\Ks}$ on the infinite graph $\Gs_n$ such that the
  corresponding probability measure
induced on the edges of $\Az_n$ is the Boltzmann measure computed by
    Theorem~\ref{thm:measure_Kenyon} from the finite matrices $\Ks$ and $\Ks^{-1}$.
    Moreover, this measure is frozen outside of the Aztec diamond: every edge of
    a quadrant is either present a.s.\ or absent a.s., see
    Proposition~\ref{prop:gibbs}.

Denote by $Q_N$, $Q_S$, $Q_W$, $Q_E$, the four quadrants made of hexagons
respectively above, below, to the left, to the right of the Aztec diamond.
We need some more terminology to describe the graph:
\begin{defi}
  \leavevmode
  \begin{itemize}
    \item We say that an edge of $Q_S$ or $Q_N$ (resp.\ of $Q_W$ or $Q_E$) is
      \emph{icy} if it is vertical (resp.\ horizontal). Every vertex of these
      quadrants is incident with exactly one icy edge. So it makes sense to talk
      about the icy edge associated to a white (or black) vertex in one of the
      four quadrants.
    \item The \emph{light cone} of a white vertex $\ws$ in $Q_N$ is the region
      of the plane above the two half lines starting from middle of the icy edge
      associated to $\ws$, and going north-west and north-east with a 45-degree
      angle. See Figure~\ref{fig:aztec_minimal}.
      
      We define in a similar way the light cone for white vertices in $Q_S$,
      and for black vertices in $Q_W$ and $Q_E$.
  \end{itemize}
\end{defi}

We can define an infinite matrix $\Kinv$ with rows (resp.\ columns) indexed by
black (resp.\ white) vertices of $\Gs_n$ by extending Formula~\eqref{eq:Kinv} used
to compute $\Ks^{-1}$ to pairs of
vertices which are not necessarily both in $\Az_n$.
It turns out that for about half of the pairs of vertices, the entries of $\Kinv$ are
trivially 0. Moreover, it becomes a formal inverse of the infinite Fock
Kasteleyn operator $\tilde{\Ks}$. This is made precised in the statement below,
and its proof which computes explicitly the entries when needed.

\begin{lem}\label{lem:inverse_Ktilde}
The infinite matrix $\Kinv$, has the following block structure, where rows
and columns are grouped by regions in the order $\Az_n$, $Q_N$, $Q_S$, $Q_W$, $Q_E$:

\begin{equation}
\Kinv  =
\begin{blockmatrix}
  \block[blue] (0,4){\Ks^{-1}}(1,1)
  \block[cyan] (1,4)0(1,1)
  \block[cyan] (2,4)0(1,1)
  \block[green](3,4)\star(2,1)

  \block[purple](0,2)\ast(1,2)
  \block[orange](2,3)0(1,1)

  \block[orange](1,2)0(1,1)

  \block[cyan](0,1)0(1,1)
  \block[orange](1,1)0(1,1)
  \block[orange](2,1)0(1,1)
  \block[orange](4,1)0(1,1)

  \block[cyan](0,0)0(1,1)
  \block[orange](1,0)0(1,1)
  \block[orange](2,0)0(1,1)
  \block[orange](3,0)0(1,1)

  \block[yellow](1,3)D_N(1,1)
  \block[yellow](2,2)D_S(1,1)
  \block[yellow](3,1)D_W(1,1)
  \block[yellow](4,0)D_E(1,1)

  \block[red](3,2)\diamond(2,2)
\end{blockmatrix}
\label{eq:blocksA}
\end{equation}

Moreover, the operator $\Kinv$ is a formal inverse of the infinite Fock
Kasteleyn operator $\tilde{\Ks}$ on $\Gs_n$: 
\begin{equation*}
  \Kinv \tilde{\Ks}=\operatorname{Id}_{\text{blacks}},
  \qquad
  \tilde{\Ks} \Kinv = \operatorname{Id}_{\text{whites}}.
\end{equation*}
\end{lem}

\begin{proof}
The proof follows from the construction and the computations below:
\begin{itemize}
  \item If $\bs$ is inside the Aztec diamond, we can extend the value of
    $\Ks^{-1}_{\bs,\ws}$ for $\ws$ in $Q_N$ and $Q_S$.
    \begin{itemize}
      \item If $\ws$ is in $Q_S$, the train-tracks of type $\alpha$ separating
        $\bs$ from $\ws$ contribute as
        zeros of $g_{\bs,\ws}$. The contour $\C_2$ does not contain any pole of
        $g_{\bs,\ws}$ and thus, the integral over $\C_2$ of $g_{\bs,\ws}$ in this
        case is zero. So we do not need to care about the meaning of ``$\ws$ is
        on the right of $\bs$'' in this case. The double integral is also zero
        by a similar argument: there is no train-track of type $\alpha$ passing
        below $\ws$ (in this quadrant, we see only train-tracks of type
        $\beta,\gamma,\delta$). Therefore, the contour $\C_2$ does not contain
        any of the poles in $v$ of the integrand (for any fixed $u$). Thus the
        whole integral is $0$.
      \item A similar result holds when $\ws$ belongs to $Q_N$, but for a
        slightly different reason. Let us evaluate first the double integral in
        Equation~\eqref{eq:Kinv}. The contour $\C_2$ for the $v$ variable in
        Equation~\eqref{eq:Kinv} contains all the poles of the integrand except
        the one at $v=u$. So for a fixed $u$, the integral over $v$ is equal to
        $-2i\pi$ times the residue at $v=u$, which is equal to $g_{\bs,\ws}(u)$.
        We need to integrate this residue over $u$ along $\C_1$.
        \begin{itemize}
          \item If $\bs$ is ``on the left of $\ws$'' (meaning here that no
            parameter from $\gammab$ appears as a pole of $g_{\bs,\ws}$), there
            is no pole inside $\C_1$ so the double
            integral is 0. Moreover, the indicator function in front of the single
            integral is also zero.
          \item If on the contrary, $\ws$ is ``on the right'' of
            $\bs$ (there are poles of type $\gammab$, but then none of type
            $\deltab$), So the only poles of $g_{\bs,\ws}$ are either some $\alpha_j$
            or some $\gamma_j$. The indicator function is equal to 1, so the final
            formula is $-\frac{1}{2i\pi}$ times an integral surrounding all the
            parameters from $\alphab$ and all the parameters from $\gammab$ of
            $g_{\bs,\ws}$ but since all the poles are encloses, the total is again
            equal to 0.
        \end{itemize}
        In this case again, $\Kinv_{\bs,\ws}=0$ for all $\ws\in Q_N$.
        \end{itemize}
      \item We can repeat the same arguments when $\ws$ is in the Aztec diamond,
        and $\bs$ is either in $Q_W$ or $Q_E$, by looking now first at the
        variable $u$ in the double integral.
        As a consequence we have the four cyan zero blocks in this infinite inverse.
      \item We now look at the formula when:
        \begin{itemize}
          \item $\ws$ is fixed in $Q_N$ (resp. $Q_S$) and $\bs$ is allowed to
            exit the Aztec region to be in $Q_W$, $Q_E$ and $Q_S$ (resp.\ and
            $Q_N$);
          \item $\bs$ is fixed in $Q_W$ (resp. $Q_E$) and $\ws$ is allowed to
            exit the Aztec region to be in $Q_N$, $Q_S$ and $Q_E$ (resp.\ and
            $Q_W$).
        \end{itemize}
        The previous arguments again show that in these situations, the value of
        $\Kinv_{\bs,\ws}$ is also 0. This is represented by the orange blocks.
      \item We then continue computing by taking $\bs$ in the Aztec region, and
        $\ws\in Q_W$: We can compute iteratively the value on every white vertex
        column by column, which leads at each step a linear equation with a
        single unknown. We do the same for $\ws\in Q_E$. This corresponds to the
        green block. We do  the same for $\ws$ in the Aztec region, and $\bs\in
        Q_W\cup Q_N$. This corresponds to the purple region.
      \item We then discuss the entries when $\bs$ and $\ws$ are in the same
        quadrant. Fix for example $\ws$ in $Q_N$. We know from computations
        above that the value of $\Kinv_{\bs,\ws}$ is 0 if $\bs$ is just outside of
        $Q_N$ (but connected to a white vertex $\ws'$ of $Q_N$). We can then
        solve iteratively for all the values of $\Kinv_{\bs,\ws}$ by solving
        $\sum_\bs \Ks_{\ws',\bs} A_{\bs,\ws}=\delta_{\ws',\ws}$, as at each step
        there will be always a white vertex $\ws'$ such that we know the value
        of $\Kinv_{\bs,\ws}$ for two black neighbors of $\ws'$, but the value of the
        third value is not computed yet, but is completely determined by the
        linear equation above. In particular, we obtain that $\Kinv_{\bs,\ws}=0$ if
        $\bs$ is not in the light cone of $\ws$, and is equal to
        $\frac{1}{\Ks_{\ws,\bs}}$ if $(\ws,\bs)$ is the icy edge attached to
        $\ws$.
        We proceed in the same way for $Q_S$ and, after exchanging the roles of
        the colors, for $Q_E$ and $Q_W$. This determines the entries of the four
        diagonal blocks $D_N$, $D_S$, $D_W$, $D_E$.
      \item Finally, we proceed as before to compute the entries when $\bs\in
        Q_N\cup Q_S$ and $\ws\in Q_W\cup Q_E$, by propagating known values on
        the boundary to the bulk by the linear equation. The actual entries of
        that block will not matter for our purposes.\qedhere
    \end{itemize}
  \end{proof}

It turns out that this inverse of the Kasteleyn operator has a probabilistic
meaning. More precisely, the determinant of minors of $\Kinv$ are related to local
statistics of the Boltzmann measure on the Aztec diamond, as follows:

\begin{lem}
  Let $e_1=(\ws_1,\bs_1),\ldots e_k=(\ws_k,\bs_k)$ be distinct edges of the
  Aztec diamond of size $n$, and
  $e_{k+1}=(\ws_{k+1},\bs_{k+1}),\ldots,e_l=(\ws_l,\bs_l)$ be distinct edges in
  the complement of $\Az_n$ in $\Gs_n$.
  Then
  \begin{equation}
    \left(
      \prod_{i=1}^l \tilde{\Ks}_{\ws_i,\bs_i}
    \right)
    \det_{1\leq i,j\leq l} \Kinv_{\bs_i,\ws_j}
    =
    \mathbb{P}_{\Az_n}[\text{$e_1,\ldots,e_k$ are dimers}]
    \times
    \prod_{i=k+1}^l \mathbb{I}_{\{\text{$e_i$ is icy}\}}
    \label{eq:infinite_meas_1}
  \end{equation}
  where $\mathbb{P}_{\Az_n}$ is the Boltzmann probability measure on perfect
  matchings of the Aztec diamond of size $n$, computed
  via Theorem~\ref{thm:measure_Kenyon}.
\end{lem}
\begin{proof}
  We can assume that the edges $e_{k+1},\ldots,e_l$ are ordered in such a way
  that they are grouped in four blocks corresponding to the four quadrants:
  $Q_N$, $Q_S$, $Q_W$, $Q_E$ (for
  edges joining a vertex of a quadrant to a vertex of a neighboring quadrant,
  consider them as belonging to either quadrant).
  In each block, order the edges in such a way
  that an edge $e'$ in the light-cone of a vertex of another edge $e$ of the
  same quadrant should come after $e$ in the list.

  Due to the structure of the matrix $\Kinv$, and the order we put on the edges,
  the entries of the submatrix
  \begin{equation*}
    (\Kinv_{\bs_i,\ws_j})_{1\leq i,j \leq l}
  \end{equation*}
  in columns indexed by white vertices in the quadrants $Q_N$ and $Q_S$ are
  zero, except maybe in the lower triangular part of the diagonal blocks. The
  diagonal entries for these columns are $\mathbb{I}_{\{\text{$(\ws_i,\bs_i)$ is
  icy}\}}\frac{1}{\Ks_{\ws_i,\bs_i}}$. We can therefore expand the determinant
  \begin{equation*}
    \det_{1\leq i,j \leq l} \Kinv_{\bs_i,\ws_j}
  \end{equation*}
  first along the columns corresponding to edges in the quadrants $Q_N$ and
  $Q_S$. Then, for similar reasons, we can expand the along the rows
  corresponding to edges in the quadrants $Q_W$ and $Q_E$. Therefore,
  \begin{equation*}
    \det_{1\leq i,j \leq l} \Kinv_{\bs_i,\ws_j}=\det_{1\leq i,j\leq k}
    \Kinv_{\bs_i,\ws_j}\times\prod_{i=k+1}^l \frac{\mathbb{I}_{\{\text{$(\ws_i,\bs_i)$ is
    icy}\}}}{\tilde{\Ks}_{\ws_i,\bs_i}},
  \end{equation*}
  which, once multiplied on both sides by $\left(\prod_{i=1}^{l} \tilde{\Ks}_{\ws_i,\bs_i}\right)$
  is exactly Equation~\eqref{eq:infinite_meas_1}.
\end{proof}

The structure of the graph $\Gs_n$ (as of the other infinite graphs
from~\cite{Speyer}) is such that constraints at infinity propagate in all the
quadrants up to the boundary of the Aztec region: suppose we are given a dimer
configuration of $\Gs_n$, which is such that in an annular region (large
enough to contain the Aztec diamond part in its interior), all the dimers are
icy edges. Then necessarily, all the icy edges in the four quadrants (and no
other in these parts of the graph) are present in this dimer configuration.
Such a configuration is called \emph{ultimately frozen}. Using the determinantal
formula from Theorem~\ref{thm:measure_Kenyon} with $\tilde{\Ks}$ and $J$ turns out to
define a Gibbs measure on dimer configurations on $\Gs_n$, supported on
ultimately frozen configurations.

\begin{prop}\label{prop:gibbs}
  Equation~\eqref{eq:infinite_meas_1} from the previous lemma defines a
  determinantal point process on edges of $\Gs_n$, which gives a Gibbs measure on
  the set of dimers configurations of $\Gs_n$ for the specification given by
  Fock's weights on edges.

  This probability measure $\mathbb{P}$ has the property to be \emph{frozen} outside of the
  Aztec diamond: every icy edge appear with probability one, and the other edges
  of the four quadrants appear with probability zero, in a random configuration
  sampled according to that measure. In other words, $\mathbb{P}$ is the product
  measure of the Dirac mass supported on the set of icy edges in the complement
  of the Aztec diamond and the Boltzmann measure on dimer configurations of the
  Aztec diamond $\mathbb{P}_{\Az_n}$ given by Fock's weights.
\end{prop}

\begin{proof}
  The collection of relations $\eqref{eq:infinite_meas_1}$ for all
  finite subsets of edges of $\Gs_n$ is a consistent set of finite-dimensional
  probability distributions, by the multilinearity of the determinant.

  By Kolmogorov's extension theorem, there is a unique probability measure on
  the subset of edges of $\Gs_n$ with those finite dimensional marginals, which
  is automatically a determinantal process, supported on dimer configurations of
  $\Gs_n$.

  The fact that the configuration is frozen outside of the Aztec diamond is a
  consequence of Equation~\eqref{eq:infinite_meas_1}. This implies that it
  automatically satisfies the
  Dobrushin-Lanford-Ruelle~\cite{Dobrushin,LanfordRuelle} condition for large
  enough annuli, containing the Aztec diamond in their interior. Therefore, it
  is automatically Gibbs.
\end{proof}

This inverse $\Kinv$ of the Kasteleyn matrix $\tilde{\Ks}$ is of different nature than the inverses
defined~\cite{BCdT:genusg}. Whereas the measures constructed from the latter
would correspond to ``almost linear'' height profile, the one constructed in the
proposition above has a different, extremal slope in each of the quadrants.

In both cases, the inverses are constructed as contour integrals involving the
family of functions $g_{\bs,\ws}(u)$ in the kernel of $\tilde{\Ks}$. This raises the
question of constructing other inverses with that mechanism for other boundary
conditions at infinity, for this family of graphs or generalizations, and
further to classify inverses of $\tilde{\Ks}$ with a
probabilistic meaning.

\section{Limit shapes}\label{sec:limit_shapes}

The formula for the inverse of the Kasteleyn matrix is very well suited for
asymptotics analysis.

In this section, we discuss results about limit shapes which can be obtained
directly from the analysis of this formula. Most of them are already present in
the literature. We mostly focus on a setting where parameters $\alphab$,
$\betab$, $\gammab$, $\deltab$ are chosen in a periodic way
(which is more general than having periodic weights on edges,
see~\cite[Section~4]{BCdT:genusg}).

From far away all typical random dimer configuration of a large Aztec
look almost the same. This statement can be made rigorous by looking at the
\emph{height function}~\cite{Thurston:height}, and saying that this random rescaled
height function converges in probability to a continuous deterministic function,
\emph{the limit shape}, which maximizes a certain functional.

This has been established for the uniform measure on dimer configurations of
simply connected subgraphs of the square lattice~\cite{CKP}, generalized to more
general periodic weights~\cite{KOS,Kuchumov2}.

In the original setup of the uniform measure for the Aztec diamond, Jokush Propp
and Shor~\cite{JPS98} proved that the behaviour of the limit shape varies
depending on the position in the domain:
\begin{itemize}
  \item Outside of the inscribed circle (the \emph{frozen regions}): the limit shape is linear, and the
    corresponding dominos display a brickwall pattern, which has different
    orientations in each corner.
  \item inside the inscribed circle (the \emph{liquid region}): the slope of the height function varies
    continuously, and all type of edges have a positive probability of
    appearance.
\end{itemize}
The inscribed circle separates \emph{frozen regions} from a \emph{temperate} (or
\emph{liquid}, using the terminology of~\cite{KOS}). It is referred to as
the~\emph{arctic circle}, and the main result of~\cite{JPS98} is known as the
\emph{arctic circle theorem}.

This has been extended to periodic weights for which the liquid region is an
ellipse~\cite{Johansson}, where the fundamental domain still contains a single
pair of white and black vertices. Then it has been obtained for the 2-periodic
case~\cite{ChhitaJohansson} and biased variant~\cite{BorodinDuits}, where
the underlying spectral curve has genus 1, and more recently by Berggren and
Borodin~\cite{BerggrenBorodin} for a generic arbitrary genus spectral curve, but with additional
assumptions which seem purely technical.

We claim that Formula~\eqref{eq:Kinv} allows us to recover and go beyond the
results cited above:
having the inverse Kasteleyn operator as an explicit contour integral is
particularly well-suited for asymptotic analysis: it is then possible to extract
from standard saddle point analysis the limiting behaviour of local
probabilities, and from them reconstruct the rescaled expected height function
which would give the limit shape.

We first give in Section~\ref{sec:arctic_ellipse} a short derivation of the
arctic circle (ellipse) theorem, by the saddle point method, as presented
in~\cite{ryg} for the uniform measure,
inspired by its use in~\cite{OR1} for plane partitions. Then, in Section~\ref{sec:crit_kl}
and \ref{sec:limit_shape_g1}, we explore some limit shape results
for Fock's weights given by genus 0 and 1 M-curves. Finally we discuss in
Section~\ref{sec:limit_shape_g} how the
geometric arguments of~\cite{BerggrenBorodin} can be adapted to this more
general context to give a less restrictive result, see also the forthecoming paper~\cite{BobenkoBobenkoSuris}. 

\subsection{A short derivation of the arctic circle
theorem}\label{sec:arctic_ellipse}

A short derivation of the arctic circle theorem for uniform weights is presented
in~\cite{ryg}. We give briefly here a variant
for Kenyon's \emph{critical} (genus 0) weights ,
where all the parameters
$\alpha_j$ (resp. $\beta_j$, $\gamma_j$, $\delta_j$) are equal to a single value
$\alpha$ (resp. $\beta$, $\gamma$, $\delta$) satisfying the following cyclic
order
\begin{equation*}
  \alpha < \gamma < \beta < \delta .
\end{equation*}

By~\cite{CKP}, we know that the rescaled height function converges in
probability, for the uniform topology, to a deterministic continuous function,
with a gradient contained in some polygon.
Our goal is to determine here the interface between the region where the
gradient of the limiting height function is extremal and where it is not.

The expected gradient of the height function is directly related to the
probability that an edge of given type/orientation at that position
is a dimer.
Let us look at a particular edge $\es=(\ws,\bs)$ crossed by train-tracks with
parameters $\delta$ and $\alpha$, where $\ws$ has coordinates
$(\ws_x,\ws_y)=(2i,2j+1)$ and
$\bs$ has coordinates $(\bs_x,\bs_y)=(2i-1,2j)$.
By~Equation~\eqref{eq:Kinv_homog_sw}, the probability to see this
edge in a random dimer configuration is given by:
\begin{equation*}
  \Ks_{\ws,\bs}\Ks^{-1}_{\bs,\ws} = (\alpha-\delta)
  \frac{1}{(2i\pi)^2}\int_{\C_2}\int_{\C_1}
  \frac{g_{\bs,0}(u) g_{0,\ws}(v)}{v-u}
  \left(\frac{(u-\beta)(v-\delta)}{(u-\delta)(v-\beta)}\right)^n du\, dv
\end{equation*}
where
\begin{align*}
  g_{0,\bs}(u)&= \frac{(u-\delta)^{i-1} (u-\alpha)^j}{(u-\gamma)^i (u-\beta)^j},
              &
  g_{0,\ws}(v)&= \frac{(v-\gamma)^i (v-\beta)^j}{(v-\delta)^i (v-\alpha)^{j+1}}.
\end{align*}

We are interested in the behaviour of the integrand when $n$ is large, and
$(i/n, j/n)\to (x,y)$, which can be rewritten in this asymptotic regime as:
\begin{equation}
  \exp(n(F(u;x,y)-F(v;x,y)+o(1)))
  \label{eq:integrand_proba_n}
\end{equation}
where
\begin{equation*}
  F(u;x,y) = y \log(u-\alpha) + (1-y)\log(u-\beta) - x \log(u-\gamma) -
  (1-x)\log(u-\delta).
\end{equation*}

The critical points of $F$ are given by a quadratic equation in $u$: there are
therefore two of them (counted with multiplicity).
In order to apply the saddle point, one needs to move continuously the contour
to make them pass through the critical points in the direction for which the
critical point will be a maximum (resp.\ a minimum) for the real part of $F$ for
the variable $u$ (resp. for $v$). In that configuration, the double integral
tends to 0 exponentially fast as $N$ goes to infinity.
See for example~\cite{OR1} where this technique has been introduced in the
context of tilings.
When the two critical points are not on the real line, one needs to make one of
the contour cross the other, at least partially, to make the two contours pass
through the two critical points (and cross orthogonally). By doing so, we get an
additional contribution, given by the integral along the path between the two
critical point of the residue of the integrand when $u=v$, which becomes the
main contribution (the remaining double integral goes to zero by the saddle
point method). See Figure~\ref{fig:saddle-point-contour}.

\begin{figure}
\centering
\hfill
\includegraphics[height=45mm]{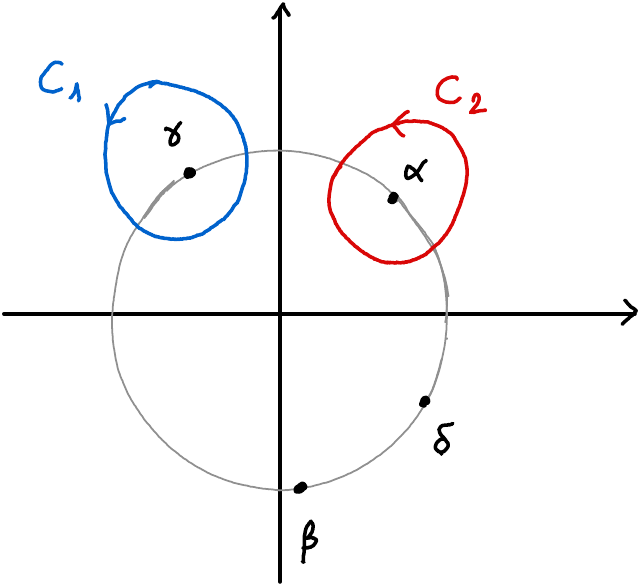}
\hfill
\includegraphics[height=5cm]{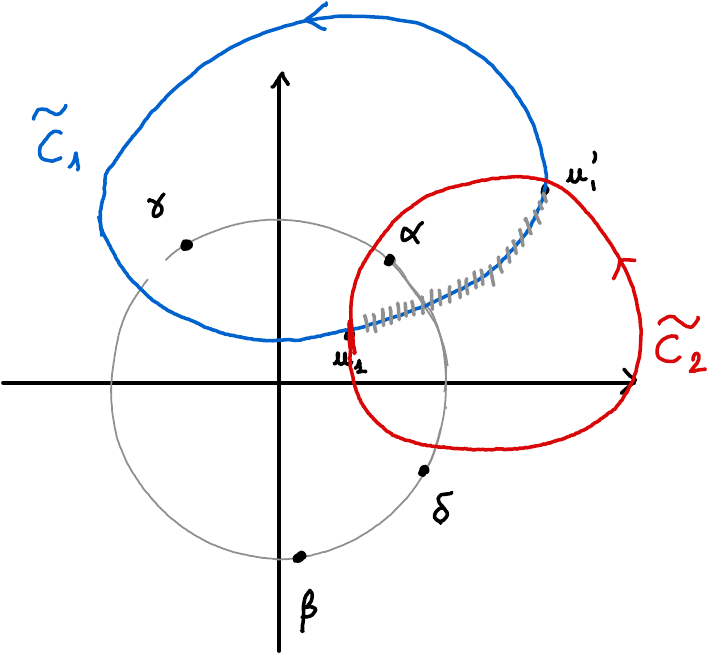}
\hfill\mbox{}
\caption{Left: representation of the two contours of integration when $\Sigma$
  is the Riemann sphere for the inverse Kasteleyn matrix of
  Section~\ref{sec:arctic_ellipse}. The curve $A_0$ is the unit circle in this
  case.
  Right: deformation of the contours in the case when the two critical points of
  $F$ are not on the unit circle (they are thus on different sides), to be in
  position to apply the saddle point method.
   Deforming the contours like this creates unwanted
  contribution to the integral of the residue at $u=v$ along the hatched segment
  joining the two critical points of $F$, which has to be subtracted.
}
\label{fig:saddle-point-contour}
\end{figure}

The
corresponding probability of the corresponding edge converges to a number
strictly between 0 and 1: we are in the so-called \emph{liquid region}, or
\emph{rough phase}.
On the contrary, when the two critical points are both real, the integral
becomes trivial, and the corresponding probability is either 0 or 1: we are then
in a \emph{frozen region}.

The \emph{arctic curve} describing the transition from the liquid region the a
frozen phase is thus given by the set of coordinates $(x,y)$ for which the two
critical points merge on the real axis.

In our case, these two critical points merge when $x$ and $y$ satisfy the
following equation:
\begin{equation*}
r^2 (x-\frac{1}{2})^2 +r^2 (y-\frac{1}{2})^2 
+ 2r(2-r)
(x-\frac{1}{2})(y-\frac{1}{2})  = r-1
\end{equation*}
where
\begin{equation*}
  r=\frac{(\beta-\alpha)(\delta-\gamma)}{(\beta-\gamma)(\delta-\alpha)}
  =(\alpha,\gamma;\beta,\delta) \in (1,+\infty)
\end{equation*}
is the cross ratio of the four points $\alpha, \gamma,\beta,\delta$.

We know by~\cite{KO:Harnack} that two sets of isoradial critical
weights correspond are gauge equivalent if the corresponding train-track
parameters are related by a M\"obius transformation preserving the unit circle.
The cross-ratio $r$ being M\"obius invariant, it is expected that the limit shape
if a function of this parameter.

When $r=2$
(for which a representative is when $\alpha$, $\gamma$, $\beta$, $\delta$
cut the unit circle in four equal arcs), the corresponding Boltzmann measure is
uniform, and the arctic curve is a circle. A cross-ration $r\neq 2$ can be
obtained by putting different weight on NE/SW and NW/SE dominos.
Choosing the parameters $\alpha,\gamma,\beta,\delta$ to be $\pm e^{\pm
i\theta/2}$, the ratio between the weights is given by $\tan\frac{\theta}{2}$ and
$r=(\cos\frac{\theta}{2})^{-2}$. In this case, the arctic curve is an ellipse,
as proved by Johansson~\cite[Theorem~2.4]{Johansson}.

\begin{figure}
  \begin{center}
    \includegraphics[width=6cm]{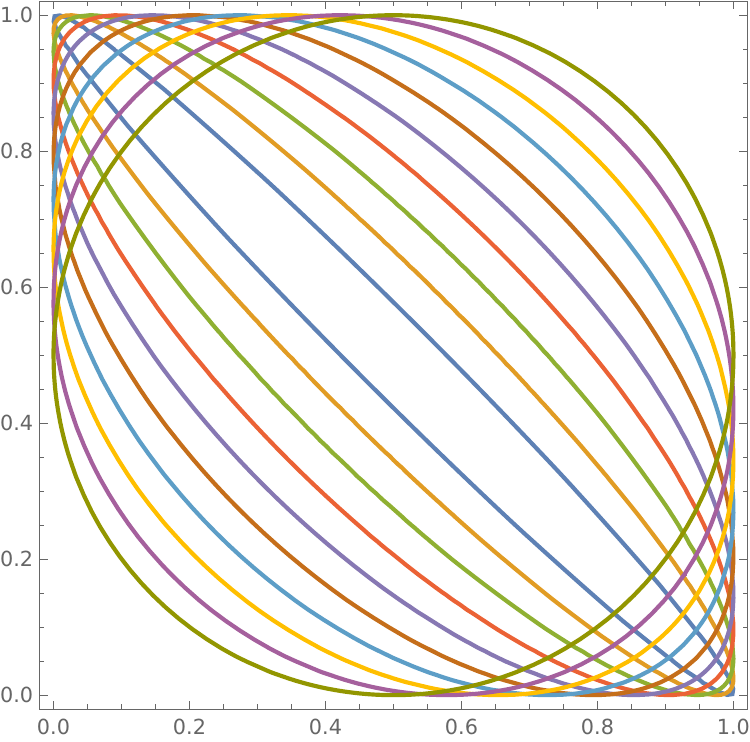}
    \caption{Arctic regions for (1-)periodic critical weights, for various values
      of the cross-ratio $r$ between $\alpha$, $\beta$, $\gamma$ and $\delta$.
      Here $r=(\cos\frac{\theta}{2})^{-2}$ and $\theta=k\frac{\pi}{20}$ for
      $k=1,\ldots, 10$. The green circle corresponding to the classical arctic circle
      for the uniform measure, for $k=10$.
    }\label{fig:ellipses_aztec}
  \end{center}
\end{figure}

\subsection{The (periodic) critical Aztec diamond has no gaseous
phase}\label{sec:crit_kl}

We assume now that the parameters $\alphab, \betab$ (resp.\
$\gammab, \deltab$) are periodic with period $l$ (resp. $k$),
for some $k,l \geq 1$.

The integrand in the formula describing the probability of a single edge has the
same form, as Equation~\eqref{eq:integrand_proba_n}, where now the function
$F(u;x, y)$ is given by the formula:
\begin{multline*}
  F(u;x,y) = \frac{1}{l} \sum_{j=1}^l \left[%
    y\log(u-\alpha_j)+(1-y)\log(u-\beta_j)
  \right]\\
  -\frac{1}{k} \sum_{j=1}^k \left[%
    x\log(u-\gamma_j)+(1-x)\log(u-\delta_j)
  \right].
\end{multline*}

The equation for critical points of $F(\cdot;x,y)$ can be rewritten as:
\begin{equation*}
  \frac{1}{l} \sum_{j=1}^l \left[%
    \frac{y}{u-\alpha_j}+\frac{1-y}{u-\beta_j}
  \right]
  =\frac{1}{k} \sum_{j=1}^k \left[%
    \frac{x}{u-\gamma_j}+\frac{1-x}{u-\delta_j}
  \right],
\end{equation*}
which by putting everything above the same denominator, becomes a polynomial
equation in $u$ of degree $2k+2l-2$. It has thus $2k+2l-2$ complex roots,
counted with multiplicity.

On the other hand, writing $u=e^{2i\bar{u}},\alpha_j =
e^{2i\bar{\alpha}_j},\ldots$, and using the fact that
\begin{equation*}
  \frac{2i u}{u-\alpha} =
  \frac{e^{i(\bar{u}-\bar{\alpha})}}{\sin(\bar{u}-\bar{\alpha})} = 
    i + \cot(\bar{u}-\bar{\alpha}),
\end{equation*}
the same critical equation has the following form:
\begin{equation*}
\frac{1}{l}\sum_{j=1}^l \left[
  y\cot(\bar{u}-\bar{\alpha}_j) +(1-y)\cot(\bar{u}-\bar{\beta}_j)
\right]
=
  -\frac{1}{k} \sum_{j=1}^k \left[%
    x\cot(\bar{u}-\bar{\gamma}_j)+(1-x)\cot(\bar{u}-\bar{\delta}_j)
  \right]
\end{equation*}
when $\bar{u}$ is real, both sides are real.
The left-hand side takes all real values between $-\infty$ and $+\infty$ when
$\bar{u}$ is in $(\bar{\alpha}_j,\bar{\alpha}_{j+1})$ or
$(\bar{\beta}_j,\bar{\beta}_{j+1})$, for some $1\leq j\leq l-1$, whereas the
right-hand side stays bounded on those intervals. As a consequence, there is at
least a solution of the critical equation on each of these intervals.
Reasoning in the same way on the intervals
$(\bar{\gamma}_j,\bar{\gamma}_{j+1})$, $(\bar{\delta}_j,\bar{\delta}_{j+1})$,
$1\leq j\leq k-1$, exchanging the roles of the left- and right-hand sides,
this gives in total at least $2k+2l-4$ real roots.
Therefore, there are at most two complex, non-real roots, for the equation in
$\bar{u}$, which have to be complex conjugated. This means that for the
original equation in $u=e^{2i\bar{u}}$, the two extra solutions are not on the unit
circle and the reciprocal of the complex conjugate of one is equal to the other.

Repeating the same saddle-point analysis as above, one can see that the
probability of the considered edge, will have a limit, which will be non trivial
in $(0,1)$ if and only the two extra critical points are not on the unit circle.
When this is the case, let us
call $u_0$ the one inside the unit disc. Then one sees that we can find a unique
point $(x,y)$ which has $u_0$ (and the inverse of its complex conjugate) as critical points for
$F$, by solving the system of two real linear equations obtained by separating
the real and imaginary part of
\begin{equation*}
    \frac{\partial}{\partial u}F(u_0;x,y) = 0.
\end{equation*}
whereas, the edge will be frozen if all the critical points are real.
The argument that this linear system has a unique solution, which is furthermore
in the square $[0,1]$ is an slight adaptation of the proof
of~\cite[Theorem~4.9]{BerggrenBorodin}.

The transition between the two regimes occur when the two extra complex critical
points merge at some $u_0=e^{2i\bar{u_0}}$ on the unit circle, where they become zeros of the first, and second
derivative of $F(\cdot; x, y)$. The \emph{arctic curve} separating the
two regimes is obtained finding $x$ and $y$ such that
\begin{equation}
  \begin{cases}
    \frac{\partial}{\partial u}F(u_0;x,y) = 0\\
    \frac{\partial^2}{\partial u^2}F(u_0;x,y) = 0
  \end{cases}
\end{equation}
for any point $u_0$ on the unit circle $S^1$. Since $F$ is an affine function of $x$ and $y$, this amounts
to solving again a linear system in $x$ and $y$ with
coefficients which are explicit functions of $u_0$, yielding an explicit
parametrization of the arctic curve 
$u_0\in S^1\mapsto (x(u_0), y(u_0))$.

The result above can be summarized in the following proposition:

\begin{prop}
  For periodic ``genus 0'' weights, with arbitrary fundamental domain,
  \begin{itemize}
    \item The map from the unit disc to the liquid region is a
      homeomorphism.
    \item the arctic curve has
      a single component and has an explicit parametrization by trigonometric (or
      rational) functions. It is thus a real algebraic curve of genus 0.
      In particular there is no gaseous phase.
    \item The tangency points with the sides of the of the square $[0,1]$
      at the bottom (resp. top, left, right) to $u_0$ taking one of the values
      of $\alphab$ (resp. $\betab$, $\gammab$, $\deltab$).
  \end{itemize}
\end{prop}
The genus 0 weights are not generic, and is not covered by the results of
Berggren and Borodin~\cite{BerggrenBorodin}.
The statement about the tangency points has been noticed by Dan Betea for Schur
measures corresponding to the Aztec diamond, when parameters are periodic (which
is essentially Stanley's weights in the periodic setting)~\cite{DanBetea}.

The discussion above is mainly about the arctic curve separating the liquid and
the frozen region. Actually,
the saddle point analysis allows to find the frequency of each type of domino
near any point of the liquid region, from which we can reconstruct the expected
slope, and then the limit shape.

\begin{figure}
  \centering
  \includegraphics[width=4cm]{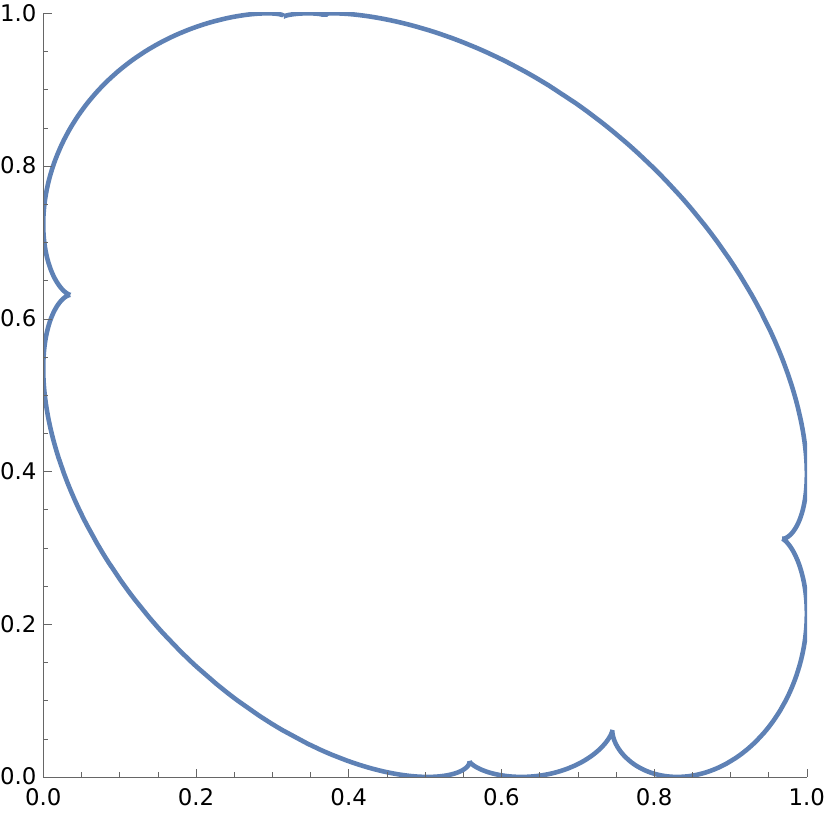}
  \caption{The arctic curve for critical periodic Aztec diamond (with period 3
  vertically, and 2 horizontally).}
\end{figure}

\subsection{The elliptic case}\label{sec:limit_shape_g1}

We briefly study here the case of elliptic weights, with periodic train-track
parameters.
The surface $\Sigma$ is now a torus
$\mathbb{C}/(\mathbb{Z}+\tau\mathbb{Z})$ where $\tau\in i\mathbb{R}^*_+$.

First, as in Section~\ref{sec:arctic_ellipse}, let us consider the case where the parameters $\alpha_j$
(resp.\ $\beta_j$, $\gamma_j$, $\delta_j$) are equal to the same value $\alpha$
(resp.\ $\beta$, $\gamma$, $\delta$) belonging to $\mathbb{R}/\mathbb{Z}$.

Rewriting the probability of a single edge using~\eqref{eq:Kinv_homog}, with the
same notation as for the critical case, we get:
\begin{equation*}
  \K_{\ws,\bs} \K^{-1}_{\bs,\ws}
= \int_{\C_1}\int_{\C_2} \exp(n(F(u;x,y)-F(v;x,y))+o(1))
\frac{du\, dv}{\theta_1(\pi(v-u))}
\end{equation*}
where now
\begin{multline*}
  F(u;x,y) =
  y \log\theta_1(\pi(u-\alpha)) + (1-y)\log\theta_1(\pi(u-\beta))\\
 -x \log\theta_1(\pi(u-\gamma)) - (1-x)\log\theta_1(\pi(u-\delta)).
\end{multline*}

The differential of $F$ is a meromorphic 1-form on the torus $\Sigma$. It has a
divisor of degree $2g-2=0$. Since $F$ is periodic in the horizontal direction of
the torus, the integral of $dF$ along the cycle
$\frac{\tau}{2}+\mathbb{R}/\mathbb{Z}$
is equal to 0. Moreover, since $dF$ is real on this cycle, the intermediate
value theorem implies that $dF$ has at least two zeros on this cycle.
Since it has four simple poles, located at $\alpha$, $\beta$, $\gamma$ and
$\delta$, it must have two additional zeros.

These extra zeros can be either both on the same real component of the torus (in
a frozen or gaseous phase),
or both non-real and symmetric, and complex conjugated (in the liquid phase).

The boundary of the liquid phase is thus obtained by looking at the transition
between the two regimes, when the two extra critical points merge into a double
critical point.
As in the critical case,
asking for $u$ to be a double critical point of $F$ yields a linear
system in $x$ and $y$ with coefficients that are given in terms of elliptic
functions of $u$. When $u$ runs along the two real components of the torus, we
get two closed curves: the outer curve separates the liquid phase from the
frozen ones in the corners; the inner one separates the liquid phase from the
gas bubble near the center.

\begin{figure}
  \begin{center}
    \includegraphics[width=4cm]{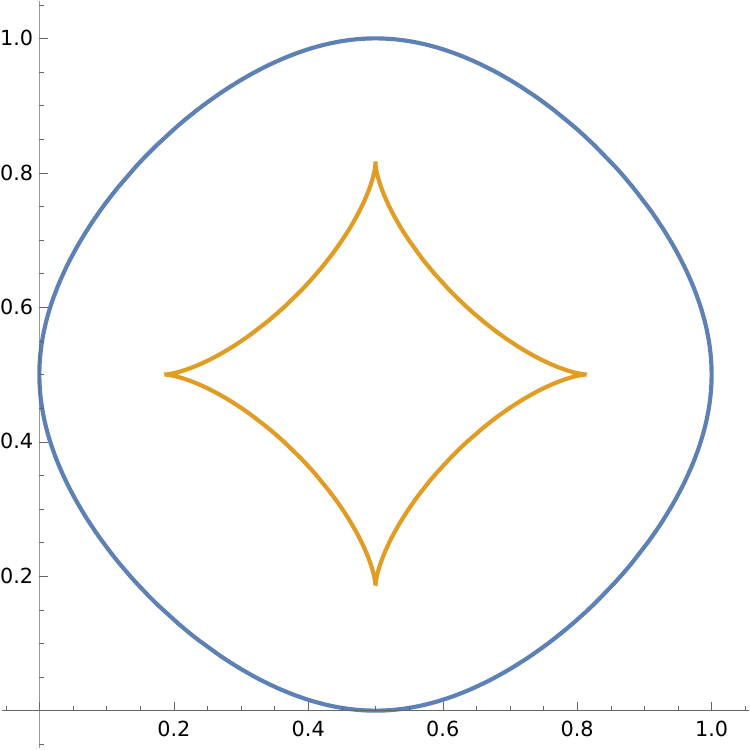}
    \includegraphics[width=4cm]{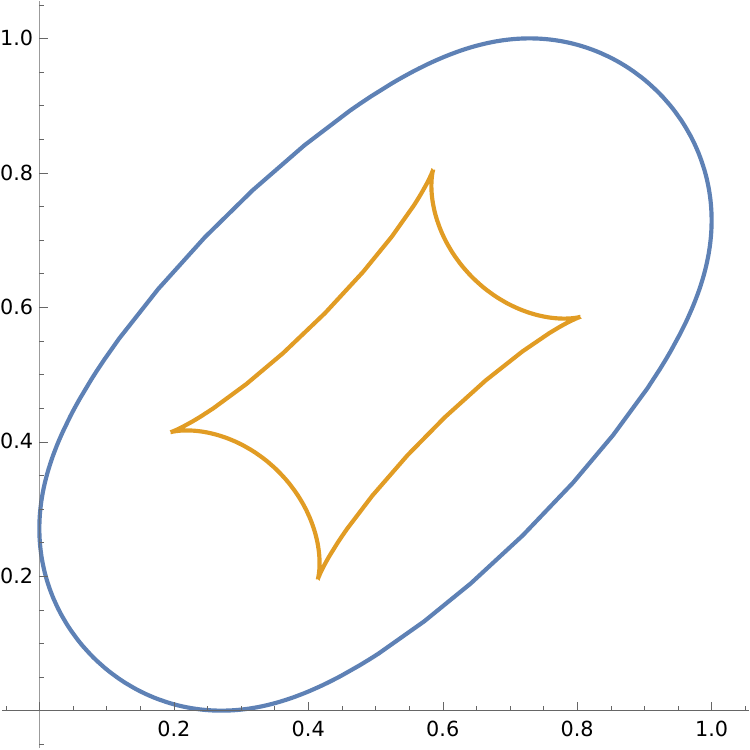}
    \includegraphics[width=4cm]{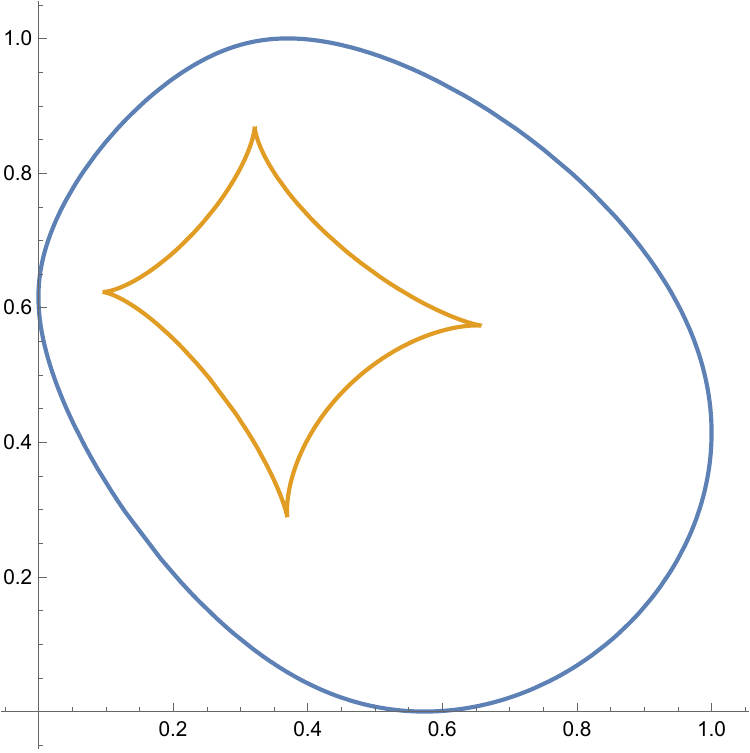}

    \caption{Three arctic curves for a torus with $\tau=i\sqrt{3}$. Left: the
      limit shape for the unbiased $2\times 2$-periodic Aztec
      diamond ($\alpha,\beta,\gamma,\delta$ cut the circle into four equal
      arcs). Center: the limit shape for the biased
      weights with $\rho=\frac{1}{6}$
      ($\alpha=0,\gamma=\frac{1}{6},\beta=\frac{1}{2},\delta=\frac{5}{6}$)
      and the same value of $\tau$.
      Right:
      $\alpha=0,\beta=\frac{3}{5},\gamma=\frac{3}{10},\delta=\frac{17}{20}$
    }
    \label{fig:aztec_2x2}
  \end{center}
\end{figure}

Note that having periodic (even constant here) train-track parameters does not
imply that the edge weights are periodic~\cite[Section~4]{BCdT:genusg}.
In order for the edge weights to be periodic, for a fundamental domain
containing two white and two black vertices, the parameters
$\alpha,\beta,\gamma,\delta$ should satisfy: 
\begin{equation}
  \alpha-\beta+\gamma-\delta = 0 \text{ mod } 1,\qquad
  \alpha-\beta-\gamma+\delta = 0 \text{ mod } 1.
  \label{eq:2x2perio}
\end{equation}
which means that $\beta-\alpha=\pm\frac{1}{2}$, and
$\delta-\gamma=\pm\frac{1}{2}$.
This case contains in particular the $2\times 2$ periodic Aztec diamond, studied
by
\begin{itemize}
  \item Chhita and Johansson~\cite{ChhitaJohansson} by taking
    $\alpha=0,\beta=\frac{1}{2},\gamma=\frac{1}{4},\delta=\frac{3}{4}$
    and letting $\tau$ vary on the imaginary axis,
  \item Borodin and Duits~\cite{BorodinDuits}, for a specific value of $t$,
    which in fact does not play a role in the computation of the limit shape.
\end{itemize}

These are the cases depicted on the left and middle of
Figure~\ref{fig:aztec_2x2}.
An example where the periodicity condition is \emph{not} satisfied is shown on
the right of that same figure. Note that in this case, the picture does not have
axial symmetries anymore.

One can extend this study when the parameters $\alphab$, $\betab$, $\gammab$,
and $\deltab$ are $k\times l$ period, like in Section~\ref{sec:crit_kl}.
With the additional constraint that the corresponding edge weights are periodic,
this covers in particular the
case $k=2, l=2$ in the study of Berggren and
Borodin~\cite{BerggrenBorodin} (without the technical assumption they have on
the weights to guarantee that the vertical tentacles of the amoeba of the corresponding
spectral curve are separated), but also the extension to the case where the
periodicity condition for edge weights is not met, like the one depicted on
Figure~\ref{fig:aztec_4x4}. One can relate here, like in the critical case, the
number of cusps along bottom, left, top, right side of the square with the
number of distinct values for
$\alphab$, $\gammab$, $\betab$, $\deltab$ respectively, whereas there are four
of them along the liquid/gas interface.

\begin{figure}
  \centering
  \hfill
  \includegraphics[width=6cm]{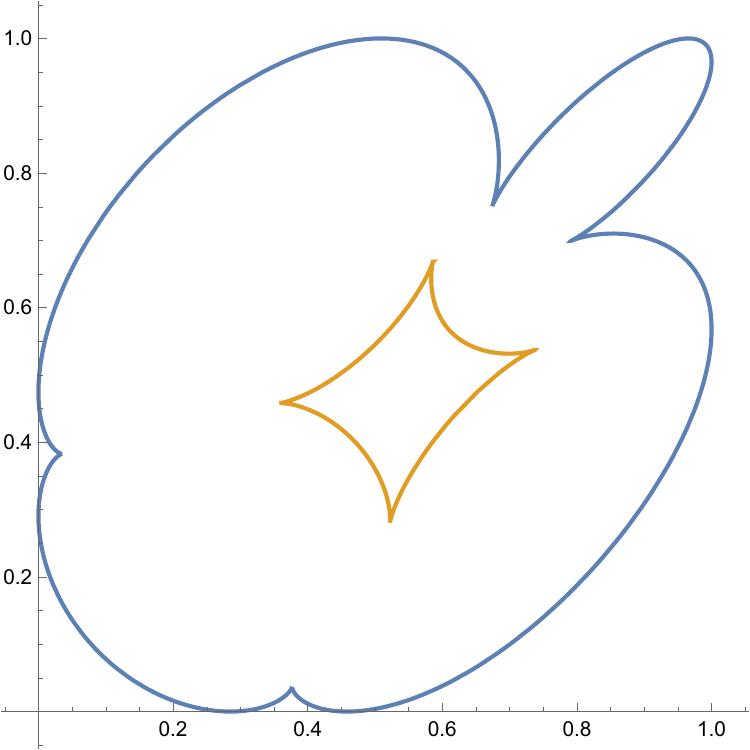}\hfill
  \includegraphics[width=6cm]{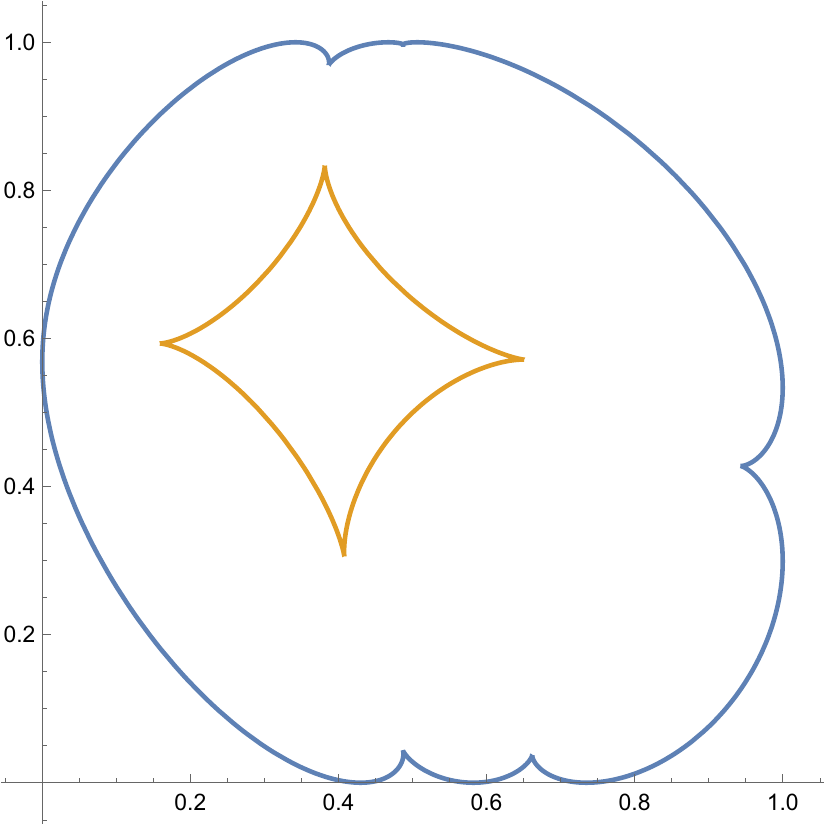}\hfill\mbox{}
  \caption{Left: a 2x2 example with $\alphab=\{0.744904, 0.448131\}$,
           $\betab=\{2.35367, 1.58096\}$,
           $\gammab=\{1.21721, 0.983251\}$,
           $\deltab=\{3.05117, 2.38214\}$. Right: a 3x2 example where
           $\alphab=\{0.599648, 0.0928766, 0.288755\}$,
           $\betab=\{2.29566, 1.7928, 2.20249\}$,
           $\gammab=\{1.18328\}$,
           $\deltab=\{3.06226, 2.74699\}$. The fact that there is a single
           contact point on the left boundary (whereas there are two on the
           right side) is a consequence of the fact that $\gammab$ is 1-periodic
           (constant) whereas $\deltab$ is 2-periodic.
         }
  \label{fig:aztec_4x4}
\end{figure}

\subsection{Higher genus}\label{sec:limit_shape_g}

The study above can be done on surfaces with higher genus,  with train-track
parameters assigned in a periodic way, even though it can
be technically challenging to work with prime forms to control the saddle point
method. Imposing periodicity for weights from periodicity of train-track
parameters implies that the fundamental domain is large enough (one needs to
pick $g$ distinct integer points in the interior points of the corresponding
Newton polygon), but we can relax this condition (and get just quasi-periodic
weights), as in the genus 1 case above.

In~\cite{BerggrenBorodin}, the authors give a similar looking formula for the
inverse of the Kasteleyn matrix of the Aztec diamond with periodic weights. The
main differences are the fact that the integrals are written with respect to the
coordinates $z$ (and $w$) of the spectral curve instead of an intrisic
coordinate of the Riemann surface $\Sigma$, and the fact that the integrand is a
block matrix%
\footnote{In particular compare the term $E(u,v)$ in Equation~\eqref{eq:Kinv} and
the term $z_1-z_2$ in their formula.
}.
Otherwise, unsurprisingly, the function $F$ appearing in the asymptotics of the
integral is the same as here. We refer to the paper~\cite{BerggrenBorodin} for a
detailed discussion of the saddle point method in the periodic case under
additional technical assumptions. Working with formula~\eqref{eq:Kinv} can be
seen as a way to circumvent the technical assumptions they have to extend their
results to non-generic cases, where the spectral curve is singular or the amoeba
has a degenerate behaviour. In that case, the geometric correspondence is to be
made between the limit shape of the Aztec diamond and the ``upper part''
$\Sigma^+$ of $\Sigma$, which is the connected component of
$\Sigma\setminus\bigcup_{i=0}^g A_i$ whose oriented boundary is $A_0-\sum_{i=1}^g A_i$
(instead of the amoeba of the spectral curve).
Most of their geometric arguments can be carried or adapted in this slightly
changed context to obtain the following:

\begin{prop}[\cite{BerggrenBorodin}]
\leavevmode
\begin{itemize}
  \item The map $\psi$ from $(x,y)$ in the liquid region to the critical point of $F$
    in the interior of $\Sigma^+$ is a homeomorphism.
  \item There is a bijective correspondence between gaseous bubbles and the real
    ovals $A_1,\cdots A_g$.
  \item There is a bijective correspondence between frozen regions and connected
    components of
    $A_0\setminus\{\alpha_j,\beta_j,\gamma_j,\delta_j\ ;\ 1\leq j \leq n\}$.
\end{itemize}
\end{prop}

The saddle point method allows one to not only compute the probability that a
given edge occurs, but also to derive the asymptotics for the entry of inverse
Kasteleyn matrix for vertices $\bs$ and $\ws$ at finite distance, near the same
macroscopic
point $(x,y)$ in the liquid region: it converges to $A^{u_0}_{\bs,\ws}$,
where $u_0=\psi(x,y)$ is the critical point of $F(\cdot;x,y)$ in $\Sigma^+$, and
$A^{u_0}$ is the inverse from the family of inverses introduced
in~\cite{BCdT:genusg} for the Fock Kasteleyn operator on any
infinite minimal graph containing as a subgraph a neighborhood of the Aztec
diamond containing $\ws$ and $\bs$ (for example the infinite square lattice).
The inverse $A^{u_0}$ on the infinite square defines a Gibbs measure
$\mathbb{P}^{u_0}$ on dimer configurations by the usual determinantal formula:
\begin{equation*}
  \mathbb{P}^{u_0}((\ws_1,\bs_1),\dots, (\ws_k,\bs_k))
  = \Bigl(\prod_{j=1}^k \Ks_{\ws_j,\bs_j}\Bigr)\det_{1\leq i,j\leq k}
  \left(A^{u_0}_{\bs_i,\ws_j}\right)
\end{equation*}
Therefore, one can obtain in this generalized setup the analogue
of~\cite[Theorem~1.10]{BerggrenBorodin}:
\begin{prop}
  The local statistics of dimers around a point $(x,y)$ of the liquid region
  converge to the infinite Gibbs measure $\mathbb{P}^{u_0}$ on the square
  lattice, with $u_0=\psi(x,y)$.
\end{prop}

We also note that the imaginary part of the equation used to determine the critical points of $F$, is the
same as the master equation for the \emph{tangent plane method} by Kenyon and
Prause~\cite{KenyonPrause:gradient, KenyonPrause:g05V, KenyonPrause:limitshapes}: indeed,
for a fixed $u$, the action $F$ is an affine
function of $x$ and $y$, and the coefficients of $x$ and $y$
(respectively $\log z(u)$ and $\log w(u)$) which are analytic functions of $u$,
and whose imaginary parts give the slope of the limit shape at the point
$(x,y)$, see~\cite[Section~4.5]{BCdT:genusg}.
The remaining constant part is encoding
the boundary conditions and should be directly related to the \emph{intercept
function}. We have thus two interpretations of the same equation: one coming
from finding critical points of the action functional $F$, the other other one by
looking at \emph{isothermal coordinates} in which the slope and the intercept of
the asymptotic height function describing the limit shape are harmonic.
Further investigation is needed to make a full connection between the two
approaches. A starting point is the expression of the expression of the surface
tension in terms of geometric quantities on $\Sigma$, obtained
in~\cite[Section~4.6]{BCdT:genusg},
see also~\cite{BobenkoBobenkoSuris}.

\bibliographystyle{alpha}
\bibliography{aztec}

\end{document}